\pdfoutput=1
\RequirePackage{ifpdf}
\ifpdf % We~are running pdfTeX in pdf mode
\documentclass[pdftex]{sigma}
\else
\documentclass{sigma}
\fi

\usepackage{bm}
\usepackage{pifont}
\usepackage{array}
\usepackage{oldgerm}
\usepackage{stmaryrd}
\usepackage{mathrsfs}
\usepackage{pict2e}

\definecolor{mycolor}{rgb}{0.11,0.6,0.41}
\definecolor{darkred}{rgb}{0.4,0,0}
\definecolor{darkgreen}{rgb}{0,0.35,0}
\definecolor{melegzold}{rgb}{0.2,0.5,0.2}
\definecolor{mzold}{rgb}{0.24,0.6,0.24}
\definecolor{darkblue}{rgb}{0,0,0.35}
\definecolor{darkyellow}{rgb}{0.8,0.8,0}
\definecolor{sand}{rgb}{1,1,0.8}
\definecolor{palegreen}{rgb}{0.9,1,0.9}
\definecolor{palered}{rgb}{1,0.9,0.9}
\definecolor{grey}{rgb}{0.5,0.5,0.5}
\definecolor{cian}{rgb}{0,0.6,0.6}
\definecolor{yellow}{rgb}{0.6,0.6,0}
\definecolor{magenta}{rgb}{0.6,0,0.6}
\definecolor{orange}{rgb}{1,0.55,0}

\definecolor{thm-fcolor}{rgb}{0.8,0,0}
\definecolor{thm-bcolor}{rgb}{1,0.9,0.9}

\definecolor{def-fcolor}{rgb}{0,0.35,0}
\definecolor{def-bcolor}{rgb}{0.9,1,0.9}

\definecolor{rmk-fcolor}{rgb}{0,0,0.35}
\definecolor{rmk-bcolor}{rgb}{0.9,0.9,1}

\definecolor{xca-fcolor}{rgb}{0.8,0.8,0}
\definecolor{xca-bcolor}{rgb}{1,1,0.8}

\numberwithin{equation}{section}

\newtheorem{thm}{Theorem}[section]
\newtheorem{pro}[thm]{Proposition}
\newtheorem{lem}[thm]{Lemma}
\newtheorem{cor}[thm]{Corollary}

\theoremstyle{definition}
\newtheorem{defi}[thm]{Definition}
\newtheorem{sch}[thm]{Scholium}
\newtheorem{exa}[thm]{Example}

\newtheorem{rmk}[thm]{Remark}

\newcommand{\ncm}{\newcommand}

\ncm{\axiomenum}[1]{

\begin{enumerate}\itemsep=0pt}
\ncm{\axiomenumend}{
\end{enumerate}

}

\ncm{\Cat}{\mathsf{Cat}}
\ncm{\Path}{\mathsf{Path}}
\ncm{\Ribb}{\mathsf{Ribb}}
\ncm{\CAT}{\mathsf{CAT}}
\ncm{\ACT}{\mathsf{ACT}}
\ncm{\ob}{\operatorname{ob}}
\ncm{\Nat}{\operatorname{Nat}}
\ncm{\Set}{\mathsf{Set}}
\ncm{\Ab}{\mathsf{Ab}}
\ncm{\Add}{\mathsf{Add}}
\ncm{\Mnd}{\mathsf{Mnd}}
\ncm{\Cmd}{\mathsf{Cmd}}
\ncm{\Fun}{\mathsf{Fun}}
\ncm{\Mon}{\mathsf{Mon}}
\ncm{\Elt}{\mathsf{Elt}\,}
\ncm{\Sub}{\mathsf{Sub}\,}
\ncm{\Flat}{\mathsf{Flat}}
\ncm{\add}{\text{\Large\textsl{a}}}
\ncm{\proj}{\mathsf{proj}}
\ncm{\Col}{\mathsf{Col}\,}
\ncm{\fgmod}[1]{\mathsf{mod}\text{-}#1}
\ncm{\UEE}{\mathsf{u\text{-}eff}}
\ncm{\EEE}{\mathsf{e\text{-}eff}}
\ncm{\HEE}{\mathsf{h\text{-}eff}}
\ncm{\EE}{\mathsf{eff}}
\ncm{\Split}{\mathsf{split}}
\ncm{\Cl}{\mathsf{Cl}}
\ncm{\ord}{\mathsf{ord}}
\ncm{\Mod}[1]{\mathbf{Mod}\text{-}#1}
\ncm{\Vect}{\mathsf{Vec}}
\ncm{\vect}{\mathsf{vec}}
\ncm{\Bimod}[2]{#1\text{-}\mathbf{Mod}\text{-}#2}
\ncm{\Comod}[1]{\mathbf{Comod}\text{-}#1}
\ncm{\Sh}{\mathbf{Sh}}
\ncm{\rMonCAT}{\textsf{r-MonCAT}}
\ncm{\Graph}{\mathbf{Graph}}
\ncm{\Rep}{\mathbf{Rep}}
\ncm{\Pol}{\mathsf{Pol}}
\ncm{\Hyp}{\mathsf{Hyp}}
\ncm{\HHyp}{\mathbf{Hyp}}

\ncm{\A}{\mathcal{A}}
\ncm{\B}{\mathcal{B}}
\ncm{\C}{\mathcal{C}}
\ncm{\D}{\mathcal{D}}
\ncm{\E}{\mathcal{E}}
\ncm{\F}{\mathcal{F}}
\ncm{\G}{\mathcal{G}}
\ncm{\Ha}{\mathcal{H}}
\ncm{\Hil}{\mathscr{H}}
\ncm{\I}{\mathcal{I}}
\ncm{\J}{\mathcal{J}}
\ncm{\K}{\mathcal{K}}
\ncm{\Ll}{\mathcal{L}}
\ncm{\M}{\mathcal{M}}
\ncm{\N}{\mathcal{N}}
\ncm{\Ou}{\mathcal{O}}
\ncm{\Pee}{\mathscr{P}}
\ncm{\R}{\mathcal{R}}
\ncm{\X}{\mathcal{X}}
\ncm{\W}{\mathcal{W}}
\ncm{\V}{\mathcal{V}}
\ncm{\U}{\mathcal{U}}
\ncm{\T}{\mathcal{T}}
\ncm{\Teven}{\mathcal{T}_\mathsf{e}}
\ncm{\Tcan}{\mathcal{T}_\mathsf{can}}
\ncm{\CV}{\operatorname{CV}}
\ncm{\one}{\mathbf{1}}

\ncm{\dom}{\operatorname{dom}}
\ncm{\cod}{\operatorname{cod}}
\ncm{\End}{\operatorname{End}}
\ncm{\Aut}{\operatorname{Aut}}
\ncm{\Hom}{\operatorname{Hom}}
\ncm{\kernel}{\operatorname{ker}}
\ncm{\Ker}{\operatorname{Ker}}
\ncm{\coker}{\operatorname{coker}}
\ncm{\Coker}{\operatorname{Coker}}
\ncm{\im}{\operatorname{im}}
\ncm{\Img}{\operatorname{Im}}
\ncm{\coim}{\operatorname{coim}}
\ncm{\id}{\operatorname{id}}
\ncm{\Center}{\operatorname{Center}}
\ncm{\colim}{\operatorname{colim}}
\ncm{\Colim}[1]{\underset{#1}{\operatorname{colim}}}
\ncm{\Lan}{\operatorname{Lan}}
\ncm{\Cone}{\operatorname{Cone}}
\ncm{\ev}{\operatorname{ev}}
\ncm{\coev}{\operatorname{coev}}
\ncm{\db}{\operatorname{db}}
\ncm{\cf}{\operatorname{cf}}
\ncm{\hgt}{\operatorname{ht}}
\ncm{\Irr}{\operatorname{Irr}}
\ncm{\Ind}{\operatorname{Ind}}
\ncm{\Bd}{\operatorname{Bd}}
\ncm{\Cb}{\operatorname{Cb}}
\ncm{\Nb}{\operatorname{Nb}}
\ncm{\Hol}{\operatorname{Hol}}
\ncm{\Ophol}{\operatorname{Hol}^\circ}
\ncm{\Cocom}{\operatorname{Cocom}}
\ncm{\nin}{\not\in}
\ncm{\Qv}{\operatorname{Qv}}

\ncm{\ci}{\,\circ\,}
\ncm{\smp}{\ast}
\ncm{\cv}{\ast}
\ncm{\del}{\partial}
\ncm{\smpq}{\smp_q}
\ncm{\bu}{\bullet}
\ncm{\bo}{\,\Box\,}
\ncm{\ot}{\otimes}
\ncm{\oV}{\odot}
\ncm{\oE}{\underset{\scriptscriptstyle E}{\odot}}
\ncm{\hot}{\,\bar{\ot}\,}
\ncm{\uot}{\,\Eob{\ot}\,} % {\,\underline{\ot}\,}
\ncm{\uoV}[1]{\underset{#1}\oV}
\ncm{\x}{\times}
\ncm{\ex}[1]{\underset{\scriptstyle #1}{\times}}
\ncm{\am}[1]{\underset{\scriptscriptstyle #1}{\ot}}
\ncm{\amo}[1]{\underset{\scriptstyle #1}{\ot}}
\ncm{\mash}{\Pisymbol{psy}{35}}
\ncm{\mashed}[1]{\underset{\scriptscriptstyle #1}{\Pisymbol{psy}{35}}}
\ncm{\cross}[1]{\underset{\scriptstyle #1}{\rtimes}}
\ncm{\rarr}[1]{\stackrel{#1}{\longrightarrow}}
\ncm{\isorarr}[1]{\underset{#1}{\overset{\sim}{\longrightarrow}}}
\ncm{\larr}[1]{\stackrel{#1}{\longleftarrow}}
\ncm{\mapsot}{\leftarrow\!\!\!\raisebox{1pt}{$\scriptscriptstyle |$}}
\ncm{\oR}{\am{R}}
\ncm{\oS}{\am{S}}
\ncm{\cop}{\Delta}
\ncm{\oneT}{^{(1)}}
\ncm{\twoT}{^{(2)}}
\ncm{\threeT}{^{(3)}}
\ncm{\nul}{_{(0)}}
\ncm{\p}{_{(1)}}
\ncm{\pp}{_{(2)}}
\ncm{\ppp}{_{(3)}}
\ncm{\pppp}{_{(4)}}
\ncm{\ppppp}{_{(5)}}
\ncm{\m}{_{(-1)}}
\ncm{\mm}{_{(-2)}}
\ncm{\mmm}{_{(-3)}}
\ncm{\mmmm}{_{(-4)}}
\ncm{\eps}{\varepsilon}
\DeclareBoldMathCommand\bra{\langle}
\DeclareBoldMathCommand\ket{\rangle}
\ncm{\dirim}[1]{{#1}_*}%{\underline{#1}}
\ncm{\invim}[1]{{#1}^*}%{\overline{#1}}
\ncm{\du}[1]{{\hat #1}}
\ncm{\tp}[1]{{#1}^{\mathsf{T}}}
\ncm{\duone}{\hat\one}
\ncm{\dueps}{\hat\eps}
\ncm{\duP}{\overset{\raisebox{-1pt}{\scriptsize $\sim$}}{\mathrm{P}}}
\ncm{\iduP}{\overset{\raisebox{-1pt}{\scriptsize $\backsim$}}{\mathrm{P}}}
\ncm{\xP}{\mathrm{P}^\times}
\ncm{\duduP}{\overset{\raisebox{-1pt}{\scriptsize $\sim$}}{\duP}}
\ncm{\duiduP}{\overset{\raisebox{-1pt}{\scriptsize $\sim$}}{\iduP}}
\ncm{\iduduP}{\overset{\raisebox{-1pt}{\scriptsize $\backsim$}}{\duP}}
\ncm{\iduiduP}{\overset{\raisebox{-1pt}{\scriptsize $\backsim$}}{\iduP}}
\ncm{\dual}[1]{\overset{\raisebox{-1pt}{\scriptsize $\sim$}}{#1}}
\ncm{\invdual}[1]{\overset{\raisebox{-1pt}{\scriptsize $\backsim$}}{#1}}
\ncm{\mir}{\mathbf{m}}
\ncm{\haar}{i} %{i_\text{H}}
\ncm{\duhaar}{\iota} %{\iota_\text{H}}
\ncm{\la}[1]{\underset{\scriptscriptstyle #1}{\triangleright}}
\ncm{\ra}[1]{\underset{\scriptscriptstyle #1}{\triangleleft}}
\ncm{\opla}[1]{\underset{\scriptscriptstyle #1}{\blacktriangleright}}
\ncm{\opra}[1]{\underset{\scriptscriptstyle #1}{\blacktriangleleft}}
\ncm{\Usep}{\underset{\text{U}}{\leftrightarrow}}
\ncm{\vac}{|\text{vac}\ket}
\ncm{\duvac}{\bra\text{vac}|}
\ncm{\Rbar}{\overline{R}}
\ncm{\sgn}{\operatorname{sgn}}
\ncm{\trans}{\mathsf{T}}

\ncm{\duT}{\overset{\sim}{\mathrm{T}}}
\ncm{\iduT}{\overset{\backsim}{\mathrm{T}}}
\ncm{\xT}{\mathrm{T}^\times}
\ncm{\dbSigma}{\mathcal{D}(\Sigma)}
\ncm{\dudbSigma}{\mathcal{D}(\Sigma)^*}
\ncm{\Arr}{\mathrm{Arr}}
\ncm{\FA}{\mathcal{F}}

\ncm{\asso}{a} %{\mathbf{a}}
\ncm{\luni}{l} % {\mathbf{l}}
\ncm{\runi}{r} %{\mathbf{r}}
\ncm{\iso}{\stackrel{\sim}{\rightarrow}}
\ncm{\iiso}{\rarr{\sim}}
\ncm{\ract}{\,\triangleleft\,}
\ncm{\lact}{\triangleright}
\ncm{\under}{\mbox{\,\rm\_}\,}
\ncm{\adj}{\dashv}
\ncm{\adjoint}{\dashv}
\ncm{\into}{\hookrightarrow}
\ncm{\onto}{\twoheadrightarrow}
\ncm{\sweedr}{\leftharpoonup}
\ncm{\sweedl}{\rightharpoonup}
\ncm{\Fin}{\mathbf{Fin}}

\ncm{\can}{\mathrm{can}}
\ncm{\fgp}{\mathrm{fgp}}
\ncm{\op}{\mathrm{op}}
\ncm{\coop}{\mathrm{coop}}
\ncm{\rev}{\mathrm{rev}}
\ncm{\inn}{\mathrm{in}}
\ncm{\out}{\mathrm{out}}
\ncm{\mini}{\mathrm{min}}
\ncm{\fl}{\mathrm{flat}}
\ncm{\sst}{\scriptstyle}
\ncm{\ssst}{\scriptscriptstyle}
\ncm{\coinv}[1]{^{\text{co-}#1}}

\ncm{\eqby}[1]{\stackrel{(\ref{#1})}{=}}
\ncm{\lef}{{\ssst <}}
\ncm{\righ}{{\ssst >}}

\ncm{\NN}{\mathbb{N}}
\ncm{\ZZ}{\mathbb{Z}}
\ncm{\QQ}{\mathbb{Q}}
\ncm{\RR}{\mathbb{R}}
\ncm{\CC}{\mathbb{C}}
\ncm{\GG}{\mathbf{G}}
\ncm{\FF}{\mathbb{F}}
\ncm{\DD}{\mathbb{D}}
\ncm{\onne}{\mathbb{1}}
\ncm{\TT}{\mathsf{T}}
\ncm{\Q}{\mathsf{Q}}

\ncm{\pisharp}{\Pisymbol{psy}{35}}
\ncm{\Pre}{\hat\C}
\ncm{\She}{\mathsf{Sh}}
\ncm{\Sig}{{\Sigma}}
\ncm{\icog}{J}

\ncm{\parallelpair}{
\parbox{43pt}{
\begin{picture}(43,8)
\put(3,6){\vector(1,0){37}}
\put(3,2){\vector(1,0){37}}
\end{picture}
}}
\ncm{\pair}[2]{\overset{#1}{\underset{#2}{\parallelpair}}}

\ncm{\shortparallelpair}{
\parbox{28pt}{
\begin{picture}(28,8)
\put(3,6){\vector(1,0){22}}
\put(3,2){\vector(1,0){22}}
\end{picture}
}}
\ncm{\shortpair}[2]{\overset{#1}{\underset{#2}{\shortparallelpair}}}

\ncm{\longparallelpair}{
\parbox{63pt}{
\begin{picture}(63,8)
\put(3,6){\vector(1,0){57}}
\put(3,2){\vector(1,0){57}}
\end{picture}
}}
\ncm{\longpair}[2]{\overset{#1}{\underset{#2}{\longparallelpair}}}

\ncm{\longerparallelpair}{
\parbox{83pt}{
\begin{picture}(83,8)
\put(3,6){\vector(1,0){77}}
\put(3,2){\vector(1,0){77}}
\end{picture}
}}
\ncm{\longerpair}[2]{\overset{#1}{\underset{#2}{\longerparallelpair}}}

\ncm{\longrightarrowtail}{
\parbox{40pt}{
\begin{picture}(40,8)
\put(6,4){\line(-1,1){1.5}}
\put(6,4){\line(-1,-1){1.5}}
\put(6,4){\vector(1,0){31}}
\end{picture}
}}

\ncm{\longerrightarrowtail}{
\parbox{60pt}{
\begin{picture}(60,8)
\put(6,4){\line(-1,1){1.5}}
\put(6,4){\line(-1,-1){1.5}}
\put(6,4){\vector(1,0){51}}
\end{picture}
}}

\ncm{\longrarr}[1]{
\overset{#1}{
\parbox{40pt}{
\begin{picture}(40,8)
\put(3,4){\vector(1,0){34}}
\end{picture}
}}}

\ncm{\longlarr}[1]{
\overset{#1}{
\parbox{40pt}{
\begin{picture}(40,8)
\put(37,4){\vector(-1,0){34}}
\end{picture}
}}}

\ncm{\longerrarr}[1]{
\overset{#1}{
\parbox{70pt}{
\begin{picture}(70,8)
\put(3,4){\vector(1,0){64}}
\end{picture}
}}}

\ncm{\longerlarr}[1]{
\overset{#1}{
\parbox{70pt}{
\begin{picture}(70,8)
\put(67,4){\vector(-1,0){64}}
\end{picture}
}}}

\ncm{\antiparallelpair}{
\parbox{23pt}{
\begin{picture}(23,4)
\put(3,3){\vector(1,0){17}}
\put(20,1){\vector(-1,0){17}}
\end{picture}
}}

\ncm{\invantiparallelpair}{
\parbox{23pt}{
\begin{picture}(23,4)
\put(3,1){\vector(1,0){17}}
\put(20,3){\vector(-1,0){17}}
\end{picture}
}}

\ncm{\dualpair}[2]{\overset{#1}{\underset{#2}{\antiparallelpair}}}

\ncm{\invdualpair}[2]{\overset{#1}{\underset{#2}{\invantiparallelpair}}}

\ncm{\coantiparallelpair}{
\parbox{23pt}{
\begin{picture}(23,4)
\put(3,1){\vector(1,0){17}}
\put(20,4){\vector(-1,0){17}}
\end{picture}
}}

\ncm{\doubleoval}{
\parbox{30pt}{
\begin{picture}(30,24)
\put(0,12){\circle*{3}}
\put(30,12){\circle*{3}}
\qbezier(0,12)(15,2)(30,12)
\qbezier(0,12)(15,8)(30,12)
\qbezier(0,12)(15,16)(30,12)
\qbezier(0,12)(15,22)(30,12)
\end{picture}
}}

\ncm{\codualpair}[2]{\overset{#1}{\underset{#2}{\coantiparallelpair}}}

\ncm{\binarydirectsum}[7]{#1\codualpair{#2}{#3}#4\dualpair{#5}{#6}#7}

\ncm{\equalizer}[1]{\overset{#1}{\longrightarrowtail}}

\ncm{\epi}[1]{\overset{#1}{\twoheadrightarrow}}

\ncm{\coequalizer}[1]{
\overset{#1}{
\parbox{40pt}{
\begin{picture}(40,8)
\put(2,4){\vector(1,0){32}} \put(37,4){\vector(1,0){0}}
\end{picture}
}}}

\ncm{\longcoequalizer}[1]{
\overset{#1}{
\parbox{60pt}{
\begin{picture}(60,8)
\put(2,4){\vector(1,0){52}} \put(57,4){\vector(1,0){0}}
\end{picture}
}}}

\ncm{\mono}[1]{\overset{#1}{\rightarrowtail}}

\ncm{\coequalizerfactorizationold}[9]{
\parbox[r]{115pt}{
\begin{picture}(115,70)(0,-5)
\put(0,48){$#1\longpair{#2}{#3}#4$} %\coequalizer{#5}#6
\end{picture}
}
\parbox[l]{80pt}{
\begin{picture}(80,70)(0,-5)
\put(2,51){\vector(1,0){70}} \put(75,51){\vector(1,0){0}}
\put(23,56){$\sst #5$}
\put(80,48){$#6$}
\put(2,47){\vector(2,-1){73}}
\put(48,30){$\sst #7$}
\dashline[+30]{3}(100,42)(100,12) \put(100,12){\vector(0,-1){0}}
\put(105,30){$\sst #8$}
\put(95,0){$#9$}
\end{picture}
}}

\ncm{\coequalizerfactorization}[9]{
\parbox[r]{143pt}{
\begin{picture}(143,70)(0,-5)
\put(0,48){$#1\longpair{#2}{#3}#4$} %\coequalizer{#5}#6
\end{picture}
}
\parbox[l]{80pt}{
\begin{picture}(80,70)(0,-5)
\put(2,51){\vector(1,0){70}} \put(75,51){\vector(1,0){0}}
\put(30,56){$\sst #5$}
\put(80,48){$#6$}
\put(2,47){\vector(2,-1){73}}
\put(41,30){$\sst #7$}
\dashline[+30]{3}(93,42)(93,12) \put(93,12){\vector(0,-1){0}}
\put(98,30){$\sst #8$}
\put(88,0){$#9$}
\end{picture}
}}

\begin{document}
%\allowdisplaybreaks

\newcommand{\arXivNumber}{2302.08027}

\renewcommand{\PaperNumber}{048}

\FirstPageHeading

\ShortArticleName{Oriented Closed Polyhedral Maps and the Kitaev Model}

\ArticleName{Oriented Closed Polyhedral Maps\\ and the Kitaev Model}

\Author{Korn\'el SZLACH\'ANYI}

\AuthorNameForHeading{K.~Szlach\'anyi}
\Address{Wigner Research Centre for Physics, Budapest, Hungary}
\Email{\href{mailto:szlachanyi.kornel@wigner.hu}{szlachanyi.kornel@wigner.hu}}

\ArticleDates{Received April 07, 2023, in final form May 14, 2024; Published online June 08, 2024}

\Abstract{A kind of combinatorial map, called arrow presentation, is proposed to encode the data of the oriented closed polyhedral complexes $\Sigma$ on which the Hopf algebraic Kitaev model lives. We develop a theory of arrow presentations which underlines the role of the dual of the double $\mathcal{D}(\Sigma)^*$ of~$\Sigma$ as being the Schreier coset graph of the arrow presentation, explains the ribbon structure behind curves on $\mathcal{D}(\Sigma)^*$ and facilitates computation of holonomy with values in the algebra of the Kitaev model. In this way, we can prove ribbon operator identities for arbitrary f.d.\ C$^*$-Hopf algebras and arbitrary oriented closed polyhedral maps. By means of a combinatorial notion of homotopy designed specially for ribbon curves, we can rigorously formulate ``topological invariance'' of states created by ribbon operators.}


\Classification{05E99; 16T05; 81T25}

\section{Introduction}

The toric code and its generalizations \cite{Bombin-Delgado,BMCA, Kitaev03} belong to the class of models called topological orders which are thought to describe
physical systems in which fault tolerant quantum computations can be realized. In its most general form, the model is a quantum system defined on the surface $\Sigma$ of an abstract polyhedron having no particular symmetry and with quantum spins sitting on the edges of $\Sigma$ and taking values in a finite-dimensional C$^*$-Hopf
algebra $H$. The mostly studied case is when $H$ is the group algebra of a finite abelian group and $\Sigma$ is the square lattice. This abelian Kitaev model exhibits
charges classified by the quantum double~$\D(H)$ that are created by semiinfinite string operators with the strings having no contribution to the energy, like Dirac strings.
These charges obey anyon statistics \cite{Kitaev03,Naaijkens-loc-endo}.
Much less is known beyond the abelian case. For general Hopf algebras, even the gauge theory interpretation is a~challenge~\mbox{\cite{Meusburger,Meusburger-Wise}}.
In this paper, we define the general model starting from scratch. As for the lattice, we do not want to go beyond ribbon graphs
used by \cite{Meusburger} but wish to replace them by something truly 2-dimensional on which taking duals and doubles are straightforward operations. As a~matter of fact, the dual of the double of $\Sigma$ is as important as the dual of the double of $H$ in the holonomy theory of \cite{Meusburger}. Indeed, the dual of the double of the surface complex is but the ``fattening'' of the ribbon graph. But it is more than that.
It is a sort of Cayley graph for the arrow presentation.

The idea of arrow presentation is well-known in the literature on graph embeddings, dessins d'enfants and permutation representations of triangle groups
under various names. For embeddings of finite, connected, simple graphs these are the ``rotation schemes'' of \cite[Section 6-6]{White} and
for finite, connected graphs these are the ``combinatorial maps'' of \cite[Definition~1.3.23]{Lando-Zvonkin}. For more general maps see \cite{Jones-Singerman,Lando-Zvonkin}.
A combinatorial map consists of a finite set $A$, the elements of which can be called directed edges, half edges, edge ends, darts or arrows, and of 3 permutations
$T_0$, $T_1$, $T_2$ of $A$ such that (1)~$T_2T_1T_0=\id_A$, (2)~the group generated by the $T_i$-s acts transitively on $A$ (connectedness) and (3)~the $T_1$ is an involution
without fixed points.
Such combinatorial data encode a ``map'', i.e., a kind of compact oriented 2-dimensional cellular complex $\Sigma$ without boundary,
in such a way that the $i$-cells of $\Sigma$ correspond bijectively to the cycles of the permutation $T_i$.
The arrow presentation we propose in this paper is a slight extension but also a specialization of combinatorial maps. We relax finiteness of $A$ and connectedness
but otherwise postulate stronger axioms on the permutations $T_i$. These stronger axioms lead to a~map $\Sigma$ in which there are no loop edges, no
degree 1 vertices and no repetitions of any cell along the boundary of a face and along the coboundary of a vertex.
Such a map $\Sigma$ will be called an oriented closed polyhedral map (OCPM).
Although the cardinality of $A$ is not restricted, the connected OCPMs are always countable. By allowing infinite $A$-s
such important examples as the square lattice on the plane and many other periodic and non-periodic tilings of the plane are given place among the connected OCPMs.

In Sections~\ref{section2},~\ref{sec: dudo} and~\ref{sec: curves}, we sketch a theory of arrow presentations and of OCPMs. The emphasis is on the structure of the dual of the double, $\D(\Sigma)^*$,
and on its ribbon curves. In Section~\ref{section3}, we define the algebra $\M=\M(\Sigma,H)$ of the Kitaev model on an OCPM $\Sigma$ and for a~finite-dimensional
semisimple Hopf algebra $H$. In contrast to its representations the algebra itself can be presented without ever introducing an orientation of the edges.
In Section~\ref{sec: holonomy}, we introduce an $\M$-valued holonomy on curves of $\D(\Sigma)^*$ which is motivated by previous works~\mbox{\cite{BMCA,Meusburger}}
although not depending on any relation to the Hopf algebra gauge theory of \cite{Meusburger-Wise}. Here we find a notion of central deformation of curves which will be useful in later Sections. In Section~\ref{section7}, we launch the hard work of computing algebraic relations of ribbon operators, i.e., of the holonomies
of ribbons. The results agree with the earlier results of \cite{Bombin-Delgado, Cowtan-Majid} obtained for group algebras or square lattices.
More efficient relations exist not for the ribbon operators themselves but for their actions on vacuum states which is the topic of Sections~\ref{section9} and~\ref{section10}.
In order to prepare for them we have to introduce certain combinatorial notions of homotopy both for curves and ribbons. This is the content of Section~\ref{section8}.
In Theorem~\ref{thm: topological invariance}, we are at last able to formulate topological invariance, or rather homotopy invariance, of states
created by ribbon operators from the vacuum.
Finding evidence for string localization of the charges in Section~\ref{section10} the paper ends abruptly leaving open the question of existence of these
charges, i.e., superselection sectors of the model.

\section{Polyhedral maps}\label{section2}

The main tool of this paper is the arrow presentation of certain surface complexes but arrow presentations themselves form an interesting mathematical structure.
\begin{defi}\label{def: AP}
An arrow presentation is a triple $\bra \Arr, T_0, T_2\ket$, where $\Arr$ is a set and for $i=0,2$ the $T_i\colon \Arr\to \Arr$
are bijective functions satisfying the following axioms:
\axiomenum{AP}\setlength{\leftskip}{0.4cm}
\item The orbits $\Ou_i(a):=\bigl\{T_i^n a\mid n\in\NN\bigr\}$ are finite sets containing at least 2 elements for all $a\in \Arr$ and for $i=0,2$.
\item Introducing $T_1:=T_0T_2$ we require that $T_1T_1=\id_\Arr$.
\item The intersection of a $T_0$-orbit and a $T_2$-orbit is either empty or a singleton.
\axiomenumend
\end{defi}
The geometric meaning of the above definition will be unveiled gradually. Here is a rough picture of what is going on:
The elements of $\Arr$ will be
called arrows and will become the directed edges of a surface complex. Assuming that the complex is closed and oriented the meaning of the $T_i$ transformations is this.
Every directed edge $a$ lies on the boundary of exactly~2~faces but only one of these faces, call it $f$, has compatible
orientation with the direction of $a$.
The orientation of $f$ determines a cyclic order of its boundary edges and we define $T_2a$ to be the successor of $a$ within this cyclic order.
Similarly one can define $T_0a$ as the successor of $a$ with respect to rotating around the source vertex of $a$ in counterclockwise direction. It is now
easy to imagine what the composite operation $T_1$ does; it reverses the direction of the arrows.\vspace{-3mm}
$$
\parbox{200pt}{
\begin{picture}(200,120)

\put(30,45){\polyline(0,0)(-10,20)(-30,10)(-30,-10)(-10,-20)(0,0)}
\put(30,45){\line(1,-1){15}} \put(30,45){\line(3,1){20}}
{\color{red}\put(15,45){\circle*{5}}}

\put(100,90){\polyline(0,0)(-10,20)(-30,10)(-30,-10)(-10,-20)(0,0)}
\put(100,90){\line(1,-1){15}} \put(100,90){\line(3,1){20}}
{\color{blue}\put(100,90){\circle*{5}}}

\put(170,45){\polyline(0,0)(-10,20)(-30,10)(-30,-10)(-10,-20)(0,0)}
\put(170,45){\line(1,-1){15}} \put(170,45){\line(3,1){20}}

\thicklines
\put(20,25){\vector(1,2){10}} \put(27,30){$a$}
\put(100,90){\vector(-1,2){10}} \put(96,103){$T_2a$}
\put(170,45){\vector(-1,-2){10}} \put(163,20){$T_0T_2a$}

\end{picture}
}\vspace{-7mm}
$$

If we added the axiom $T_2^3=\id_\Arr$, then we would be dealing with simplicial complexes but not doing so has the benefit of having a self-dual notion
of complex, as we shall see later.

Now we derive some formal consequences of Definition~\ref{def: AP}. Note that the set $\Arr$ is not assumed to be finite.
\begin{lem}\label{lem: OCPM}
Let $\bra \Arr, T_0, T_2\ket$ be an arrow presentation and let $T_1=T_0T_2$. Then
\begin{enumerate}\itemsep=0pt
\item[$(i)$] $T_2^{-1}=T_1T_0$,
\item[$(ii)$] $T_0^{-1}=T_2T_1$,
\item[$(iii)$] $\Ou_i(a)\cap\Ou_j(a)=\{a\}$ for all $a\in \Arr$ and for $0\leq i< j\leq 2$,
\item[$(iv)$] $T_1(a)\neq a$ for all $a\in \Arr$.
\end{enumerate}
\end{lem}
\begin{proof}
(i) $T_1 T_0 T_2=T_1T_1=\id_\Arr$ shows that $T_2$ has left inverse $T_1T_0$. Since $T_2$ is invertible, $T_1T_0$ is the inverse of $T_2$.

(ii) Using (i), we have $T_2T_1T_0=\id_\Arr$. Therefore, $T_2T_1$ is a left inverse of $T_0$. Since $T_0$ is invertible, $T_2T_1$ is the inverse of $T_0$.

(iii) The $i=0$, $j=2$ case is trivial since $a$ belongs to both orbits and by axiom (AP-3) no other element does so.
In the $i=0$, $j=1$ case, assume $b\in \Ou_0(a)\cap\Ou_1(a)$. Then either $b=a$ or $b=T_1a=T_0^ma$ for some $m$. In the latter case,
$T_0^{-1}b=T_0^{m-1}a=T_2a$ belongs to $\Ou_0(a)\cap\Ou_2(a)=\{a\}$ hence $T_2a=a$ which contradicts (AP-1).
Similarly, in the $i=1$, $j=2$ case a $c\in\Ou_1(a)\cap\Ou_2(a)$ which is not $a$ must satisfy $c=T_1a=T_2^na$ for some $n$ and therefore
$c=T_0T_2a=T_2^{n-1}T_2a$ belongs to $\Ou_0(T_2a)\cap\Ou_2(T_2a)$ hence $c=T_2a$. Now we have $T_1a=T_2a$ which means that
$T_2a$ is a fixed point of $T_0$ which contradicts~(AP-1).

(iv) Let $T_1a=a$. Then $T_2a=T_2T_1a=T_0^{-1}a$, so $T_2a\in \Ou_0(a)\cap\Ou_2(a)$. By (iii), $T_2a=a$ is a~fixed point of $T_2$ which
contradicts (AP-1).
\end{proof}

Out of an arrow presentation, we would like to build a 2-dimensional cellular complex, a kind of abstract polyhedron.
The necessary data are the following. Let
\begin{equation}\label{Sigma of AP}
\Sigma^i:=\{O\subseteq\Arr\mid \exists a\in\Arr,\,\Ou_i(a)=O\}
\end{equation}
be the set of $T_i$-orbits for $i=0,1,2$. We call the elements $v\in\Sigma^0$ vertices, the elements $e\in \Sigma^1$ (undirected) edges and the elements $f\in\Sigma^2$ faces. For an edge $e=\{a,T_1a\}$, we say that $u=\Ou_0(a)$ and
$v=\Ou_0(T_1a)$ are the endpoints, or boundary vertices, of $e$. In this case, we write $\Bd(e)=\{u,v\}$ and call it the boundary of $e$. The coboundary
of $e$ is defined by $\Cb(e):=\{\Ou_2(a),\Ou_2(T_1a)\}$. More generally, we define
\begin{align}
\label{Bd of AP}
&\Bd(x):=\{\Ou_{i-1}(a)\mid a\in x\}\qquad\forall x\in\Sigma^i,\quad i=1,2,\\
\label{Cb of AP}
&\Cb(x):=\{\Ou_{i+1}(a)\mid a\in x\}\qquad\forall x\in\Sigma^i,\quad i=0,1,
\end{align}
and we declare the boundaries of vertices and the coboundaries of faces to be empty.
Denoting by $\Sigma$ the disjoint union $\Sigma^0\sqcup\Sigma^1\sqcup\Sigma^2$ the $\Sigma^i$-s can be recovered by setting up a dimension function
$\dim\colon\Sigma\to\{0,1,2\}$. Then $\Bd$ and $\Cb$ become functions $\Sigma\to\Fin(\Sigma)$ to the set of finite subsets of~$\Sigma$.\looseness=-1

For a while, we forget about arrow presentations because first we have to formulate the requirements for a bunch of data, such as the
$\bra \Sigma,\dim,\Bd,\Cb\ket$ above, which qualifies it as a~polyhedral complex.

\begin{defi} \label{def: CPM}
The data $\bra \Sigma,\dim,\Bd,\Cb\ket$ consisting of a set $\Sigma$ and functions $\dim\colon\Sigma\to\{0,1,2\}$, $\Bd,\Cb\colon\Sigma\to\Fin(\Sigma)$
is called a closed polyhedral map (CPM) if the following axioms hold:
\axiomenum{CPM}\setlength{\leftskip}{0.72cm}
\item For $s,t\in\Sigma$, if $s\in\Bd(t)$, then $\dim t=1+\dim s$.
\item $s\in\Bd(t)$ if and only if $t\in\Cb(s)$ for all $s,t\in\Sigma$.
\item $|{\Bd}(e)|=2$ and $|\Cb(e)|=2$ for all edges $e\in\Sigma^1$.
\item For all $v\in\Sigma^0$, the elements of the vertex neighbourhood $\Nb(v):=\Cb(v)\cup\bigcup\{\Cb(e)\mid e\in\Cb(v)\}$ can be enumerated
in such a way that it consists of edges $e_0,\dots,e_{n-1}$ and faces $f_0,\dots,f_{n-1}$ for some $n\geq 2$ with
$\Cb(e_k)=\{f_k,f_{k+1}\}$ for $0\leq k<n$, where $f_n=f_0$.

\item For all $f\in\Sigma^2$, the elements of the face neighbourhood $\Nb(f):=\Bd(f)\cup\bigcup\{\Bd(e)\mid e\in\Bd(f)\}$ can be enumerated
in such a way that it consists of edges $e_0,\dots,e_{n-1}$ and vertices $v_0,\dots,v_{n-1}$ for some $n\geq 2$ with
$\Bd(e_k)=\{v_k,v_{k+1}\}$ for $0\leq k<n$, where $v_n=v_0$.
\axiomenumend
\end{defi}

For comparison of CPM-s with the complexes one obtains from combinatorial maps see Example~\ref{exa: 3 torus maps}.

The geometric realization of a CPM $\bra \Sigma,\dim,\Bd,\Cb\ket$ is a 2-manifold $[\Sigma]$ without boundary that can be constructed from polygonal disks
$[f]$ for each face $f\in\Sigma^2$ by sewing edges together. Axiom (CPM-5) ensures that the neighbourhood $\Nb(f)$ of a face $f$ consists of $n_f$
edges and $n_f$ vertices for some $n_f\geq 2$ which form a circular chain so we can associate to the abstract cell $f$ a topological closed disk $[f]$
the boundary of which is an $n_f$-gon and the edges and vertices of this polygon are labelled by the cells of $\Nb(f)$.
Doing this for all $f\in\Sigma^2$ every edge label $e\in\Sigma^1$ occurs twice sitting on the boundary of two different polygonal disks, by axiom CPM-3.
There is only one way to identify two edges of the same label which is compatible with the labels on their boundary vertices also by axiom CPM-3.
Having done these identifications for all edges, we get a topological space $[\Sigma]$ in which every point has a neighbourhood homeomorphic to an open disk,
hence a topological closed surface. The points where this is non-trivial are the vertices but for them axiom CPM-4 ensures open disk neighbourhoods.

$[\Sigma]$ is compact if and only if $\Sigma$ is finite. The non-trivial `only if' part follows from the existence of a cover of $[\Sigma]$
by open disks, with one disk for each cell, such that no proper subset of the cover is a cover.

Although Definition~\ref{def: CPM} deals only with closed complexes, i.e., ones without boundary and without coboundary,
it is intuitively clear how to modify the definition in order to get polyhedral map (PM) satisfying the following theorem: \emph{Every PM can be obtained
from a CPM by discarding a set of faces $($holes$)$ and a set of vertices $($punctures$)$}.

\begin{lem}\label{lem: site}
For a vertex $v$ and a face $f$ of a CPM $\Sigma$ the following conditions are equivalent:
\begin{enumerate}\itemsep=0pt
\item[$({\rm i})$] $v\in\Nb(f)$,
\item[$({\rm ii})$] $f\in\Nb(v)$,
\item[$({\rm iii})$] $\Cb(v)\cap\Bd(f)\neq\varnothing$.
\end{enumerate}
Provided these conditions hold the intersection $\Cb(v)\cap\Bd(f)$ consists of precisely $2$ edges.
\end{lem}
\begin{proof}
The relation $v\in\Nb(f)$ is the square of the boundary relation, i.e., $\exists e\, (v\in\Bd(e)\,\wedge\,e\in\Bd(f))$. By (CPM-2), the transpose of $\Bd$
is $\Cb$ so (i) is equivalent to both (ii) and (iii).
If $v\in\Nb(f)$, then axiom (CPM-5) shows explicitly the 2 coboundary edges of $v$ that lie on the boundary of $f$.
\end{proof}

If the CPM $\Sigma$ arises from an arrow presentation as in equations (\ref{Sigma of AP}) and (\ref{Bd of AP}), then the conditions of Lemma~\ref{lem: site}
can be supplemented by a fourth one, namely,
\begin{equation}\label{site}
v\cap f\neq\varnothing\quad\Leftrightarrow\quad \Cb(v)\cap\Bd(f)\neq\varnothing.
\end{equation}
In this case, the two edges can be explicitly given by
\begin{equation}\label{2 edges of a site}
\Cb(v)\cap\Bd(f)=\{\Ou_1(a),\Ou_1(T_0a)\}\qquad\text{where}\quad\{a\}=v\cap f.
\end{equation}

\begin{defi}\label{def: site}
A pair $\bra v,f\ket\in\Sigma^0\x\Sigma^2$ satisfying the conditions of Lemma~\ref{lem: site} is called a site. Two sites $s_i=\bra v_i,f_i\ket$, $i=1,2$,
can be equal (if $v_1=v_2$ and $f_1=f_2$) or neighbour (if either $v_1=v_2$ or $f_1=f_2$ but not both) or disjoint (if none of the 2 equations holds).
\end{defi}

The following equivalent conditions define connectedness of a CPM $\Sigma$:
\begin{itemize}\itemsep=0pt
\item For every $u,v\in\Sigma^0$, there exists a sequence $(v_0,e_1,v_1,\dots,e_n,v_n)$ of edges and vertices such that $v_0=u$, $v_n=v$ and
$\Bd(e_k)=\{v_{k-1},v_k\}$ for $0<k\leq n$.
\item For every $p,q\in\Sigma^2$, there exists a sequence $(f_0,e_1,f_1,\dots,e_n,f_n)$ of edges and faces such that $f_0=p$, $f_n=q$ and
$\Cb(e_k)=\{f_{k-1},f_k\}$ for $0<k\leq n$.
\item For every pair $s$, $t$ of sites, there exists a sequence $(s_i)_{i=0}^n$ of sites such that $s_0=s$, $s_n=t$ and the $\{s_{i-1},s_i\}$ are
neighbours for $0<i\leq n$.
\end{itemize}

If $\Sigma$ is connected, then it is countable. As a matter of fact, such a $\Sigma$ can be written as a~countable union of finite subsets.

Axioms (CPM-4) and (CPM-5) offer an easy way to define orientation of CPMs. Axiom (CPM-5) allows to choose a cyclic order
$[v_0,e_0,v_1,e_1,\dots,v_{n-1},e_{n-1}]$ on every face neighbourhood $\Nb(f)$ where $[x_1,\dots,x_k]$
denotes sequences $(x_1,\dots,x_k)$ modulo cyclic permutations. Since $n\geq 2$ by assumption, this cyclically ordered set has at least 4 elements.
Therefore, there exists exactly two cyclic orders on $\Nb(f)$ satisfying (CPM-5). (If we did not include the vertices into~$\Nb(f)$, we would be in trouble
with the 2-gons.)
Similarly, (CPM-4) allows to choose among the two possible cyclic orders on every vertex neighbourhood $\Nb(v)$.

\begin{defi}\label{def: ori}
Let $\bra v,f\ket$ be a site. Cyclic orders $[f_0,e_0,f_1,e_1,\dots,f_{m-1},e_{m-1}]$ of $\Nb(v)$ and
$[v_0,e_0,v_1,e_1,\dots,v_{n-1},e_{n-1}]$ of $\Nb(f)$ are called compatible if the predecessor and successor edges $e_p<f<e_s$ of $f$
in $\Nb(v)$ and the predecessor and successor edges $e'_p<v<e'_s$ of $v$ in~$\Nb(f)$ are related by $e'_p=e_s$ and $e'_s=e_p$. Otherwise
\big(i.e., if $e'_p=e_p$ and $e'_s=e_s$\big), they are called incompatible.

A system of cyclic orders given on all vertex and face neighbourhoods of $\Sigma$ is called an orientation if for every site
$\bra v,f\ket$ on $\Sigma$ the cyclic orders of $\Nb(v)$ and $\Nb(f)$ are compatible.
A~CPM $\Sigma$ is called orientable if it has an orientation.
\end{defi}

\begin{thm}\label{thm: CPM}
Given an arrow presentation $\bra \Arr,T_0,T_2\ket$ the associated data
$\bra \Sigma,\dim,\Bd,\Cb\ket$ defined by \eqref{Sigma of AP}, \eqref{Bd of AP}
and \eqref{Cb of AP} is a closed polyhedral map in the sense of Definition~$\ref{def: CPM}$.
Furthermore, the actions of $T_0$ and $T_2$ define canonical cyclic orders
\begin{gather}
\notag[f_0,e_0,\dots,f_{m-1},e_{m-1}]\\
\qquad {} :=\bigl[\Ou_2(a),\Ou_1(T_0a),\Ou_2(T_0a),\Ou_1\bigl(T_0^2a\bigr),\dots,\Ou_2\bigl(T_0^{m-1}a\bigr),\Ou_1(a)\bigr]\label{cycord Nb(v)}
\end{gather}
on vertex neighbourhoods $\Nb(\Ou_0(a))$ and
\begin{gather}
\notag[v_0,e_0,\dots,v_{n-1},e_{n-1}]\\
\qquad {}:=\big[\Ou_0(a),\Ou_1(a),\Ou_0(T_2a),\Ou_1(T_2a),\dots,\Ou_0(T_2^{n-1}a),\Ou_1\bigl(T_2^{n-1}a\bigr)\big]\label{cycord Nb(f)}
\end{gather}
on face neighbourhoods $\Nb(\Ou_2(a))$. These cyclic orders are independent of the choice of $a$ within $v=\Ou_0(a)$ and within $f=\Ou_2(a)$,
respectively, and constitute an orientation of $\Sigma$ in the sense of Definition~$\ref{def: ori}$.
In this way every arrow presentation determines an oriented closed polyhedral map.
\end{thm}
\begin{proof}
(CPM-1) is clear. The relation $s\in\Bd(t)$ means that there is an $a\in t$ such that $s=\Ou_{i-1}(a)$ where $i=\dim t$. But this is just the symmetric
relation $s\cap t\neq\varnothing$. Inspecting also~$\Cb(s)$, we get
\begin{equation}\label{Bd cap Cb}
s\in\Bd(t)\quad\Leftrightarrow\quad s\cap t\neq\varnothing\quad\Leftrightarrow\quad t\in\Cb(s),
\end{equation}
which proves (CPM-2).

Every edge $e$ is a 2-element set $\{a,T_1a\}$ by (AP-2) and by Lemma~\ref{lem: OCPM}\,(iv). Therefore, both the boundary and the coboundary
of $e$ can contain at most 2 elements. The case $\Ou_i(a)=\Ou_i(T_1a)$ for $i=0$ or $i=2$, however, is forbidden because it implies that
$a\neq T_1a\in\Ou_1(a)\cap\Ou_i(a)$ contradicting Lemma~\ref{lem: OCPM}\,(iii). This proves (CPM-3).

Let $v\in\Sigma^0$. Then the transformation $T_0$ gives a natural cyclic order on the set $v$ of arrows, hence also on the set $\Cb(v)$ of edges.
Let an $a\in v$ be fixed and set $e_k:=\Ou_1\bigl(T_0^ka\bigr)$ for $0\leq k<n=|v|$. The coboundary of $e_k$ is the pair
$\Cb(e_k)=\bigl\{\Ou_2\bigl(T_0^ka\bigr),\Ou_2\bigl(T_1T_0^ka\bigr)\bigr\}$. Noticing that Lemma~\ref{lem: OCPM}\,(i) implies that $\Ou_2\bigl(T_1T_0^ka\bigr)=\Ou_2\bigl(T_0^{k-1}a\bigr)$, we
will denote $\Ou_2\bigl(T_0^{k-1}a\bigr)$ by $f_k$ and then get the desired formula $\Cb(e_k)=\{f_k,f_{k+1}\}$ for all $k$. It remains to show that there are
exactly $n$ different $e_k$-s and $n$ different $f_k$-s. If $e_k=e_l$, then $T_0^ka$ is equal to either $T_0^la$ or $T_1T_0^la$. In the first case
$k=l \mod n$ because $n$ is the size of $\Ou_0(a)$. In the second case $T_1T_0^la$ belongs to both the $T_1$-orbit and the $T_0$-orbit of
$b=T_0^la$. So $T_1T_0^la =b$ by Lemma~\ref{lem: OCPM}\,(iii) and $b$ is a fixed point of $T_1$ contradicting Lemma~\ref{lem: OCPM}\,(iv). Thus, all $e_k$-s
are different. In particular, $n=|\Cb(v)|$. Now assume $f_k=f_l$ for some $0\leq k<l<n$. Then, there is a power $T_2^m$ such that $T_2^mb=T_0^{l-k}b$
for $b=T_0^{k-1}a$. Again by Lemma~\ref{lem: OCPM}\,(iii), we obtain $T_0^{l-k}b=b$ hence $k=l \mod n$. So all $f_k$-s are different and the vertex
neighbourhood $\Nb(v)$ has the required circular structure. This proves (CPM-4).

The proof of (CPM-5) goes exactly as that of (CPM-4). Namely $\bigl\langle A,T_2^{-1},T_0^{-1}\bigr\rangle$
is also an arrow presentation the $T_1$-orbits of which coincide with the original because $T_2^{-1}T_0^{-1}=T_0T_2$. The $T_0$-orbits and $T_2$-orbits in turn are interchanged so (CPM-5) for $\bra \Arr,T_0,T_2\ket$
is equivalent to (CPM-4) for $\bigl\langle A,T_2^{-1},T_0^{-1}\bigr\rangle$.

It remains to show compatibility of the cyclic orders given in (\ref{cycord Nb(v)}) and (\ref{cycord Nb(f)}) whenever~$v$ and~$f$ form a site.
By (\ref{site}) and Lemma~\ref{lem: site}, this means that $v\cap f\neq\varnothing$ so we can choose the $a$ in both~(\ref{cycord Nb(v)}) and~(\ref{cycord Nb(f)}) to be the unique element of $v\cap f$.
Then the 2 common edges of~$\Nb(v)$ and~$\Nb(f)$ are~$\Ou_1(a)$ and~$\Ou_1(T_0a)=\Ou_1\bigl(T_2^{-1}a\bigr)$ the first
of which is the predecessor of $\Ou_2(a)=f$ in $\Nb(v)$ and the successor of $\Ou_0(a)=v$ in $\Nb(f)$ while the second is the successor of $f$ in $\Nb(v)$
and the predecessor of $v$ in $\Nb(f)$. Therefore, they are compatible.
\end{proof}

A CPM $\Sigma$ together with an orientation will be called an oriented closed polyhedral map (OCPM).

We will not formalize the converse of the above theorem which claims that every OCPM determines an arrow presentation; the idea has already been sketched
below Definition~\ref{def: AP}.
The two constructions AP $\rarr{}$ OCPM and OCPM $\rarr{}$ AP establish a bijective correspondence between (isomorphism classes of) oriented closed polyhedral
maps and arrow presentations. The discussion of isomorphisms is postponed until Section~\ref{sec: dudo}.

\begin{exa}\label{exa: 3 torus maps}
In order to illustrate the difference between combinatorial maps and arrow presentations, in particular, the role of axiom (AP-3), let us consider three maps on the torus:
\begin{align*}
\Sigma_{1\x 1}&=\quad
\parbox{60pt}{
\begin{picture}(40,40)
\put(0,0){\circle*{4}}
\put(40,0){\circle{4}}
\put(0,40){\circle{4}}
\put(40,40){\circle{4}}
\thicklines
\put(0,0){\vector(1,0){38}} \put(18,3){$\scriptstyle 1$}
\put(0,0){\vector(0,1){38}} \put(3,18){$\scriptstyle 2$}
\thinlines
\put(2,40){\vector(1,0){36}} \put(18,43){$\scriptstyle 1$}
\put(40,2){\vector(0,1){36}} \put(43,18){$\scriptstyle 2$}
\end{picture}
}
\qquad
\begin{aligned}
&A=\{1,2,\bar 1,\bar 2\},\\
&T_0=(12\bar 1\bar 2),\\
&T_2=(12\bar 1\bar 2),
\end{aligned}
\\
\Sigma_{1\mash 2}&=\quad
\parbox{60pt}{
\begin{picture}(60,30)
\put(0,0){\circle*{4}}
\put(30,0){\circle*{4}}
\put(0,30){\circle{4}}
\put(30,30){\circle{4}}
\put(60,0){\circle{4}}
\put(60,30){\circle{4}}
\thicklines
\put(2,0){\vector(1,0){26}} \put(15,3){$\scriptstyle 1$}
\put(0,2){\vector(0,1){26}} \put(3,13){$\scriptstyle 2$}
\put(32,0){\vector(1,0){26}} \put(45,3){$\scriptstyle 3$}
\put(30,2){\vector(0,1){26}} \put(33,13){$\scriptstyle 4$}
\thinlines
\put(2,30){\vector(1,0){26}} \put(15,33){$\scriptstyle 3$}
\put(32,30){\vector(1,0){26}} \put(45,33){$\scriptstyle 1$}
\put(60,2){\vector(0,1){26}} \put(63,13){$\scriptstyle 2$}
\end{picture}
}
\qquad
\begin{aligned}
&A=\{1,2,3,4,\bar 1,\bar 2,\bar 3,\bar 4\},\\
&T_0=(12\bar 3\bar 4)(\bar 1\bar 234),\\
&T_2=(14\bar 3\bar 2)(32\bar 1\bar 4),
\end{aligned}
\\
\Sigma_{2\x 2}&=\quad
\parbox{60pt}{
\begin{picture}(60,48)
\put(0,0){\circle*{4}}
\put(24,0){\circle*{4}}
\put(0,24){\circle*{4}}
\put(24,24){\circle*{4}}
\put(0,48){\circle{4}}
\put(24,48){\circle{4}}
\put(48,0){\circle{4}}
\put(48,24){\circle{4}}
\put(48,48){\circle{4}}
\thicklines
\put(2,0){\vector(1,0){20}} \put(12,3){$\scriptstyle 1$}
\put(0,2){\vector(0,1){20}} \put(3,10){$\scriptstyle 2$}
\put(26,0){\vector(1,0){20}} \put(36,3){$\scriptstyle 3$}
\put(24,2){\vector(0,1){20}} \put(27,10){$\scriptstyle 4$}
\put(2,24){\vector(1,0){20}} \put(12,27){$\scriptstyle 5$}
\put(0,26){\vector(0,1){20}} \put(3,34){$\scriptstyle 6$}
\put(26,24){\vector(1,0){20}} \put(36,27){$\scriptstyle 7$}
\put(24,26){\vector(0,1){20}} \put(27,34){$\scriptstyle 8$}
\thinlines
\put(2,48){\vector(1,0){20}}
\put(26,48){\vector(1,0){20}}
\put(48,2){\vector(0,1){20}}
\put(48,26){\vector(0,1){20}}
\end{picture}
}
\qquad
\begin{aligned}
&A=\{1,2,3,4,5,6,7,8,\bar 1,\bar 2,\bar 3,\bar 4,\bar 5,\bar 6,\bar 7,\bar 8\},\\
&T_0=(12\bar 3\bar 6)(34\bar 1\bar 8)(56\bar 7\bar 2)(78\bar 5\bar 4),\\
&T_2=(14\bar 5\bar 2)(32\bar 7\bar 4)(58\bar 1\bar 6)(76\bar 3\bar 8).
\end{aligned}
\end{align*}
The arrows are named $1,\dots,n$, $\bar 1,\dots,\bar n$ with $\bar k$ being the reverse of $k$. The permutations $T_i$ are given by their cycle decompositions.
So, in each of the cases $T_1=(1\bar 1)\dots(n\bar n)$.

$\Sigma_{1\x 1}$ arises from embedding the figure 8 graph (of 1 vertex and 2 loops)
into the torus so that the two loops become a pair of meridian and longitude.
$\Sigma_{1\mash 2}$ arises from embedding \smash{$\ \doubleoval\ $} into the torus in the following way. First, we draw a length 2 longitude and two disjoint length~1
meridians of distance 1 from each other. Then, we cut along the longitude, shift it by 1 and glue it back (half Dehn twist).
$\Sigma_{2\x 2}$ is also an embedding of a graph which can be best drawn on the torus itself.

It is easy to see that in each of the three cases the data $\bra A, T_0,T_1,T_2\ket$ obey the three axioms of combinatorial maps mentioned in the introduction.
Moreover, neither of the $T_i$ have a fixed point so also the axioms (AP-1) and (AP-2) are fulfilled.
In the map $\Sigma_{1\x 1}$ we see the same face on both hand sides of an edge (isthmus) because there is only one face. Also the edges are bordered by one and the same
vertex. So this example is breaking the axioms (CPM-3), (CPM-4) and \mbox{(CPM-5)}. In terms of the permutations, we see the weird identity $T_0=T_2$
implying immediately that (AP-3) fails.
In the map $\Sigma_{1\mash 2}$, there are no loops and no isthmuses but there are repeated faces
in the vertex neighbourhoods and repeated vertices in the face neighbourhoods. So \mbox{(CPM-4)} and \mbox{(CPM-5)}
are broken. Both the first cycle of~$T_0$ and the second cycle of~$T_2$ contains~$2$ and~$\bar 4$
which means that (AP-3) fails. This is not an arrow presentation, either.
The map $\Sigma_{2\x 2}$ is the only one among the three which is an OCPM and the only one
for which $\bra A,T_0,T_2\ket$ is an arrow presentation.
This example proves independence of the axiom (AP-3) and shows also that
in order to produce a map with given genus arrow presentations need more arrows than combinatorial maps do.
\end{exa}

Let us mention briefly two important constructions of CPMs which work also for non-orientable $\Sigma$. See Figure~\ref{fig1}.
\begin{defi}\label{def: double}
Let $\Sigma=\bra\Sigma,\dim,\Bd,\Cb\ket$ be a closed polyhedral map.
Its dual is
\begin{equation} \label{du Sigma}
\Sigma^*:=\bra\Sigma,2-\dim,\Cb,\Bd\ket
\end{equation}
and its double $\D=\D(\Sigma)=\bigl\bra \D^0,\D^1,\D^2,\mathbf{Bd},\mathbf{Cb}\bigr\ket$ is defined by{\samepage
\begin{align*}
&\D^0:=\Sigma^0\sqcup \Sigma^1 \sqcup \Sigma^2,\\
&\D^1:=\bigl\{\bra v,e\ket\in\Sigma^0\x\Sigma^1\mid v\in\Bd(e)\bigr\} \sqcup \bigl\{\bra e,f\ket\in\Sigma^1\x\Sigma^2\mid e\in\Bd(f)\bigr\},\\
&\D^2:=\bigl\{\bra v,f\ket\in\Sigma^0\x\Sigma^2\mid \Cb(v)\cap\Bd(f)\neq\varnothing\bigr\},\\
&\mathbf{Bd}(\bra s,t\ket):=\{s,t\}\qquad\text{for}\ \bra s,t\ket\in\D^1,\\
&\mathbf{Bd}(\bra v,f\ket):=\{\bra v,e_1\ket,\bra v,e_2\ket,\bra e_1,f\ket,\bra e_2,f\ket\},
\end{align*}
where $\{e_1,e_2\}=\Cb(v)\cap\Bd(f)$
for $\bra v,f\ket\in\D^2$.}

The coboundary $\mathbf{Cb}$ is then obtained by transposing the relation $x\in\mathbf{Bd}(y)$.
\end{defi}

\begin{figure}[t]\centering
\includegraphics[width=6cm]{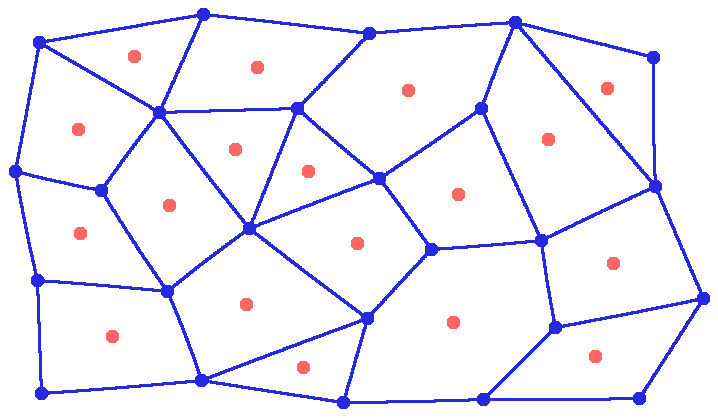}\hskip 0.5cm \includegraphics[width=6cm]{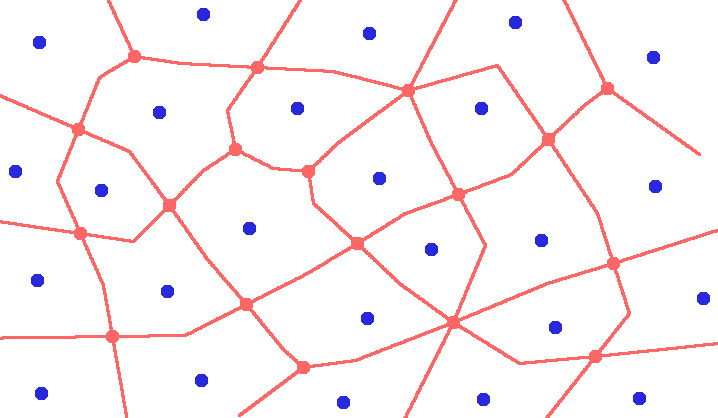}
\vskip 0.3cm
\includegraphics[width=6cm]{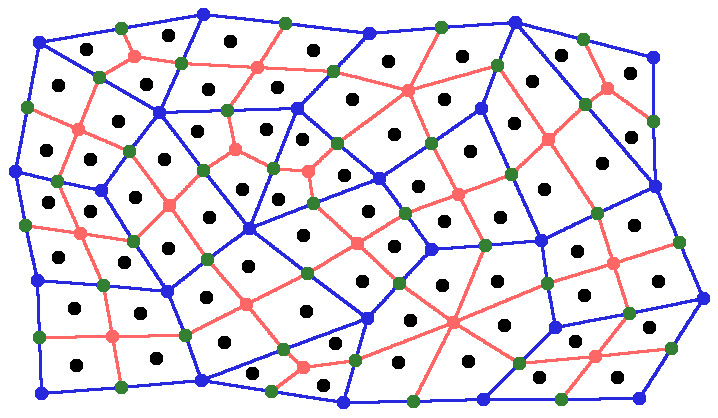}\hskip 0.5cm \includegraphics[width=6cm]{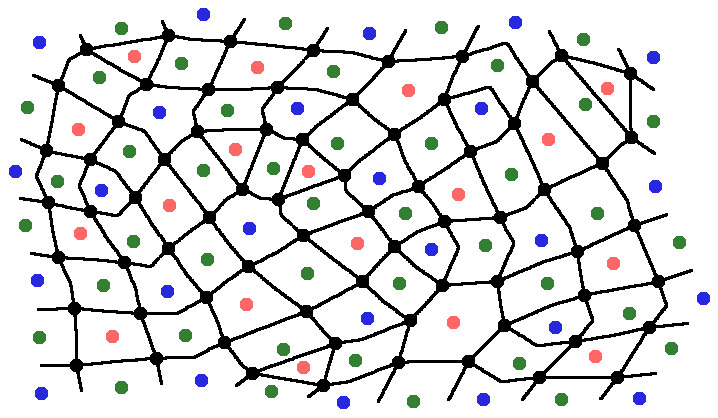}
\caption{A planar region of a CPM $\Sigma$ and the same region of $\Sigma^*$, $\D(\Sigma)$ and $\D(\Sigma)^*$. Blue dots are the vertices, red dots are the faces.
In the 2nd row the green dots are the edges, black ones are the sites.}\label{fig1}
\end{figure}

Notice that the edges of $\D(\Sigma)$ are naturally divided into two classes: the $ve$-type edges and the $ef$-type edges.
Pairs of $ve$ edges organize into the edges of $\Sigma$ and pairs of $ef$ edges do so for the dual complex $\Sigma^*$, explaining the name ``double'' of
$\Sigma$. The double $\D(\Sigma)$ is always free of multiple edges and loops.

Notice that every face of $\D(\Sigma)$ is a quadrangle. The faces of the double are precisely the sites of Definition~\ref{def: site}.
They are just the vertices of the dual of the double. Therefore, every vertex of $\D(\Sigma)^*$ is 4-valent.

\begin{exa}\label{exa: minimal OCPM}
The minimal OCPM. Let $\Sigma_\mini$ be the CPM consisting of 2 vertices, 2 edges and~2~faces. Topologically $\Sigma_\mini$ is a sphere made up of two hemispheres
$f_1$, $f_2$ separated by an equator consisting of two edges $e_1$, $e_2$ and two vertices $v_1$, $v_2$ such that $\Bd(e_1)=\Bd(e_2)=\{v_1,v_2\}$ and
$\Cb(e_1)=\Cb(e_2)=\{f_1,f_2\}$. Fixing $a$ to be one of the 4 arrows, we give an orientation to $\Sigma_\mini$ by the arrow presentation
\[
\bra\{a,T_1a,T_0a,T_2a\}, T_0,T_2\ket,
\]
where the action of $T_0$ and $T_2$ are fixed by the relations $T_0^2=\id$, $T_2^2=\id$. Since also $(T_0T_2)^2=\id$, this is the standard presentation of the left regular
$G$-set over the Klein four-group $G=Z_2\x Z_2$. It follows that if $\bra A,T_0,T_2\ket$ is an arrow presentation of a connected OCPM $\Sigma$ such that there is
an $a\in A$ satisfying $T_0^2a=a=T_2^2a$ then $\Sigma=\Sigma_\mini$.
The dual of the double
\[
\D(\Sigma_\mini)^*\ =
\parbox{100pt}{
\begin{picture}(100,70)
\linethickness{0.5pt}
\polyline(25,15)(75,15)(75,55)(25,55)(25,15)
\put(25,15){\circle*{4}}
\put(75,15){\circle*{4}}
\put(75,55){\circle*{4}}
\put(25,55){\circle*{4}}
\qbezier(25,15)(50,-20)(75,15)
\qbezier(75,15)(115,35)(75,55)
\qbezier(25,55)(50,90)(75,55)
\qbezier(25,15)(-15,35)(25,55)
\color{melegzold}
\put(48,32){$e_1$}
\put(120,32){$e_2$}
\color{blue}
\put(11,32){$v_1$}
\put(82,32){$v_2$}
\color{red}
\put(46,60){$f_1$}
\put(46,3){$f_2$}
\end{picture}}
\]
is another spherical OCPM the faces of which are the cells of $\Sigma_\mini$.
\end{exa}

In the rest of the paper, we shall work mainly with connected oriented closed polyhedral maps and freely mix the languages of $\bra \Arr,T_0,T_2\ket$ and
$\bra \Sigma,\dim,\Bd,\Cb\ket$. We write $\Sigma\,{::}\,P$ if $P=\bra \Arr,T_0,T_2\ket$ is a presentation of the OCPM $\Sigma$.

\section{Kitaev's model over a Hopf algebra}\label{section3}

Before launching the definition of the model, let us fix our conventions and notations about Hopf algebras. We work with finite-dimensional Hopf algebras
$\bra H,\,\cdot\,,1,\cop,\eps,S\ket$ over a field $K$. Later in Section~\ref{section10}, $K$ will be the complex field.
Let $H^*$ denote the dual Hopf algebra structure on the space of linear functions $H\to K$. The evaluation of $\varphi\in H^*$ on $h\in H$ is denoted either by~$\varphi(h)$ or, if~$\varphi$ is a long expression, by $\bra \varphi,h\ket$ or even by $\bra h,\varphi\ket$ since $H^{**}$ can be identified with~$H$.
We use the standard Hopf algebraist notation
$\cop(h)=h\p\ot h\pp$ for the coproduct of an element which is a concise notation for a linear combination $\sum_i h_{(1)i}\ot h_{(2)_i}$.
We often push this convention to the extreme and write, for example, $\varphi\ot h\in H^*\ot H$ but mean a general linear combination $\sum_i\varphi_i\ot h_i$ of
some $\varphi_i\in H^*$ and $h_i\in H$.
In want of enough letters and due to abundance of $^*$-s in this paper the antipode and counit of $H^*$ will be denoted also by~$S$ and~$\epsilon$, respectively.

A basis $\{x_i\}$ of $H$ and a basis $\{\xi_i\}$ of $H^*$ are called dual bases if $ \xi_i(x_j)=\delta_{i,j}$.
For the indices of these bases, we use Einstein's summation convention: repeated indices are summed over. In this sense, we have the useful identities
\begin{alignat*}{3}
&x_i\ot x_j\ot \xi_i\xi_j=x_{k(1)}\ot x_{k(2)}\ot\xi_k,\qquad&& x_ix_j\ot \xi_i\ot\xi_j=x_k\ot\xi_{k(1)}\ot\xi_{k(2)},&\\
&S(x_i)\ot\xi_i=x_i\ot S(\xi_i),&& &\\
&\bra\varphi,x_i\ket\,\xi_i=\varphi,\qquad && x_i \bra\xi_i,y\ket=y.&
\end{alignat*}
The Hopf algebra $H$ is a left and right module algebra over $H^*$ with respect to the actions $\varphi\sweedl h:=h\p\varphi\bigl(h\pp\bigr)$
and $h\sweedr \varphi:=\varphi\bigl(h\p\bigr)h\pp$, respectively.

The Drinfeld double $\D(H)$, a.k.a.\ the quantum double, is the following Hopf algebra structure on the vector space $H^* \ot H$: The multiplication is
\[
(\varphi\ot g)(\psi\ot h)=\varphi\psi\pp\ot g\pp h\cdot \psi\ppp\bigl(g\p\bigr)\psi\p\bigl(S^{-1}\bigl(g\ppp\bigr)\bigr)
\]
with unit $1_\D=\eps\ot 1$ and the comultiplication is
\[
\cop(\psi\ot h)=\bigl(\psi\pp\ot h\p\bigr)\ot\bigl(\psi\p\ot h\pp\bigr)
\]
with counit $\eps_\D(\psi\ot h)=\psi(1)\eps(h)$. The antipode is given by the formula
\[
S_\D(\varphi\ot h)=S^{-1}\bigl(\varphi\pp\bigr)\ot S\bigl(h\pp\bigr)\cdot\varphi\p\bigl(h\ppp\bigr)\varphi\ppp \bigl(S^{-1}\bigl(h\p\bigr)\bigr).
\]
Equipped with the $R$-matrix
\[
R\equiv R_1\ot R_2:=(\eps\ot x_i)\ot(\xi_i\ot 1)
\]
the double $\D(H)$ is a quasitriangular Hopf algebra for any finite-dimensional $H$.

The dual Hopf algebra of $\D(H)$ is denoted by $\D(H)^*$. It is the vector space $H\ot H^*$ equipped with algebra structure
\begin{equation} \label{D^* mul}
(h\ot \varphi)(g\ot\psi)=gh\ot\varphi\psi,\qquad 1_{\D^*}=1\ot\eps,
\end{equation}
with coalgebra structure
\begin{equation} \label{D^* comul}
\cop_{\D^*}(h\ot\varphi)=\bigl(h\p\ot\xi_i\varphi\p\xi_j\bigr)\ot\bigl(S^{-1}(x_j)h\pp x_i\ot \varphi\pp\bigr),\qquad \eps_{\D^*}(h\ot\varphi)=\eps(h)\varphi(1)
\end{equation}
and with antipode
\begin{align}\label{Sdudo}
S_{\D^*}(h\ot\varphi)&=x_jS^{-1}(h)S^{-1}(x_i)\ot\xi_i S(\varphi)\xi_j=(x_i\ot\xi_i)^{-1}\bigl(S^{-1}(h)\ot S(\varphi)\bigr)(x_j\ot\xi_j).
\end{align}

An element $h\in H$ is called cocommutative if $h\pp\ot h\p=h\p\ot h\pp$. The set $\Cocom H$ of cocommutative elements is a subalgebra
(but not a subcoalgebra, in general) of $H$.

If $H$ is involutive, i.e., $S^2=\id_H$, then so is $H^*$, $\D(H)$ and $\D(H)^*$. Involutive Hopf algebras are very close to being semisimple.
By a theorem of Larson and Radford, finite-dimensional involutive Hopf algebras are semisimple provided the ground field $K$ has characteristic 0 \cite{LR}.
The Hopf algebra~$H$ is semisimple if and only if it contains an element $i$ satisfying $hi=i\eps(h)$, $\forall h\in H$ and $\eps(i)=1$. If it exists, such an element is unique and called the Haar
integral. The Haar integral further satisfies $S(i)=i$ and $i\pp\ot i\p=i\p\ot i\pp$. So $i\in\Cocom H$.
If $H$ is semisimple, then so is $H^*$, $\D(H)$ and $\D(H)^*$ with respective Haar integrals $\iota$, $i_\D=\iota\ot i$ and $i_{\D^*}=i\ot\iota$.

After this preparation, we can define the algebra of the model as follows.
Let $\Sigma\,{::}\,\bra \Arr,T_0,T_2\ket$ be an OCPM and assume that $\Arr$ is finite. Let $H$ be a finite-dimensional Hopf algebra which is involutive.
We construct the algebra $\M=\M(\Sigma,H)$ as the tensor product algebra
\[
\M(\Sigma,H)=\bigotimes_{e\in\Sigma^1} \M_e,
\]
where the algebras $\M_e$ can be presented by a redundant set of generators
\[
P_a(h),\ Q_a(\varphi), \qquad \text{where}\quad h\in H,\ \varphi\in H^*,\ a\in e
\]
and by the relations
\begin{align}
\label{PQ lin}
&P_a(h),\ Q_a(\varphi)\ \text{are linear in }h \text{ and } \varphi, \\
\label{P alg}
&P_a(h)P_a(k)=P_a(hk),\qquad P_a(1)=\one_\M,\\
\label{Q alg}
&Q_a(\varphi)Q_a(\psi)=Q_a(\varphi\psi),\qquad Q_a(\eps)=\one_\M,\\
\label{QP comm rel}
&Q_a(\varphi)P_a(h)=P_a\bigl(\varphi\p\sweedl h\bigr)Q_a\bigl(\varphi\pp\bigr),
\end{align}
if $\{x_i\}$ and $\{\xi_i\}$ are dual bases, then
\begin{align}
\label{Pbar}
&P_{T_1a}(h)=P_a(x_jS(h)x_i)Q_a(S(\xi_i)\xi_j),\\
\label{Qbar}
&Q_{T_1a}(\varphi)=P_a(x_jS(x_i))Q_a(\xi_i S(\varphi)\xi_j)
\end{align}
for all $h,k\in H$, $\varphi,\psi\in H^*$ and $a\in e$.
The reader may observe that $P_a$ and $Q_a$ are nothing but the triangle operators $L_+$ and $T_+$ of \cite{Kitaev03}, although defined as abstract operators without representation on a (Hilbert) space.
(The $L_-$ and $T_-$ are the $L_+$ and $T_+$ of the opposite arrow~$T_1a$, as explained below.)

The first 4 relations tell us that for each $a\in\Arr$ the algebra generated by $P_a(H)$ and $Q_a(H^*)$ is the Heisenberg double of $H$, also
called the smash product algebra $H\pisharp H^*$.

The redundancy of the generators is made explicit by the last 2 relations so what needs explanation is that why the $P_{T_1a}$ and $Q_{T_1a}$ obey
the same relations as the $P_a$ and $Q_a$. This is closely related to involutivity of the antipode. But for the sake of explanation let $H$ be any finite-dimensional Hopf algebra. Recall \cite{Montgomery} that the smash product algebra is isomorphic to the full endomorphism algebra $\End H$. In fact, there are
two canonical isomorphisms
\begin{align}
\label{lambda}
&\mathcal{L}\colon \ H\ot H^*\to\End H,\qquad\mathcal{L}(h\ot\varphi)(x)=h(\varphi\sweedl x),\\
\label{rho}
&\mathcal{R}\colon \ H\ot H^*\to\End H,\qquad\mathcal{R}(h\ot\varphi)(x)=\bigl(x\sweedr S^{-1}(\varphi)\bigr) S(h)
\end{align}
related by $\mathcal{R}(\cdot)=S\ci\mathcal{L}(\cdot)\ci S^{-1}$. Introducing $p(h):=\mathcal{L}(h\ot\eps)$, $q(\varphi):=\mathcal{L}(1\ot\varphi)$ and
$\bar p(h):\mathcal{R}(h\ot\eps)$, $\bar q(\varphi):=\mathcal{R}(1\ot\varphi)$, we see that the operators $p$ and $q$ satisfy the same relations as $\bar p$ and $\bar q$, namely (\ref{P alg}), (\ref{Q alg}) and (\ref{QP comm rel}). It is clear that $\bar p$ and $\bar q$ can be expressed in terms of $p$ and $q$
and vice versa; we only have to compute
\begin{align}
\label{lambinv-rho}
&\mathcal{L}^{-1}\ci\mathcal{R}(h\ot\varphi)=x_jS(h)S^{-1}(x_i)\ot\xi_i S^{-1}(\varphi)\xi_j,\\
\label{rhoinv-lamb}
&\mathcal{R}^{-1}\ci\mathcal{L}(h\ot\varphi)=x_jS^{-1}(h)S^{-1}(x_i)\ot\xi_i S(\varphi)\xi_j,
\end{align}
which are identical expressions precisely if $S$ is involutive. Specializing them by substituting $\varphi=\eps$ or $h=1$, we obtain the required
expressions (\ref{Pbar}) and (\ref{Qbar}). At the end we did not have to choose whether $P_a$ and $Q_a$ were represented by $\mathcal{L}$ and $P_{T_1a}$
and $Q_{T_1a}$ by $\mathcal{R}$ or vice versa. That is to say, with the given generators, the presentation of $\M$ does not require any orientation of the edges:
$a$ and $T_1a$ are treated symmetrically as well as the $P$, $Q$ operators associated with them. In representations, however, such a choice has to be made.

Summarizing, we have a full matrix algebra $\M_e\cong H\pisharp H^*\cong\End H$ for each edge $e$ of $\Sigma$ and~$\M(\Sigma)$ is the tensor product
of these, so itself is a full matrix algebra. The long definition above should therefore be considered not just as the definition of the algebra $\M$ but
as the definition of $\M$ together with a distinguished set of generators.
This set of generators strongly influences what we shall ask and what we can answer in this model.

Some consequences of the defining relations of $\M$ are the following:
\begin{gather}
\label{PT_1 P}
P_{T_1a}(k)\quad \text{commutes with }\ P_a(h),\\
\label{QT_1 Q}
Q_{T_1a}(\psi)\quad \text{commutes with }\ Q_a(\varphi),\\
\label{QP inv-comm rel}
P_a(h)Q_a(\varphi) =Q_a\bigl(\varphi\sweedr S\bigl(h\pp\bigr)\bigr)P_a\bigl(h\p\bigr),\\
\label{QPbar comm rel}
Q_a(\varphi)P_{T_1a}(h) =P_{T_1a}\bigl(h\sweedr S\bigl(\varphi\pp\bigr)\bigr)Q_a\bigl(\varphi\p\bigr),\\
\label{QPbar inv-comm rel}
P_{T_1a}(h)Q_a(\varphi) =Q_a\bigl(h\p\sweedl\varphi\bigr)P_{T_1a}\bigl(h\pp\bigr)
\end{gather}
for all $h,k\in H$, $\varphi,\psi\in H^*$ and $a\in \Arr$.
(\ref{QP inv-comm rel}) is the inverse of (\ref{QP comm rel}).
(\ref{QPbar comm rel}) can be most easily obtained using the $\mathcal{L}$-$\mathcal{R}$ representations mentioned above.
(\ref{QPbar inv-comm rel}) is related to (\ref{QPbar comm rel}) as~(\ref{QP inv-comm rel}) does to (\ref{QP comm rel}).
Fortunately, we will never need the complicated relations~(\ref{Pbar}) and~(\ref{Qbar}); the~(\ref{PT_1 P}),~(\ref{QT_1 Q}) and~(\ref{QPbar inv-comm rel})
can be used as perfect substitutes for them.

The fact that the degrees of freedom
are attached to the arrows of $\Sigma$ suggests that we should treat the Kitaev model as a gauge theory. Since there are two fields $P_a$ and $Q_a$ at
each arrow satisfying Heisenberg double relations means something like having a quantized gauge theory with an involutive finite-dimensional Hopf algebra
$H$ playing the role of the gauge group. The field $Q_a$ can be interpreted as an exponential form of a vector potential and its conjugate momentum $P_a$ as an exponentiated electric field. In spite of this interpretation, we are not going to follow the standard strategy of gauge theory of declaring
the $H$-gauge invariant subalgebra as the algebra of observables. Following \cite{Kitaev03}, we shall treat to whole $\M$ as consisting of physically realizable operators.
There exist certain copies of the Drinfeld double $\D(H)$ within~$\M$ as subalgebras. Their induced action on $\M$ will be called ``gauge transformations'' with the warning that we are just borrowing a word from gauge theory and applying it to a new situation in which their role is yet to be explored.

In fact, there are plenty of ways to embed the double into $\M$ as a subalgebra.
\begin{lem} \label{lem: doubles in M}
Let $\mathbf{a}=(a_1,\dots,a_m)$ and $\mathbf{b}=(b_1,\dots, b_n)$ be sequences of arrows. Define
\begin{align}
\label{G}
&G_{a_1,\dots,a_m}(h):=P_{a_1}\bigl(h\p\bigr)\dots P_{a_m}(h_{(m)}),\qquad h\in H,\\
\label{F}
&F_{b_1,\dots, b_n}(\psi):=Q_{b_1}\bigl(\psi\p\bigr)\dots Q_{b_n}(\psi_{(n)}),\qquad\psi\in H^*.
\end{align}
Assume that
\begin{enumerate}\itemsep=0pt
\item[$({\rm i})$] neither the sequence $\Ou_1(a_1),\dots,\Ou_1(a_m)$ nor the sequence $\Ou_1(b_1),\dots,\Ou_1(b_n)$ contains repeated edges,
\item[$({\rm ii})$] $a_m=b_1$, $b_n=T_1a_1$,
\item[$({\rm iii})$] if $\Ou_1(a_i)=\Ou_1(b_j)$ then either $i=1$ and $j=n$ or $i=m$ and $j=1$.
\end{enumerate}
Then the map $\mathbf{D}(\psi\ot h):=F_{\mathbf{b}}(\psi)G_{\mathbf{a}}(h)$ is an algebra homomorphism $\D(H)\to \M$. That is to say,
\begin{gather*}
G_{\mathbf{a}}(g)G_{\mathbf{a}}(h) =G_{\mathbf{a}}(gh),\\
F_{\mathbf{b}}(\varphi)F_{\mathbf{b}}(\psi) =F_{\mathbf{b}}(\varphi\psi),\\
G_{\mathbf{a}}(h)F_{\mathbf{b}}(\psi) =F_{\mathbf{b}}\bigl(\psi\pp\bigr)G_{\mathbf{a}}\bigl(h\pp\bigr)\cdot \psi\ppp\bigl(h\p\bigr)\psi\p\bigl(S\bigl(h\ppp\bigr)\bigr)
\end{gather*}
holds for all $g,h\in H$ and for all $\varphi,\psi\in H^*$. If $H$ is semisimple, then $\mathbf{D}$ is injective.
\end{lem}
\begin{proof}
The first two relations are simple consequences of multiplicativity of the coproduct and of assumption (i).
The third can be proven by writing both $F$ and $G$ as a product of
3 terms (disregarding implicit linear combinations hidden in the coproducts), two side terms and one middle term where the middle term commutes with all
other $P$-s and $Q$-s due to assumption~(iii). The essence of the calculation is the exchange relations~(\ref{QP inv-comm rel}) and
(\ref{QPbar inv-comm rel}) applied to the side terms:
\begin{gather*}
G_{\mathbf{a}}(h)F_{\mathbf{b}}(\psi)= P_{a_1}\bigl(h\p\bigr)P_{\dots}\bigl(h\pp\bigr)P_{a_m}\bigl(h\ppp\bigr)\cdot Q_{b_1}\bigl(\psi\p\bigr)Q_{\dots}\bigl(\psi\pp\bigr)Q_{b_n}\bigl(\psi\ppp\bigr)\\ \hphantom{G_{\mathbf{a}}(h)F_{\mathbf{b}}(\psi)}{}
=P_{a_1}\bigl(h\p\bigr)P_{\dots}\bigl(h\pp\bigr)Q_{b_1}\bigl(\psi\p\sweedr S\bigl(h\pppp\bigr)\bigr)P_{a_m}\bigl(h\ppp\bigr)Q_{\dots}\bigl(\psi\pp\bigr)Q_{b_n}\bigl(\psi\ppp\bigr)\\ \hphantom{G_{\mathbf{a}}(h)F_{\mathbf{b}}(\psi)}{}
=Q_{b_1}\bigl(\psi\p\sweedr S\bigl(h\pppp\bigr)\bigr)Q_{\dots}\bigl(\psi\pp\bigr)P_{a_1}\bigl(h\p\bigr)Q_{b_n}\bigl(\psi\ppp\bigr)P_{\dots}\bigl(h\pp\bigr)P_{a_m}\bigl(h\ppp\bigr)\\ \hphantom{G_{\mathbf{a}}(h)F_{\mathbf{b}}(\psi)}{}
=Q_{b_1}\bigl(\psi\p\sweedr S\bigl(h\ppppp\bigr)\bigr) Q_{\dots}\bigl(\psi\pp\bigr)Q_{b_n}\bigl(h\p\sweedl\psi\ppp\bigr)P_{a_1}\bigl(h\pp\bigr)P_{\dots}\bigl(h\ppp\bigr)P_{a_m}\bigl(h\pppp\bigr)\!\!\\ \hphantom{G_{\mathbf{a}}(h)F_{\mathbf{b}}(\psi)}{}
=\psi\p\bigl(S\bigl(h\ppp\bigr)\bigr)\cdot F_{\mathbf{b}}\bigl(\psi\pp\bigr) G_{\mathbf{a}}\bigl(h\pp\bigr)\cdot \psi\ppp\bigl(h\p\bigr).
\end{gather*}
It remains to show injectivity of $\mathbf{D}$. Since the image of $\mathbf{D}$ lies in the tensor product of two copies of $H\pisharp H^*$ at the edges
$e_1:=\Ou_1(a_1)$, $e_2:=\Ou_1(a_m)$ and copies of $H$ or $H^*$ at the intermediate edges $\Ou_1(a_2),\dots,\Ou_1(a_{m-1})$ and
$\Ou_1(b_2),\dots,\Ou_1(b_{n-1})$, respectively, if we compose $\mathbf{D}$ with the counit of $H$, resp.\ $H^*$ at the intermediate places we obtain the
$m=n=2$ version of $\mathbf{D}$. So injectivity will follow from that of the $m=n=2$ case. Further composing with representation~$\mathcal{L}$ at both
$e_1$ and $e_2$, the map becomes a representation on $H\ot H$ given by
\[
\mathbf{D}'(\psi\ot h)\cdot(g\ot k)= \bigl(h\p g\bigr)\sweedr S\bigl(\psi\pp\bigr)\ot\psi\p\sweedl\bigl(h\pp k\bigr).
\]
Let $i\in H$ be the Haar integral of $H$. We claim that $i\ot 1$ is a cyclic vector of this representation.
As a matter of fact, if $\mathbf{D}'(\psi\ot h)\cdot (i\ot 1)=i\sweedr S\bigl(\psi\pp\bigr)\ot\psi\p\sweedl h$ is zero then so is
$\lambda\ot v:=S\bigl(\psi\pp\bigr)\ot\psi\p\sweedl h$, and therefore
\[
0=S\bigl(\lambda\p\bigr)\ot \lambda\pp\sweedl v=\psi\ppp\ot S\bigl(\psi\pp\bigr)\sweedl\bigl(\psi\p\sweedl h\bigr)=\psi\ot h.
\]
Hence, the kernel of $\mathbf{D}'$, containing the kernel of $\mathbf{D}$, is zero.
\end{proof}

Note that the sequences of arrows on which (\ref{F}) and (\ref{G}) are defined are arbitrary sets of arrows with an ordering, they do not have to respect
any neighbourhood relations on the surface complex $\Sigma$. Of course, $\M$ is just the tensor product of edge algebras, so knows nothing about how
the edges organize into a complex.

By some vague analogy with gauge theory, we nevertheless consider
those sequences ${\mathbf{a}}$, ${\mathbf{b}}$ which are closed curves on $\Sigma$: a Wilson loop of $Q$ operators and a dual Wison loop of $P$
operators. (Although it works for the minimal loops, taking larger ones we have to realize that these are not going to define a holonomy!)
The easiest way to construct such loops is to take the $T_2$-orbits or the $T_0$-orbits of a fixed arrow. For $a\in\Arr$, we define
\begin{align}
\label{Gauss}
&G_a(h):=P_{T_0a}\bigl(h\p\bigr)P_{T_0^2a}\bigl(h\pp\bigr)\cdots P_{T_0^ma}\bigl(h_{(m)}\bigr),\qquad m=|\Ou_0(a)|,\\
\label{flux}
&F_a(\varphi):=Q_a\bigl(\varphi\p\bigr)Q_{T_2 a}\bigl(\varphi\pp\bigr)\cdots Q_{T_2^{n-1}a}\bigl(\varphi_{(n)}\bigr),\qquad n=|\Ou_2(a)|.
\end{align}
$G_a$ is a sort of divergence of the electric field at the vertex $\Ou_0(a)$ so we call it the Gauss' law operator.
$F_a$ is sort of curl of the vector potential at the face $\Ou_2(a)$ so we call it the magnetic flux operator.
By Lemma~\ref{lem: doubles in M}, the operators
\[
\mathbf{D}_a(\varphi\ot h):=F_a(\varphi)G_a(h),\qquad \varphi\ot h\in\D(H)
\]
provide homomorphic images of the double in $\M$ for each $a\in\Arr$, and therefore
\begin{align*}
(\psi\ot k)\la{a} M :={}& \mathbf{D}_a\bigl((\psi\ot k)\p\bigr)\,M\,\mathbf{D}_a\bigl(S_\D\bigl((\psi\ot k)\pp\bigr)\bigr)\\ ={}& F_a\bigl(\psi\pp\bigr)G_a\bigl(k\p\bigr)\,M\,G_a\bigl(S\bigl(k\pp\bigr)\bigr)F_a\bigl(S\bigl(\psi\p\bigr)\bigr),\\
M\ra{a}(\psi\ot k) :={} & \mathbf{D}_a\bigl(S_\D\bigl((\psi\ot k)\p\bigr)\bigr)\,M\,\mathbf{D}_a\bigl((\psi\ot k)\pp\bigr)\\ ={}& G_a\bigl(S\bigl(k\p\bigr)\bigr)F_a\bigl(S\bigl(\psi\pp\bigr)\bigr)\,M\,F_a\bigl(\psi\p\bigr)G_a\bigl(k\pp\bigr)
\end{align*}
define a left $\D(H)$-module algebra and a right $\D(H)$-module algebra structure on $\M$ for each $a\in\Arr$. For different arrows, however, these actions
do not necessarily commute. The~$F_a$~loop operators are supported on the oriented boundary of the face $\Ou_2(a)$, so $F_a(\psi)$ and $F_b(\varphi)$
commute whenever $\Ou_2(a)\neq \Ou_2(b)$ because of (\ref{QT_1 Q}). Similarly, the $G_a$-s are supported on the oriented coboundary (a star) of the vertex
$\Ou_0(a)$, so $G_a(k)$ and $G_b(h)$ commute whenever $\Ou_0(a)\neq \Ou_0(b)$ because of (\ref{PT_1 P}). It follows that if $\Ou_0(a)\neq \Ou_0(b)$ and
$\Ou_2(a)\neq \Ou_2(b)$ then the images of~$\mathbf{D}_a$ and~$\mathbf{D}_b$ commute and so do the actions $\triangleright_a$ or $\triangleleft_a$ with
$\triangleright_b$ or $\triangleleft_b$. Thus, local $\D(H)$ gauge symmetry cannot be thought as a fixed system of pairwise commuting $\D(H)$-actions.
Although there are maximal systems of such actions, they are not unique and may not involve all~$F$ and~$G$ operators, especially if
$\big|\Sigma^0\big|\neq \big|\Sigma^2\big|$.

There is also another pair of left and right actions,
\begin{align*}
(\psi\ot k)\opla{a} M :={}& \mathbf{D}_a\bigl((\psi\ot k)\pp\bigr)\,M\,\mathbf{D}_a\bigl(S_\D\bigl((\psi\ot k)\p\bigr)\bigr)\\ ={}& F_a\bigl(\psi\p\bigr)G_a\bigl(k\pp\bigr)\,M\,G_a\bigl(S\bigl(k\p\bigr)\bigr)F_a\bigl(S\bigl(\psi\pp\bigr)\bigr),\\
M\opra{a}(\psi\ot k) :={}& \mathbf{D}_a\bigl(S_\D\bigl((\psi\ot k)\pp\bigr)\bigr)\,M\,\mathbf{D}_a\bigl((\psi\ot k)\p\bigr)\\ ={}& G_a\bigl(S\bigl(k\pp\bigr)\bigr)F_a\bigl(S\bigl(\psi\p\bigr)\bigr)\,M\,F_a\bigl(\psi\pp\bigr)G_a\bigl(k\p\bigr)
\end{align*}
with respect to which, however, not $\M$ but $\M^\op$ is a module algebra.

It is worth investigating the subspaces $P_a(H)$ and $Q_a(H^*)$ because they are the smallest $D(H)$-invariant subspaces, in fact module subalgebras,
of $\M$ or $\M^\op$. The exchange relations
\begin{align}
&G_a(k)Q_a(\varphi)=Q_a\bigl(\varphi\sweedr S\bigl(k\pp\bigr)\bigr)G_a\bigl(k\p\bigr), \label{Q G}\\
&F_a(\psi)Q_a(\varphi)=Q_a\bigl(\psi\p\varphi S\bigl(\psi\pp\bigr)\bigr)F_a\bigl(\psi\ppp\bigr), \\
%\label{GT_2 Q}
&Q_a(\varphi)G_{T_2a}(k)=G_{T_2a}\bigl(k\pp\bigr)Q_a\bigl(S\bigl(k\p\bigr)\sweedl\varphi\bigr), \\
&Q_a(\varphi)F_{T_2a}(\psi)=F_{T_2a}\bigl(\psi\p\bigr)Q_a\bigl(S\bigl(\psi\pp\bigr)\varphi\psi\ppp\bigr)
\end{align}
follow easily from the defining relations of $\M$, as well as
\begin{gather}
%\label{F PT_0}
F_{T_0^{-1}a}(\psi)P_a(h) =P_a\bigl(h\sweedr S\bigl(\psi\pp\bigr)\bigr)F_{T_0^{-1}a}\bigl(\psi\p\bigr),\\
G_{T_0^{-1}a}(k)P_a(h) =P_a\bigl(k\p h S\bigl(k\pp\bigr)\bigr)G_{T_0^{-1}a}\bigl(k\ppp\bigr),\\
P_a(h)F_a(\psi) =F_a\bigl(\psi\pp\bigr)P_a\bigl(S\bigl(\psi\p\bigr)\sweedl h\bigr),\label{P F}\\
P_a(h)G_a(k) =G_a\bigl(k\p\bigr)P_a\bigl(S\bigl(k\pp\bigr)hk\ppp\bigr).
\end{gather}
From these relations, one infers that $P_a(H)$ is an invariant subspace under $\triangleright_{T_0^{-1}a}$ and $\triangleleft_a$ while $Q_a(H^*)$ is an invariant subspace under $\blacktriangleright_a$ and $\blacktriangleleft_{T_2a}$. So we have
\begin{alignat*}{4}
& \text{left $\D(H)$-module subalgebras}\qquad&& P_a(H)\subset\bigl\bra\M,\la{T_0^{-1}a}\,\bigr\ket,\qquad && Q_a(H^*)^\op\subset\bigl\bra\M^\op,\opla{a}\,\bigr\ket,&\\
& \text{right $\D(H)$-module subalgebras}\qquad && P_a(H)\subset\bigl\bra\M,\ra{a}\,\bigr\ket,\qquad &&Q_a(H^*)^\op\subset\bigl\bra\M^\op,\opra{T_2a}\,\bigr\ket.&
\end{alignat*}
Despite of having a left and a right action neither $P_a(H)$ nor $Q_a(H^*)$ is a bimodule over $\D(H)$, the former because
$\Ou_0\bigl(T_0^{-1}a\bigr)=\Ou_0(a)$ and the latter because $\Ou_2(T_2a)=\Ou_2(a)$.
Denoting these modules by drawing an arrow from the left action to the right action,
\begin{equation}\label{P-arrow Q-arrow}
\mathbf{D}_{T_0^{-1}a}\longrarr{P_a(H)}\mathbf{D}_a,\qquad \mathbf{D}_a\longrarr{Q_a(H^*)}\mathbf{D}_{T_2a},
\end{equation}
we see that from the point of view of gauge transformations the $P_a$ and the $Q_a$ seem to belong to different edges and the vertices associated
to individual $\D(H)$ actions are not the vertices of~$\Sigma$, not its faces either, but are in bijection with the arrows themselves. Such a complex is
the dual of the double of $\Sigma$ and will be discussed, among others, in the next two sections.

The Hamiltonian of the model, however, can be introduced without moving to the dual of the double. If we assume that $H$ is semisimple
then $H$ has a Haar integral $i$. The Haar integral being a cocommutative element its
iterated coproducts $i\p\ot\cdots \ot i_{(n)}$ are invariant under cyclic permutations. Therefore, $G_a(i)$ does not depend on the starting
point of the dual loop, i.e., $G_{T_0a}(i)=G_a(i)$. So we may rename $G_a(i)$ as $A_v$, where $v=\Ou_0(a)$. Similarly, $H^*$ also has a Haar
integral $\iota$ and $F_{T_2a}(\iota)=F_a(\iota)$. Renaming $F_a(\iota)$ as $B_f$, where $f=\Ou_2(a)$, we have constructed a system of pairwise
commuting projections $A_v$, for $v\in\Sigma^0$, and $B_f$, for $f\in\Sigma^2$. The Hamiltonian is the sum
\begin{equation}\label{hamiltonian}
\mathbf{H}_{\Sigma,H}=\sum_{v\in\Sigma^0}(\one-A_v) + \sum_{f\in\Sigma^2}(\one- B_f).
\end{equation}
There is a remarkable duality of these models. Replacing $H$ with $H^*$ and $\Sigma$ with $\Sigma^*$, there is an algebra isomorphism
$\M(\Sigma,H)\iso\M(H^*,\Sigma^*)$ sending the Hamiltonian (\ref{hamiltonian}) to $\mathbf{H}_{\Sigma^*,H^*}$.
This (non-canonical) isomorphism can be lifted to the ribbon operators, see the duality formula~(\ref{duality formula}).
The general structure of this ``duality'' is still to be investigated. In this paper, we will use it only as a technical tool to prove one half of Theorem~\ref{thm: ribb bim}.

\section{Duals, mirror images and doubles in the arrow presentation}\label{sec: dudo}

Let $\Sigma\,{::}\,\bra \Arr,T_0,T_2\ket$ be an OCPM. The dual of $\Sigma$ is easy to formulate in the $\Sigma$-language, see~(\ref{du Sigma}).
In the arrow presentation, however, we work with directed edges and there are several natural and less natural ways to assign to an arrow of $\Sigma$ an
arrow of $\Sigma^*$. Such problems forces us to speak about, at least, the isomorphisms of arrow presentations.
(More general morphisms of OCPMs or of their presentations would require to extend the theory to 3-complexes which we cannot afford here.)
\begin{defi}
An isomorphism $\bra A,S_0,S_2\ket\to \bra B,T_0,T_2\ket$ of arrow presentations is a bijection $f\colon A\to B$ such that $f\ci S_i=T_i\ci f$ for $i=0,2$.
\end{defi}
In order to see how this definition translates to the polyhedral maps, we need
\begin{defi}
An isomorphism $\phi$ of CPMs $\bra\Sigma,\dim,\Bd,\Cb\ket\rarr{\phi}\bra\Sigma',\dim',\Bd',\Cb'\ket$ is a~bijection $\phi\colon\Sigma\to\Sigma'$
which preserves the dimensions of cells, $\dim'\ci\phi=\dim$, and such that $\Bd(\phi(x))=\{\phi(y)\mid y\in\Bd(x)\}$.
If $\phi\colon \Sigma\to\Sigma'$ is an isomorphism of CPMs between oriented closed polyhedral maps, then $\phi$ is called
orientation preserving if for all $v\in\Sigma^0$ and $f\in\Sigma^2$ the restrictions of $\phi$ to the neighbourhoods $\Nb(v)$ and $\Nb(f)$
preserve the cyclic orders given on them by the orientations. Such an orientation preserving isomorphism is called briefly an isomorphism of OCPMs.
\end{defi}
\begin{lem}
Let $\Sigma\,{::}\,\bra A,T_0,T_2\ket$, $\Sigma'\,{::}\,\bra A',T'_0,T'_2\ket$ and let $f\colon \bra A,T_0,T_2\ket\to\bra A',T'_0,T'_2\ket$ be an isomorphism of
arrow presentations. Then $\phi(\Ou_i(a)):=\Ou'_i(f(a))$ is an isomorphism of OCPMs. This proves, at least for isomorphisms, functoriality of the construction of OCPMs from
arrow presentations described by Theorem~$\ref{thm: CPM}$.
\end{lem}
\begin{proof}
The isomorphism $\phi$ is the restriction to the orbits of the direct image $f_*\colon 2^A\to 2^{A'}$ which is the natural lift of $f$ to the power sets. So
\[
\phi(\Ou_i(a))=\{f\ci T_i^n(a)\mid n\in\NN\}=\bigl\{{T'}^n_i(f(a))\mid n\in\NN\bigr\}=\Ou'_i(f(a))
\]
holds a priori for $i=0,2$ but then also for $i=1$, i.e., $\phi$ maps $T_i$-orbits to $T'_i$-orbits for all~$i$.
Therefore, it preserves the dimension of cells and satisfies $b\in\Ou_i(a)$ if and only if $f(b)\in \phi(\Ou_i(f(a)))$. If follows that
$\Ou_{i-1}(b)\cap\Ou_i(a)\neq\varnothing$ if and only if $\Ou'_{i-1}(f(b))\cap\Ou'_i(f(a))\neq\varnothing$, so it preserves the boundary and coboundary operations
by (\ref{Bd cap Cb}), i.e., $\Bd(\phi(s))=\{\phi(t)\mid t\in\Bd(s)\}$ for all $s\in\Sigma$.
Finally, since $f$ commutes with the $T_i$ actions, it is obvious that $\phi$ preserves the orientation.
\end{proof}

After this preparation, we can discuss duals of arrow presentations. They are characterized by producing dual polyhedral maps in the sense of
(\ref{du Sigma}). Let $P=\bra \Arr,T_0,T_2\ket$ be a presentation of an OCPM $\Sigma$.
A dual \smash{$\duP$} of $P$ can be obtained by following the schema
\smash{$\begin{picture}(14,12){\color{blue}\put(0,3){\vector(1,0){14}}}{\color{red}\put(6,-3){\vector(0,1){14}}}\end{picture}$}
 which means that the blue arrow $a$ of $\Sigma$ is
interpreted as a red arrow of $\Sigma^*$ directed to point from the right face of~$a$ to the left face of~$a$. The $T_2$ of the dual, denoted
$\tilde{T}_2$, is rotation around the left face of the red arrow, i.e., around the source vertex of the blue one which is $T_0$. The $\tilde{T}_0$, however,
is obtained by rotating around the source vertex of the red arrow, i.e., around the right face of the blue one which is not $T_2$ applied to the blue
arrow. In its stead $\tilde{T}_0=T_1T_2T_1=T_0T_2T_0^{-1}$. Thus, \begin{equation}\label{AP du}
\duP:=\bigl\bra\Arr,T_0T_2T_0^{-1},T_0\bigr\ket
\end{equation}
is a presentation producing the dual polyhedral map $\Sigma^*$. If we follow the other natural schema \smash{$\begin{picture}(14,12){\color{blue}\put(0,3){\vector(1,0){14}}}{\color{red}\put(6,9){\vector(0,-1){14}}}\end{picture}$},
then we get another presentation
\begin{equation}\label{iduP}
\iduP:=\bigl\bra\Arr,T_2,T_2^{-1}T_0T_2\bigr\ket,
\end{equation}
which should produce the same (or isomorphic) dual map $\Sigma^*$. Encouraged by these examples, we may experiment with a simpler presentation
of the same sort. Let
\[
\xP:=\bra\Arr,T_2,T_0\ket.
\]
\begin{lem}
The bijections $T_i\colon \Arr\to\Arr$ provide isomorphisms of arrow presentations
\begin{gather*}
\xP\rarr{T_0}\duP\rarr{T_1}\iduP\rarr{T_2}\xP,\\
\duiduP\rarr{\id}P\rarr{\id}\iduduP,\\
\duduP\rarr{T_1}P\rarr{T_1}\iduiduP.
\end{gather*}
\end{lem}
\begin{proof}
Let us prove the first
\begin{align*}
&\duT_0T_0=T_0T_2T_0^{-1}T_0=T_0T_2=T_0\xT_0,\\
&\duT_2T_0=T_0T_0=T_0\xT_2.
\end{align*}
The remaining ones are similarly simple.
\end{proof}

Although all the three duals \smash{$\duP$}, \smash{$\iduP$} and $\xP$ are isomorphic, the first two differs from the third in that
\[
\duT_0\duT_2=\iduT_0\iduT_2=T_0T_2 \qquad \text{while} \quad \xT_0\xT_2=T_2T_0.
\]
In other words, the arrow-opposite arrow pairs, i.e., the $T_1$-orbits, of~\smash{$\duP$} and~\smash{$\iduP$} are the same as those of $P$
while in $\xP$ they are connected by the transformation $T_2T_0$ which is non-local in~$P$, it sends an arrow by distance 2 away.
Of course, $T_2T_0$ is also an
involution, $T_2T_0T_2T_0=T_2T_1T_0=T_0^{-1}T_0=\id$, algebraically as good as~$T_1$, so why is $T_1$ local and the other is not? We have made a choice,
not in Definition~\ref{def: AP} but in the construction of $\Sigma$ when we called edges the orbits of $T_1=T_0T_2$ and not of $T_2T_0$. The form of~$\xP$ shows that if we had chosen edges to be the $T_2T_0$-orbits the interpretation of~$T_0$,~$T_2$ would have changed to~$T_0$ rotating around faces and~$T_2$ rotating around vertices. (Speaking about non-locality of a presentation~$Q$ is meaningful only with respect to
 another isomorphic presentation~$P$:
$Q$ is non-local if the isomorphism $P\to Q$ is non-local, i.e., tears apart the arrows of an edge. Of course, both~$T_0$ and~$T_2$ are non-local
and~$T_1$ is local in this sense.)

When we draw a portion of a polyhedral map we look at it from a specific side of 3-space determined by the orientation. If we look at it from the other
side what we see is the mirror image. In more precise terms, if $P$ is a presentation of an OCPM $\Sigma$, then the mirror image should be the presentation
of the same underlying CPM but equipped with opposite orientation. It is easy to see that the mirror image of $P=\bra \Arr,T_0,T_2\ket$ can be presented by
\begin{equation}\label{AP mir}
P^\mir:=\bigl\bra \Arr,T_0^{-1},T_1T_2^{-1}T_1\bigr\ket = \bigl\bra \Arr,T_0^{-1},T_0T_2^{-1}T_0^{-1}\bigr\ket.
\end{equation}
Another, simpler but non-local, presentation is $P^\mathbf{M}:=\bigl\bra \Arr,T_0^{-1},T_2^{-1}\bigr\ket$ with isomorphism $T_0\colon P^\mathbf{M}\rarr{}P^\mir$.

In order to get an arrow presentation for the double $\D(\Sigma)$, we first have to enumerate the arrows of the double in some way.
From Definition~\ref{def: double}, one infers that $\D(\Sigma)^1$ is 4 times bigger than $\Sigma^1$ and the same should hold for the arrows.
Let $e=\Ou_1(a)\in\Sigma^1$ be considered as a vertex of the double. There are 4 edges joining at $e$, 2 ve-type edges and 2 ef-type edges.
We denote the arrow $\Ou_0(a)\to e$ by $a^+_0$ and the arrow $\Ou_2(a)\to e$ by $a^+_2$. Their opposites are called~$a^-_0$ and~$a^-_2$, respectively.
The arrows of the remaining two edges joining to $e$ can then be expressed as $(T_1a)^\pm_0$ and $(T_1a)^\pm_2$. The identities
\begin{align*}
&\Ou_0(a)\cap\Ou_1(T_0a)=\Ou_0(T_0a)\cap\Ou_1(T_0a)=\{T_0a\},\\
&\Ou_2(a)\cap\Ou_1\bigl(T_2^{-1}a\bigr)=\Ou_2\bigl(T_2^{-1}a\bigr)\cap\Ou_1\bigl(T_2^{-1}a\bigr)=\bigl\{T_2^{-1}a\bigr\}
\end{align*}
show that the arrows and vertices of every face neighbourhood of the double can be labelled by
$$
\parbox{200pt}{
\begin{picture}(200,120)(-100,-60)
\put(-50,0){\color{blue}{\circle*{3}}} \put(-75,-4){$\sst\Ou_0(a)$}
\put(50,0){\color{red}{\circle*{3}}} \put(55,-4){$\sst\Ou_2(a)$}
\put(0,-50){\color{melegzold}{\circle*{3}}} \put(-7,-60){$\sst\Ou_1(a)$}
\put(0,50){\color{melegzold}{\circle*{3}}} \put(-33,55){$\sst\Ou_1(T_0a)=\Ou_1\left(T^{-1}_2a\right)$}
\put(-45,5){\vector(1,1){40}} \put(-26,18){$\sst(T_0a)_0^+$}
\put(45,5){\vector(-1,1){40}} \put(30,22){$\sst\left(T_2^{-1}a\right)_2^+$}
\put(-45,-5){\vector(1,-1){40}} \put(-24,-24){$\sst a_0^+$}
\put(45,-5){\vector(-1,-1){40}} \put(29,-29){$\sst a_2^+$}
\end{picture}
}
$$
Thus, the arrows of $\D(\Sigma)$ is the set $\mathbb{A}:=\{a_i^\sigma\mid a\in\Arr,\, i\in\{0,2\},\,\sigma\in\{+.-\}\}$ and the arrow presentation
\[
\D(\Sigma)\,{::}\,\D(P)=\bra\mathbb{A},\mathbb{T}_0,\mathbb{T}_2\ket
\]
is given by the following permutations:
\begin{gather}\label{double T}
\mathbb{T}_0 a_i^\sigma:=\begin{cases}
(T_0a)_0^+&\text{if}\ i=0,\,\sigma=+,\\
(T_1a)_2^-&\text{if}\ i=0,\,\sigma=-,\\
(T_2a)_2^+&\text{if}\ i=2,\,\sigma=+,\\
a_0^-&\text{if}\ i=2,\,\sigma=-,
\end{cases} \qquad
\mathbb{T}_2 a_i^\sigma:=\begin{cases}
a_2^-&\text{if}\ i=0,\,\sigma=+,\\
(T_0^{-1}a)_0^+&\text{if}\ i=0,\,\sigma=-,\\
(T_1a)_0^-&\text{if}\ i=2,\,\sigma=+,\\
(T_2^{-1}a)_2^+&\text{if}\ i=2,\,\sigma=-.
\end{cases}
\end{gather}
Defining also $\mathbb{T}_1=\mathbb{T}_0\mathbb{T}_2$, we obtain
\begin{equation*}
\mathbb{T}_1a_i^\sigma=a_i^{-\sigma},\qquad \mathbb{T}_2^4=\id_\mathbb{A},\qquad \mathbb{T}_0^4a_i^-=a_i^-
\end{equation*}
for all $a$, $i$ and $\sigma$. The second expresses the fact that all faces of the double are quadrangles and the third holds
(only for $\sigma=-$) because the $e$-type vertices of $\D(\Sigma)$ all have degree 4.

Roughly speaking, the $T$-operators of the double $\D(\Sigma)$ are matrix amplifications of the $T$-operations of $\Sigma$,
\[
\mathbb{T}_0=\left(\begin{matrix}T_0&\ &\ &\ \\ \ &\ &\ & 1\\ \ &\ &T_2&\ \\ \ &T_1&\ &\ \end{matrix}\right),\qquad
\mathbb{T}_2=\left(\begin{matrix}\ &T_0^{-1}&&\\&&T_1&\\&&&T_2^{-1}\\1&&&\end{matrix}\right),
\]
where the order of rows and columns is that of formulas (\ref{double T}) and $1$ stands for $\id_\Arr$.
But these are not matrices in the usual sense: The addition involved in matrix multiplications is not defined at all.
(We are working in the category of sets after all.)
This is why we have left the nonvalid entries empty rather than put a zero there.

Having been constructed arrow presentations for $\Sigma^*$ and $\D(\Sigma)$ of every OCPM $\Sigma$ has the secondary role of extending the Definition
\ref{def: double} by giving them orientations.

Now it is easy to obtain an arrow presentation for the dual of the double of $\Sigma$. We take the above presentation of $\D(\Sigma)$ and then apply the
construction \smash{$\duP$}. This leads to the following presentation for $\D(\Sigma)^*$. The set of arrows is $\{a^\sigma_i\mid i=0,e,\,\sigma=\pm\}$,
the same as that of the double. \big(We spare the work of writing $\tilde{a_i^\sigma}$ to distinguish them from the arrows of $\D(\Sigma)$, eventually it is the same set.\big)
These arrows are shown on Figure~\ref{darrows}.
\begin{figure}\centering
\parbox{100pt}{
\begin{picture}(100,100)(-50,-50)
\put(0,0){\circle*{4}}
\put(0,0){\vector(1,0){50}} \put(30,-10){$\sst a_2^+$}
\put(0,-50){\vector(0,1){48}} \put(-12,-35){$\sst a_0^+$}
\put(0,0){\vector(-1,0){50}} \put(-40,5){$\sst(T_0a)_0^+$}
\put(0,50){\vector(0,-1){48}} \put(3,36){$\sst(T_2^{-1}a)_2^+$}
\put(-25,-25){\color{blue}{\circle*{4}}}
\put(25,25){\color{red}{\circle*{4}}}
\put(-25,25){\color{melegzold}{\circle*{4}}}
\put(25,-25){\color{melegzold}{\circle*{4}}}
\end{picture}
}
\hskip 1cm
\parbox{100pt}{
\begin{picture}(100,100)(-50,-50)
\put(-25,-25){\vector(0,1){48}} \put(-37,7){$\sst a_0^+$}
\put(-25,25){\vector(1,0){48}} \put(2,30){$\sst a_2^+$}
\put(25,25){\vector(0,-1){48}} \put(27,-12){$\sst (T_1a)_0^+$}
\put(25,-25){\vector(-1,0){48}} \put(-12,-33){$\sst (T_1a)_2^+$}
\put(-25,-25){\circle*{4}}
\put(-25,25){\circle*{4}}
\put(25,25){\circle*{4}}
\put(25,-25){\circle*{4}}
\put(-50,0){\color{blue}{\circle*{4}}}
\put(0,50){\color{red}{\circle*{4}}}
\put(50,0){\color{blue}{\circle*{4}}}
\put(0,-50){\color{red}{\circle*{4}}}
\put(0,0){\color{melegzold}{\circle*{4}}}
\end{picture}
}
\caption{The vertex neighbourhood of the site $\bra\Ou_0(a),\Ou_2(a)\ket$ and the face neighbourhood of the type~1 face $\Ou_1(a)$.}
\label{darrows}
\end{figure}
Thus, we have the presentation
\begin{equation}\label{AP dudo}
\D(\Sigma)^*\,{::}\,\D(P)^\sim=\bigl\bra\mathbb{A},\tilde{\mathbb{T}}_0,\tilde{\mathbb{T}}_2\bigr\ket
\end{equation}
with
\[
\tilde{\mathbb{T}}_0=\left(\begin{matrix}\ &&&T_1\\T_0^{-1}&&&\\&1&&\\&&T_2^{-1}&\end{matrix}\right),\qquad
\tilde{\mathbb{T}}_2=\left(\begin{matrix}T_0&\ &\ &\ \\ \ &\ &\ & 1\\ \ &\ &T_2&\ \\ \ &T_1&\ &\ \end{matrix}\right).
\]

The set $\D(\Sigma)^{*2}$ of faces of the dual of the double is the same set as the set $\Sigma$ of all cells of the original OCPM. Changing between
these two interpretations of the same set will be a recurring tool in this paper. Since the role of these faces depend on their dimensions in $\Sigma$
and we want to avoid saying ``vertex face'' or ``face face'' we shall use the term ``type $d$ face'' for an element $c\in\Sigma^d\subset \D(\Sigma)^{*2}$.

\begin{pro}
There exist natural isomorphisms of arrow presentations
\begin{alignat*}{5}
& \delta_P\colon \ && \D(P)^\sim \to \D(P^\sim)^\sim,\qquad&& a_0^\sigma\mapsto a_2^\sigma,\qquad && a_2^\sigma\mapsto(T_1a)_0^\sigma,&\\
&\mu_P\colon \ && \D(P^\mir)^\sim \to(\D(P)^\sim)^\mir,\qquad && a_0^\sigma\mapsto a_0^{-\sigma},\qquad && a_2^\sigma\mapsto(T_1a)_2^{-\sigma}&
\end{alignat*}
or in matrix form
\[
\delta_P=\left(\begin{matrix}\ &&T_1&\\&&&T_1\\1&&&\\&1&&\end{matrix}\right),
\qquad \mu_P=\left(\begin{matrix}\ &1&&\\1&&&\\&&&T_1\\&&T_1&\end{matrix}\right).
\]
\end{pro}
\begin{proof}
Collect the definitions (\ref{AP du}), (\ref{AP dudo}) and (\ref{AP mir}) and perform the matrix multiplications.
\end{proof}

\section{Curves, dual curves and ribbons} \label{sec: curves}

Curves on surfaces do not require orientation of the surface. Nevertheless, we assume that $\Sigma$ is an OCPM equipped with arrow presentation
$P=\bra\Arr(\Sigma),T_0,T_2\ket$.
Both curves and opposite curves (see opcurves below) on $\Sigma$ formulate the idea of what is called a walk on a directed graph.
By directed graph, we mean a quiver
\[
A\pair{\text{source}}{\text{target}}V
\]
 consisting of two sets, two functions and no axioms.
Although the CPM $\bra\Sigma,\dim,\Bd,\Cb\ket$ has an underlying graph, it is undirected. It consists of the sets $\Sigma^1$ and $\Sigma^0$ and
of the function $\Bd$ which associates to every edge $e\in\Sigma^1$ the 2-element subset $\Bd(e)\subseteq \Sigma^0$. Since this undirected graph has
no loops, there is a natural way to make it a quiver by doubling the set of edges. Making use of the arrow presentation~$P$, we can take the set
$\Arr(\Sigma)$ of arrows and define the quiver
\[
\Qv(P):=\biggl(\Arr(\Sigma)\shortpair{\del_0}{\del_1}\Sigma^0\biggr)\qquad\text{where}\quad \del_0 a:=\Ou_0(a),\quad \del_1 a:=\Ou_0(T_1a).
\]
This quiver is involutive in the sense of having an involution $T_1$ such that $\del_0\ci T_1=\del_1$; this way reminding us that it stems from
an undirected graph.

The difference between curves and opcurves is merely conventional; the literature on Kitaev models seems to use both curves \cite{BMCA,Meusburger}
and opcurves \cite{Bombin-Delgado}.
In this paper, most of the calculations will be done with opcurves. The reason is that opcurves fit better with the conventions we have chosen in
Definition~\ref{def: AP}.
\begin{defi}\quad
\begin{enumerate}\itemsep=0pt
\item[$({\rm a})$] An ``opcurve on $\Sigma$'', more precisely a $P$-opcurve, is a sequence $\gamma=(a_1,\dots,a_n)$ of arrows such that $\del_1 a_i=\del_0 a_{i+1}$ for $0<i<n$.
If $\del_0 a_1=v_0$ and $\del_1 a_n=v_n$, then we say that $\gamma$ is an opcurve from $v_0$ to $v_n$ and write $\gamma\colon \ v_0\to v_n$.
The opcurves on $\Sigma$ form a category with composition of opcurves \smash{$u\rarr{\alpha}v\rarr{\beta}w$} being the concatenation $\alpha\beta\colon \ u\to w$.

\item[$({\rm b})$] A ``curve on $\Sigma$'', i.e., a $P$-curve, is a sequence $\gamma=(a_1,\dots,a_n)$ of arrows such that the reverse sequence $\gamma^\rev=(a_n,\dots,a_1)$ is an opcurve,
that is to say, such that $\del_0 a_i=\del_1 a_{i+1}$ for $0<i<n$.
This curve has source $\del_0 a_n$ and target $\del_1 a_1$.
The curves on $\Sigma$ form a~category with composition of curves \smash{$u\rarr{\alpha}v\rarr{\beta}w$} being $\beta\ci\alpha=\alpha \beta\colon \ u\to w$.
So this category is just the opposite of the category of opcurves.

\item[$({\rm c})$] If an opcurve $\gamma=(a_1,\dots,a_n)$ is open, i.e., $\del_0a_1\neq\del_1 a_n$, then it is called simple if the sequence $(\del_0 a_1,\del_1 a_1,\del_1 a_2,\dots
\del_1 a_n)$ of visited vertices contains no repetitions. If~$\gamma$~is closed, i.e., $\del_0a_1=\del_1 a_n$, then it is called simple if $(\del_1 a_1,\del_1 a_2,\dots
\del_1 a_n)$ contains no repeated vertices. A curve $\gamma$ is called simple if the opcurve $\gamma^\rev$ is simple.
\end{enumerate}
\end{defi}

It is clear that passing to the reverse sequence provides a bijection between curves and opcurves. Another important operation is the ``inverse'' curve:
If $\gamma=(a_1,\dots,a_n)$ is a curve (opcurve), then
$\gamma^{-1}:=(T_1a_n,\dots,T_1a_1)$
is also a curve (opcurve). But it is not the inverse of $\gamma$ until we pass to the paths of the (op)curve.
Paths are equivalence classes of curves with respect to insertion and deletion of length 2 subcurves of the form $(a,T_1a)$.
Every equivalence class contains a unique element which does not contain any $(a,T_1a)$ subcurves, i.e., which is reduced.
We shall not introduce any
special notation for paths, just say ``the path $\gamma$'' for a curve $\gamma$ if we think of its equivalence class. This convention is used also
for opcurves, so we say ``the path $\gamma$'' even if $\gamma$ is an opcurve.
The paths of curves, as well as the paths of opcurves, form a groupoid $\Path(\Sigma)$ and $\Path(\Sigma)^\op$, respectively.
Of course, this path groupoid is not the fundamental groupoid of the surface~$[\Sigma]$, it is that of the 1-skeleton of $\Sigma$, so this groupoid
sees every face of~$\Sigma$ as a hole on~$[\Sigma]$.

Although the definition of (op)curves and therefore also of the path groupoid depend on the arrow presentation, one can easily verify that any isomorphism of
arrow presentations induces an isomorphism between the corresponding categories of (op)curves as well as an isomorphism of their path groupoids.

Using the term $P$-curve instead of ``curve on $\Sigma$'' is nevertheless advisable, especially if we~want to define dual curves.
The dual (op)curves on $\Sigma$ can be thought as (op)curves on the dual $\Sigma^*$. But if we want to see them as sequences of arrows
in $\Sigma$ then we need a mapping $\Arr(\Sigma^*)\to\Arr(\Sigma)$. Such mappings are not unique as it is witnessed by the two duals~\smash{$\duP$} and
\smash{$\iduP$} of a presentation. So, if the context requires the precision, we shall speak about \smash{$\duP$}-(op)curves and \smash{$\iduP$}-(op)curves instead of
``dual (op)curves on $\Sigma$''.
\begin{defi}
A \smash{$\duP$}-opcurve is a sequence $\beta=(a_1,\dots,a_n)$ of arrows $a_i\in\Arr(\Sigma)$ such that $d_0a_i=d_1 a_{i+1}$ for $0<i<n$,
where we introduced the notation
\[
d_0 a:=\Ou_2(a),\qquad d_1 a:=\Ou_2(T_1a).
\]
The \smash{$\iduP$}-opcurves are defined as the reverses of \smash{$\duP$}-opcurves.
\end{defi}

These definitions are based on the quivers
\[
\Qv\bigl(\duP\bigr)=\biggl(\Arr(\Sigma)\shortpair{d_1}{d_0}\Sigma^0\biggr),\qquad\Qv\bigl(\iduP\bigr)=\biggl(\Arr(\Sigma)\shortpair{d_0}{d_1}\Sigma^0\biggr).
\]
In fact, the \smash{$\duP$}-opcurves coincide with the \smash{$\iduP$}-curves and vice versa.

\begin{rmk}
Passing from curves to opcurves, or back, can be automatized by introducing ${P^\op:=\bra\Arr(\Sigma),T_1T_0T_1,T_1T_2T_1\ket}$ which is a presentation isomorphic to~$P$
via the transformation~$T_1$. Since $T_1$ is also the involution of the quiver, so
\[
\Qv(P^\op)=\biggl(\Arr(\Sigma)\shortpair{\del_1}{\del_0}\Sigma^0\biggr),
\]
and therefore the $P^\op$-curves are the $P$-opcurves. Also, we have \smash{$\bigl(\duP\bigr)^\op=\iduP$}. Although this might suggest a redundancy of using both curves and opcurves but
the concept of $P$-opcurve comprises not only the use of $P^\op$-curves but also the convention of using concatenation for their composition.
\end{rmk}

Now we turn to the description of curves on the dual of the double $\D(\Sigma)^*$ of an OCPM~$\Sigma$.
The dual of the double is particularly interesting for two reasons: Geometrically because it provides the Schreier coset graph for the
arrow presentation, and physically because it is the surface complex on which holonomy of the Kitaev model can be defined.

Fixing an arrow presentation $P=\bra A,T_0,T_2\ket$ for $\Sigma$, the vertices of $\D(\Sigma)^*$ can be identified with the set of sites, i.e.,
with pairs $\bra v,f\ket\in\Sigma^0\x\Sigma^2$ such that $v\cap f\neq\varnothing$, cf.~(\ref{site}) and Lemma~\ref{lem: site}. The set of
arrows of $\D(\Sigma)^*$ can be identified with the set $\mathbb{A}=\{a_i^\sigma\mid a\in A,\,i=0,2,\,\sigma=\pm\}$, see Figure~\ref{darrows}.
The meaning of the arrows can be made explicit by giving the source and target maps $\nabla_0,\nabla_1\colon \mathbb{A}\to\D(\Sigma)^{*0}$,
\begin{equation}\label{nabla}
\begin{aligned}
\nabla_0a_0^+&=\bra\del_0a,d_1a\ket\\
\nabla_1a_0^+&=\bra\del_0a,d_0a\ket\\
\nabla_0a_2^+&=\bra\del_0a,d_0a\ket\\
\nabla_1a_2^+&=\bra\del_1a,d_0a\ket
\end{aligned}
\qquad\text{and}\qquad
\begin{aligned}
\nabla_0a_i^-&=\nabla_1a_i^+\\
\nabla_1a_i^-&=\nabla_0a_i^+.
\end{aligned}
\end{equation}
The quiver $\mathbb{A}\shortpair{\nabla_0}{\nabla_1}\D(\Sigma)^{*0}$ can be seen to be the quiver associated to the arrow presentation~(\ref{AP dudo}) so we call it $\Qv(\D(P)^\sim)$. Curves and opcurves on $\D(\Sigma)^*$ will always be meant with respect to this quiver.

\begin{lem} \label{lem: E}
The following map embeds the involutive quiver of $\D(P)^\sim$ to the cartesian product of the involutive quivers of $P$ and \smash{$\duP$}:
\begin{align*}
\biggl(\mathbb{A}\shortpair{\nabla_0}{\nabla_1}\D(\Sigma)^{*0}\biggr)&\rarr{E}
\biggl(A\shortpair{\del_0}{\del_1}\Sigma^0\biggr)\x\biggl(A\shortpair{d_1}{d_0}\Sigma^2\biggr),\\
\bra v,f\ket&\mapsto\bra v,f\ket,\\
a_0^+&\mapsto\bra\del_0a,a\ket,\\
a_0^-&\mapsto\bra\del_0a,T_1a\ket,\\
a_2^+&\mapsto\bra a,d_0a\ket,\\
a_2^-&\mapsto\bra T_1a,d_0a\ket.
\end{align*}
\end{lem}
\begin{proof}
Recall that the cartesian product of quivers $\mathbb{Q}$ and $\mathbb{Q'}$ has set of vertices $(\mathbb{Q}\x\mathbb{Q'})^0=\mathbb{Q}^0\x\mathbb{Q'}^0$
and set of arrows \smash{$(\mathbb{Q}\x\mathbb{Q'})^1=\bigl(\mathbb{Q}^0\x\mathbb{Q'}^1\bigr)\sqcup\bigl(\mathbb{Q}^1\x\mathbb{Q'}^0\bigr)$} and it is endowed with the obvious
source and target maps.
In case of the quivers of $\Sigma$ and $\Sigma^*$, these obvious maps are $\del_0\x d_1$ and $\del_1\x d_0$ with the warning that here the $\del_i$ and
$d_i$ are used in a generalized sense: The $\del_i$ is extended to the trivial edges of $\Sigma$, i.e., to the vertices, by $\del_0 v=\del_1 v=v$ and the
$d_i$ is extended to the trivial edges of $\Sigma^*$, i.e., to the faces, by $d_0 f=d_1 f=f$. Then the fact that $E$ is a homomorphism of
involutive quivers,
\begin{align*}
(\del_0\x d_1)\ci E=E\ci\nabla_0,\qquad
(\del_1\x d_0)\ci E=E\ci\nabla_1,\qquad
(T_1\x \tilde{T}_1)\ci E=E\ci\tilde{\mathbb{T}}_1
\end{align*}
can be easily verified. That $E$ is injective is also clear.
\end{proof}

In the image of $E$ the arrows can be visualized as ``isosceles triangles''. The triangle $E\bigl(a_2^\pm\bigr)$ has base edge $a$ with opposite vertex being the face
$f=d_0a$ and the legs of the triangle being the sites $\bra v,f\ket$ where $v$ is either $\del_0a$ or $\del_1a$. The triangles $E\bigl(a_0^\pm\bigr)$ are called
dual triangles because in them the roles of vertices and faces are interchanged. If $\rho=(d_j)$ is a curve on $\D(\Sigma)^*$, then two
consecutive triangles $E(d_j)$ and $E(d_{j+1})$ are pasted together at a leg so the whole picture of the curve looks like a ribbon with the base
edges of triangles forming one borderline and that of the dual triangles (drawn as edges of $\Sigma^*$) forming the other borderline.
\begin{figure}[t]\centering
\includegraphics[width=10cm]{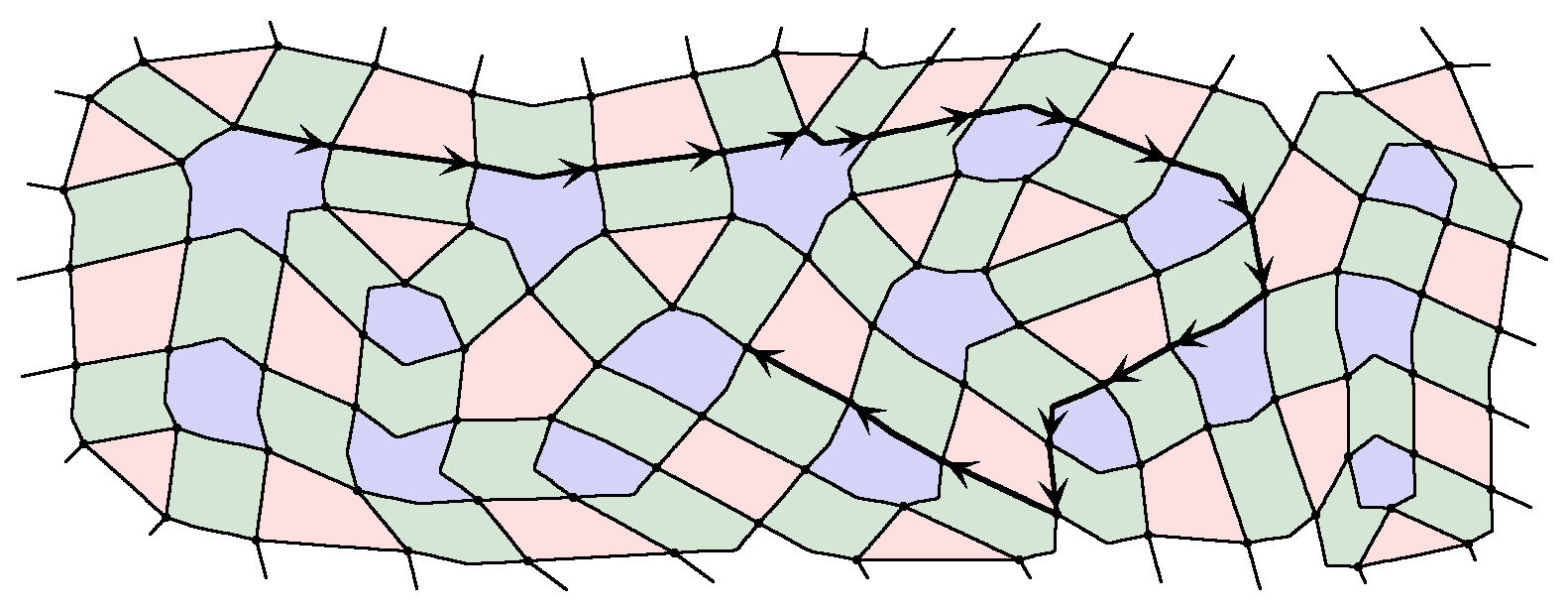}
\includegraphics[width=10cm]{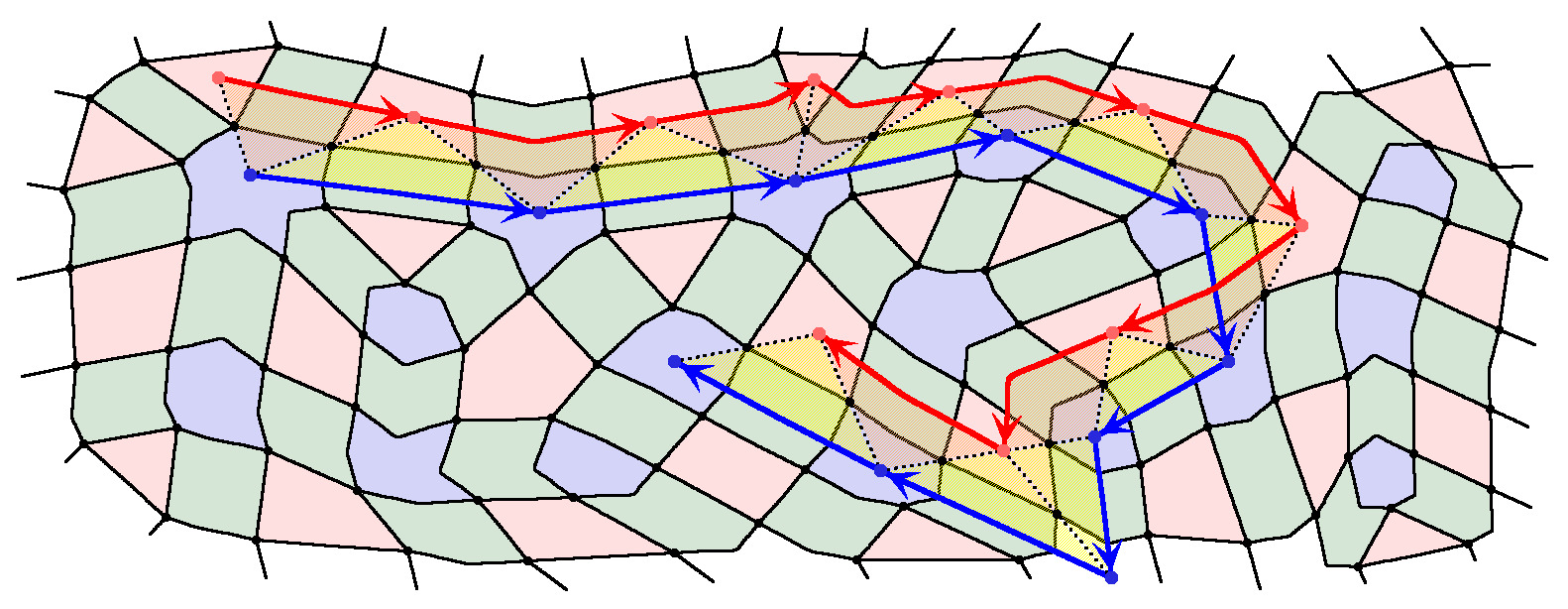}
\caption{Every curve $\rho$ (black) on $\D(\Sigma)^*$ is accompanied by a curve $\gamma$ (blue) on $\Sigma$ and a curve $\beta$ (red) on $\Sigma^*$
according to Lemma~\ref{lem: E}. They border a strip made of triangles but the strip may fail to be embedded into the surface of $\Sigma$:
The $\gamma$ and $\beta$ intersect locally, i.e., the strip folds, precisely if the $\rho$ makes a turn around a type 1 (green) face of $\D(\Sigma)^*$.}
\label{fig: triangles of a curve}
\end{figure}
On the two ends, we see two unpaired legs of some triangles and the texture of the ribbon itself is a triangulated surface, see
Figure~\ref{fig: triangles of a curve}.
The 2 borderline curves can be obtained by writing $E(d_j)$ as $\bra c_j,b_j\ket$ and then $\gamma=(c_j)$ is a $P$-curve and $\beta=(b_j)$ is a \smash{$\duP$}-curve
in the generalized sense: $\gamma$ may contain vertices and $\beta$ may contain faces. But $c_j$ is a vertex if and only if $b_j$ is a genuine arrow and
$c_j$ is a genuine arrow if and only if $b_j$ is a face, hence the pair $\bra c_j,b_j\ket$ is always a triangle.

This is the origin of the names such as ``triangle operators'' and ``ribbon curves'' in the literature on the Kitaev model, although the name
ribbon curve is usually reserved for a special kind of curve on $\D(\Sigma)^*$ which we will define below.

In the next lemma, we take a closer look into the structure of $\D(\Sigma)^*$. We will see that we can forget about the complicated arrow presentation
(\ref{AP dudo}) and navigate through $\D(\Sigma)^*$ using essentially only the arrow presentation of $\Sigma$.
\begin{lem} %\label{lem: code word}
The map $a\mapsto s(a):=\bra\Ou_0(a),\Ou_2(a)\ket$ is a bijection $\Arr(\Sigma)\cong\D(\Sigma)^{*0}$.
Every vertex of $\D(\Sigma)^*$ has degree $4$ and the four arrows pointing out from $s(a)$ can be labelled, or colored, by the letters
$\bigl\{T_0,T_0^{-1},T_2,T_2^{-1}\bigr\}$ as follows
\begin{alignat*}{4}
&(T_0a)_0^+\colon \ && s(a)\to s(T_0a)\qquad&&\text{labelled by }T_0,&\\
&a_0^-\colon \ && s(a)\to s\bigl(T_0^{-1}a\bigr)\qquad&&\text{labelled by }T_0^{-1},&\\
&a_2^+\colon \ && s(a)\to s(T_2a)\qquad&&\text{labelled by }T_2,&\\
&\bigl(T_2^{-1}a\bigr)_2^-\colon \ && s(a)\to s\bigl(T_2^{-1}a\bigr)\qquad&&\text{labelled by }T_2^{-1}&.
\end{alignat*}
The order in which they are listed corresponds to the cyclic order induced by the action of $\tilde{\mathbb{T}}_0$ of~\eqref{AP dudo}. Moreover,
if an arrow $d$ carries the label $T_j^\epsilon$ then $\tilde{\mathbb{T}}_1d$ has label $T_j^{-\epsilon}$.
\end{lem}

\begin{proof}
$\Ou_0(a)\cap\Ou_2(a)=\{a\}$ is non-empty, so $s(a)$ is a site, cf.\ Definition~\ref{def: site} and (\ref{site}). If~$\bra v,f\ket$~is a~site, then
$v\cap f\neq\varnothing$, so $v=\Ou_0(a)$ and $f=\Ou_2(a)$
for the unique element $a\in v\cap f$. So $s$ is indeed a bijection.

The four arrows of source $s(a)$ are dual to the 4 boundary arrows of the rhombus face $s(a)$ of the double $\D(\Sigma)$. (The reason for drawing $\D(\Sigma)$ and not
$\D(\Sigma)^*$ is merely typographical.)
\begin{equation}\label{rhombus nb}
\parbox{120pt}{
\setlength{\unitlength}{1.2pt}
\begin{picture}(120,80)
\linethickness{0.2pt}
\polyline(20,0)(80,60)(60,80)(0,20)(20,0)
\polyline(60,0)(80,20)(20,80)(0,60)(60,0)
\color{melegzold}
\put(20,0){\circle*{4}}
\put(60,0){\circle*{4}}
\put(20,40){\circle*{4}}
\put(60,40){\circle*{4}}
\put(20,80){\circle*{4}}
\put(60,80){\circle*{4}}
\put(7,36){$e^L$}
\put(65,36){$e^R$}
\color{blue}
\put(40,20){\circle*{4}}
\put(0,60){\circle*{4}}
\put(80,60){\circle*{4}}
\put(38,25){$v$}
\put(-13,56){$v^L$}
\put(83,56){$v^R$}
\color{red}
\put(0,20){\circle*{4}}
\put(80,20){\circle*{4}}
\put(40,60){\circle*{4}}
\put(37,49){$f$}
\put(-13,18){$f^L$}
\put(83,18){$f^R$}
%\color{black}
\linethickness{2pt}
\color{black}
\put(43,19){\vector(1,1){40}} %intentionally shifted
\put(78,46){$a$}
\color{black}
\put(33,38){$\scriptstyle s(a)$}
\put(8,18){$\scriptstyle s(T_0a)$}
\put(48,18){$\scriptstyle s(T_0^{-1}a)$}
\put(6,58){$\scriptstyle s(T_2^{-1}a)$}
\put(49,58){$\scriptstyle s(T_2a)$}
\end{picture}}
\qquad
\begin{alignedat}{3}
&v=\Ou_0(a),\qquad&& f=\Ou_2(a)&\\
&e^L=\Ou_1(T_0a),\qquad&& e^R=\Ou_1(a)&\\
&v^L=\Ou_0\bigl(T_2^{-1}a\bigr),\qquad&& v^R=\Ou_0(T_1a)&\\
&f^L=\Ou_2(T_0a),\qquad&& f^R=\Ou_2(T_1a).&
\end{alignedat}
\end{equation}
The $e^L$, $v$, $e^R$, $f$ are the 4 vertices of the rhombus by (\ref{2 edges of a site}), $\bigl\{v,v^L\bigr\}=\Bd\bigl(e^L\bigr)$, $\bigl\{v,v^R\bigr\}=\Bd\bigl(e^R\bigr)$,
$\bigl\{f,f^L\bigr\}=\Cb\bigl(e^L\bigr)$, $\bigl\{f,f^R\bigr\}=\Cb\bigl(e^R\bigr)$ and the four neighbour sites of $\bra v,f\ket$ are $\bigl\bra v,f^L\bigr\ket$, $\bigl\bra v,f^R\bigr\ket$, $\bigl\bra v^R,f\bigr\ket$ and
$\bigl\bra v^L,f\bigr\ket$. Computing the intersections of the vertex and face within these sites, we recognize them as $s(T_0a)$, $s\bigl(T_0^{-1}a\bigr)$, $s(T_2a)$
and $s\bigl(T_2^{-1}a\bigr)$, respectively. \big(For exam\-ple, $v\cap f^L=\Ou_0(a)\cap\Ou_2(T_0a)=\Ou_0(T_0a)\cap\Ou_2(T_0a)=\{T_0a\}$.\big)
This partially explains the labels of the four arrows, partially because some of these 4 sites can be identical. (Indeed, $\D(\Sigma)^*$ may have double arrows but never
multiple arrows of larger multiplicity.)
$s(T_0a)=s\bigl(T_0^{-1}a\bigr)$ occurs when $v$ has degree 2 and $s(T_2a)=s\bigl(T_2^{-1}a\bigr)$ occurs when $f$ is a 2-gon. In these degenerate cases
\[
\parbox{100pt}{
\begin{picture}(100,85)
\put(44,50){$\scriptstyle s(a)$}
\polyline(50,40)(70,40)(50,70)(30,40)(50,40)
\polyline(50,70)(25,85)(5,55)(30,40)
\polyline(50,70)(75,85)(95,55)(70,40)
\polyline(30,40)(50,10)(70,40)
\color{blue}
\put(50,40){\circle*{4}} \put(48,32){$v$}
\put(5,55){\circle*{4}} \put(-8,52){$v^L$}
\put(95,55){\circle*{4}} \put(98,52){$v^R$}
\color{red}
\put(50,70){\circle*{4}} \put(48,75){$f$}
\put(50,10){\circle*{4}} \put(54,6){$f^L=f^R$}
\color{melegzold}
\put(30,40){\circle*{4}} \put(17,32){$e^L$}
\put(70,40){\circle*{4}} \put(73,32){$e^R$}
\put(25,85){\circle*{4}}
\put(75,85){\circle*{4}}
\end{picture}}
\qquad\text{or}\qquad
\parbox{100pt}{
\begin{picture}(100,85)(0,-10)
\put(44,32){$\scriptstyle s(a)$}
\polyline(50,45)(70,45)(50,15)(30,45)(50,45)
\polyline(50,15)(25,0)(5,30)(30,45)
\polyline(50,15)(75,0)(95,30)(70,45)
\polyline(30,45)(50,75)(70,45)
\color{red}
\put(50,45){\circle*{4}} \put(48,50){$f$}
\put(5,30){\circle*{4}} \put(-9,28){$f^L$}
\put(95,30){\circle*{4}} \put(97,28){$f^R$}
\color{blue}
\put(50,15){\circle*{4}} \put(48,6){$v$}
\put(50,75){\circle*{4}} \put(54,72){$v^L=v^R$}
\color{melegzold}
\put(30,45){\circle*{4}} \put(22,47){$e^L$}
\put(70,45){\circle*{4}} \put(72,47){$e^R$}
\put(25,0){\circle*{4}}
\put(75,0){\circle*{4}}
\end{picture}}
\]
there is a double
arrow from $s(a)$ to, let's say, $s(T_2a)$ (the 2nd picture above) and it seems to be indifferent which is labelled by $T_2$ because $T_2$ acts on $a$ as $T_2^{-1}$.
However, the letters $T_2$ and~$T_2^{-1}$ are different and so are the two parallel arrows. The labelling in this case is defined by forcing the lemma:
The labels must follow the cyclic order $\big[T_0,T_0^{-1},T_2,T_2^{-1}\big]$ when we apply~$\tilde{\mathbb{T}}_0$.
There is only one case left. If both types of degeneracy occur for the same arrow~$a$, then $T_0^2a=a=T_2^2a$ so by Example~\ref{exa: minimal OCPM} the connected component
of $a$ is nothing but the minimal OCPM, $\Sigma_\mini$. In this case, the correctness of the labelling can be checked by inspection.
\end{proof}

The above lemma has two important consequences. The first one is an interpretation of~$\D(\Sigma)^*$ as the coset graph of the arrow presentation of $\Sigma$.
This is the next corollary. The second one is a useful encoding of the curves on $\D(\Sigma)^*$ in terms of words in the alphabet $\bigl\{T_0,T_0^{-1},T_2,T_2^{-1}\bigr\}$, see
the definition~below.

\begin{cor}Let $\Sigma$ be a connected OCPM. Every arrow presentation $\bra\Arr,T_0,T_2\ket$ of $\Sigma$ determines a transitive action on $\Arr(\Sigma)$ of the
$\Sigma$-independent group $\T=\bra T_0,T_2\mid T_0T_2T_0T_2\ket$ and therefore the arrow presentation can be thought as a presentation of the left
coset space $\T/\Ha$ of the group $\T$ with respect to some subgroup. Then the Schreier coset graph of $\T/\Ha$ can be drawn on
the surface $[\Sigma]$ as follows. Each $a\in\Arr(\Sigma)$
is drawn as a point at $s(a)\in \D(\Sigma)^{*0}$. For each $a\in\Arr(\Sigma)$, we draw two arrows: $(T_0a)_0^+\colon s(a)\to s(T_0a)$ colored blue and
$a_2^+\colon s(a)\to s(T_2a)$ colored red. Then every edge of $\D(\Sigma)^*$ is drawn exactly once. The blue arrows represent the flow of the permutation
$T_0$ and the red ones represent that of $T_2$. In this sense drawing the dual of the double we draw the arrow presentation of~$\Sigma$
$($Figure~$\ref{fig: flow})$.
\end{cor}

\begin{figure}[t]\centering
\includegraphics[width=9cm]{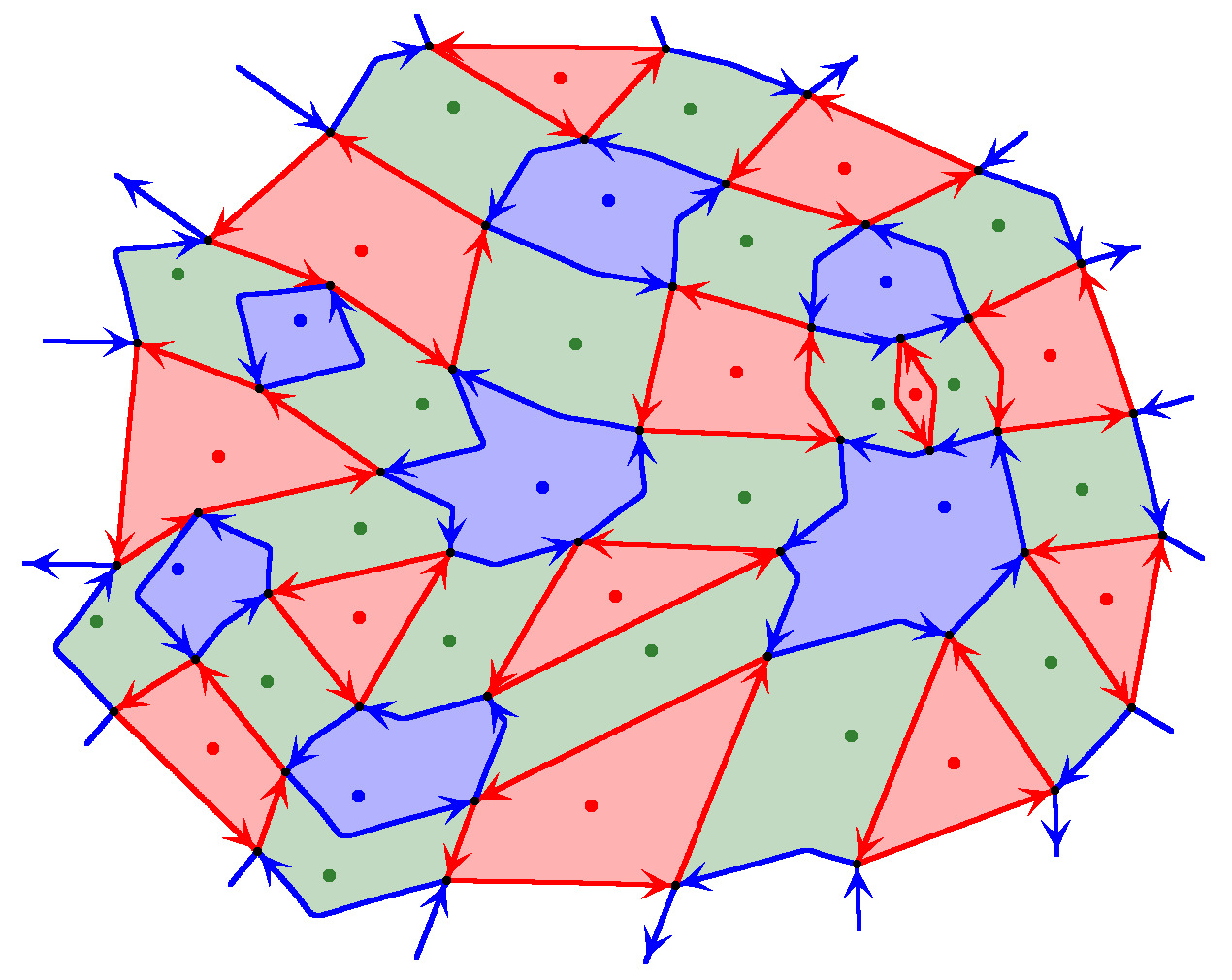}
\caption{The dual of the double of an OCPM $\Sigma$ provides the Schreier coset graph of the arrow presentation of $\Sigma$. The blue and red arrows
show the action of $T_0$ and $T_2$, respectively. The green quadrangles represent the universal relation $T_0T_2T_0T_2=\id$.}
\label{fig: flow}
\end{figure}

\begin{defi}\label{def: code word}
Using the function $s\colon \Arr(\Sigma)\to\D(\Sigma)^{*0}$ defined in the above lemma a curve $\rho=(d_n,\dots,d_1)$ on $\D(\Sigma)^*$ can be encoded by the pair
\[
\bra a_0,w\ket\in\Arr(\Sigma)\x\bigl\{T_0,T_0^{-1},T_2,T_2^{-1}\bigr\}^*,
\]
where $s(a_0)=\nabla_0 d_1$ is the source vertex of the curve, $w=W_n\cdots W_1$ is the code word made of letters $W_j\in\bigl\{T_0,T_0^{-1},T_2,T_2^{-1}\bigr\}$
in such a way that for $0<j\leq n$ the $d_j$ is the unique arrow of source $s(a_{j-1})$ and label $W_j$, so its target can be written as $s(a_j)$ with
$a_j=W_ja_{j-1}$.

The same pair $\bra a_0,w\ket$ is used to encode the opcurve $\rho^\rev=(d_1,\dots,d_n)$.
\end{defi}

Now we are in a position to define ribbons.
\begin{lem}\label{lem: ribbon}
Let $\rho=(d_1,\dots,d_n)$ be either a curve or an opcurve on $\D(\Sigma)^*$ with encoding $\bra a_0,w\ket$.
The following conditions are equivalent:
\begin{enumerate}\itemsep=0pt
\item[$(i)$] The triangles $E(d_j)$ and $E(d_{j+1})$ are non-overlapping $($their interiors are disjoint$)$ for $0<j<n$.
\item[$(ii)$] $\rho$ is reduced and the $2$ borderlines $\gamma$ and $\beta$ of the ``ribbon'' $\rho$ do not intersect locally, i.e., $\Ou_1(b_j)\neq\Ou_1(c_{j\pm 1})$
whenever both $b_j$ and $c_{j\pm 1}$ are non-trivial arrows.
\item[$(iii)$] $\rho$ is reduced and if $(d_j,d_{j+1})$ is a left or a right turn (not straight), then the unique face it turns around is either a type $0$ face
or a type $2$ face, never a type $1$ face.
\item[$(iv)$] The word $w$ is reduced and does not contain $T_0T_2$, $T_2T_0$ and their inverses.
\item[$(v)$] The word $w$ belongs to either $\bigl\{T_0^{-1},T_2\bigr\}^*$ or $\bigl\{T_0,T_2^{-1}\bigr\}^*$.
\end{enumerate}
\end{lem}
\begin{proof}
Let $\bra v,f\ket=\nabla_0d_{j+1}=\nabla_1d_j$. Then the triangles $E(d_j)$ and $E(d_{j+1})$ contain $\bra v,f\ket$ as a~leg.
There are 4 such triangles. In the notations of (\ref{rhombus nb}), they are $vfv^L$, $fvf^L$, $vfv^R$ and $fvf^R$. The first two and the second two
overlap because their base edges intersect at their ``middle'' points $e^L$ and $e^R$, respectively. This intersection is, at the same time, an intersection
of the two borderlines $\gamma$ and $\beta$ of the ribbon.
No other pairs of triangles are overlapping except when the two triangles are the same, this happens precisely when $d_{j+1}=\tilde{\mathbb{T}}_1d_j$.

In case of intersection at $e^L$ the $(d_j,d_{j+1})$ is either \smash{$\bigl(\bigl(T_2^{-1}a\bigr)_2^+,(T_0a)_0^+\bigr)$} or its inverse and the corresponding code word is
$T_0T_2$ or its inverse.
In case of intersection at $e^R$, the $(d_j,d_{j+1})$ is either $\bigl(a_0^+,a_2^+\bigr)$ or its inverse and the corresponding code word is
$T_2T_0$ or its inverse. These pairs of arrows make turn around the type 1 face $e^L$ and $e^R$, respectively. All other pairs of arrows are either straight
or make a turn around $v$ or $f$.

If a reduced word $w$ does not contain the alphabet changing words $T_0T_2$, $T_2T_0$ and their inverses, then $W_1\in\bigl\{T_0^{-1},T_2\bigr\}$ is equivalent to
$w\in\bigl\{T_0^{-1},T_2\bigr\}^*$. The same holds for the alphabet $\bigl\{T_0,T_2^{-1}\bigr\}$.
\end{proof}

\begin{defi}\label{def: ribbon}
A ribbon (op)curve is a (op)curve on $\D(\Sigma)^*$ satisfying the conditions of Lemma~\ref{lem: ribbon}.
Let $\Ribb(\Sigma)$ and $\Ribb^\circ(\Sigma)$ be the set of ribbon curves and ribbon opcurves, respectively.
If the code word of a ribbon (op)curve belongs to $\bigl\{T_0^{-1},T_2\bigr\}^*$, then it is called a~left ribbon (op)curve.
If it belongs to $\bigl\{T_0,T_2^{-1}\bigr\}^*$, we speak about right ribbon (op)curve.
Let $\Ribb_L(\Sigma)$, $\Ribb_R(\Sigma)$, $\Ribb^\circ_L(\Sigma)$, $\Ribb^\circ_R(\Sigma)$ be the set of left/right ribbon curves/opcurves, respectively.
\end{defi}

Since every ribbon is reduced, $\Ribb(\Sigma)$ embeds into the path groupoid $\Path(\Sigma)$. But it is not a subgroupoid and not even a subcategory
as one can infer from the characterization Lemma~\ref{lem: ribbon}\,(v). The $\Ribb_L(\Sigma)$ and $\Ribb_R(\Sigma)$, however, are subcategories with
the inverse of every left ribbon being a right ribbon. Moreover, $\Ribb(\Sigma)=\Ribb_L(\Sigma)\cup\Ribb_R(\Sigma)$.
Analogous remarks can be made for the ribbon opcurves.

The above definition of ribbons seems to agree with that of~\cite{Bombin-Delgado} but it is less restrictive than that of \cite{Meusburger}.
The latter type of ribbons will appear here as proper ribbons in Definition~\ref{def: prop ribb}. The necessity of distinguishing two kinds of ribbons
has been realized already in~\cite{CCY}, where they are called locally clockwise and locally counterclockwise. Our terminology of left and right refers to that whether
the red dual curve on Figure~\ref{fig: triangles of a curve} is on the left or on the right of the blue curve.
So ribbons may also be described as curves on $\Sigma$ together with a consistent framing by a dual curve. It is a left ribbon if and only if the framing goes, and stays, on the left-hand side.

\section{Holonomy} \label{sec: holonomy}

According to \cite{Meusburger, Meusburger-Wise}, the holonomy of a Hopf algebra gauge theory on $\Sigma$ over the quasitriangular Hopf algebra $Q$ is a
groupoid homomorphism
\[
\Hol\colon \ \Path(\Sigma)\to\Hom_K^\x\biggl(Q^*,\bigotimes_{e\in\Sigma^1}Q^*\biggr),
\]
where the $\Hom_K(C,A)$ for a coalgebra $C$ and algebra $A$ is the convolution algebra \cite{Sweedler} and $\Hom_K^\x(C,A)$ refers to the group of
convolution invertible elements of $\Hom_K(C,A)$.
For a gauge theory over $\D(H)$ we must choose $Q=\D(H)$. But since we need holonomy in the Kitaev model merely for the purpose of constructing string operators,
we will follow the above recipe very loosely and use the name holonomy for any groupoid homomorphism
\[
\Hol\colon \ \Path(\D(\Sigma)^*)\to\Hom_K^\x(\D(H)^*,\M(\Sigma)).
\]
The convolution product of two functions $f,g\colon \D(H)^*\to\M$ is $f\cv g(\Psi)=f\bigl(\Psi\p\bigr)g\bigl(\Psi\pp\bigr)$ for $\Psi\in \D(H)^*$ and the holonomy of a path
$(d_n,\dots,d_1)$ on $\D(\Sigma)^*$ is uniquely determined by the holonomies of the length 1 paths by
\begin{equation}\label{Hol long}
\Hol_{(d_n,\dots,d_1)}=\Hol_{(d_n)}\cv\cdots\cv\Hol_{(d_2)}\cv\Hol_{(d_1)}.
\end{equation}
In order for this formula to depend only on the path of the curve the length 1, holonomies should satisfy that $\Hol_{(\tilde{\mathbb{T}}_1d)}$
be the convolution inverse of $\Hol_{(d)}$. If the function $\Hol_{(d)}\colon {\D(H)^*\to\M}$ happens to be either an algebra homomorphism or an antialgebra
homomorphism, then its convolution inverse is simply $\Hol_{(d)}\ci S_{\D^*}$.

We shall use also opholonomies
\[
\Ophol\colon \ \Path(\D(\Sigma)^*)^\op\to\Hom_K^\x(\D(H)^*,\M(\Sigma)),
\]
which can be defined as functions on opcurves $(d_1,\dots,d_n)$ by
\begin{equation}\label{Ophol long}
\Ophol_{(d_1,\dots,d_n)}=\Ophol_{(d_1)}\cv\Ophol_{(d_2)}\cv\cdots\cv\Ophol_{(d_n)}.
\end{equation}
The condition of descending to the paths of opcurves is again that $\Ophol_{(\tilde{\mathbb{T}}_1d)}$ be the convolution inverse of $\Ophol_{(d)}$.

The ribbon operators of \cite{Bombin-Delgado}, as well as the symbolic arrows (\ref{P-arrow Q-arrow}), suggest that opholonomy and holonomy
should be defined by
\begin{equation}\label{Ophol}
\Ophol_{(d)}(h\ot\varphi):=\begin{cases}
\varphi(1)P_a(h)&\text{if}\ d=a_0^+,\\
\eps(h)Q_a(\varphi)&\text{if}\ d=a_2^+,\\
\varphi(1)P_a(S(h))&\text{if}\ d=a_0^-,\\
\eps(h)Q_a(S(\varphi))&\text{if}\ d=a_2^-
\end{cases}
\end{equation}
for all $a\in\Arr(\Sigma)$ and
\begin{equation}\label{Hol}
\Hol_{(d)}:=S_\M\ci\Ophol_{(d)}\ci S_{\D^*},\qquad d\in\Arr(\D(\Sigma)^*),
\end{equation}
where $S_\M$ is the involutive algebra automorphism $\M\to\M$ which sends $P_a(h)$ to $P_{T_1a}(h)$ and $Q_a(\varphi)$ to $Q_{T_1a}(\varphi)$.
By (\ref{D^* mul}), the $\Ophol_{(d)}$ are antialgebra maps for $d=a_0^+$ and algebra maps for $d=a_2^+$. Moreover, in spite of the complicated
expression (\ref{Sdudo}) it is easy to see that $\Ophol_{(\tilde{\mathbb{T}}_1d)}=\Ophol_{(d)}\ci S_{\D^*}$. So (\ref{Ophol}) together with
(\ref{Ophol long}) defines indeed an opholonomy. Then (\ref{Hol}) together with (\ref{Hol long}) is automatically a holonomy.

The relation (\ref{Hol}) between holonomy and opholonomy is rather arbitrary. We could have defined them in such a way that $\Hol_{(d)}=\Ophol_{(d)}$
for all arrows $d$ of $\D(\Sigma)^*$. The choice (\ref{Hol}) perhaps allows easier comparison with the existing definitions of the literature.
\begin{defi}\label{def: prop ribb}
Let $\gamma=(d_1,\dots,d_n)$ be a (op)curve on $\D(\Sigma)^*$ with $d_k=(a_k)^{\sigma_k}_{i_k}$. Then
$\gamma$ is called proper if for all $1\leq k<l\leq n$ the following statements hold: $a_k\neq a_l$ and if $a_k=T_1 a_l$ then $i_k=i_l$.
In other words, no pairs $d_k$, $d_l$ of arrows in the (op)curve are equal or incident and lie on the boundary of the same type 1 face $e\in\Sigma^1$.
\end{defi}
\begin{pro}\label{pro: Hol-Ophol}
The relation of proper $($op$)$curves to ribbons is the following:
\begin{enumerate}\itemsep=0pt
\item[$(i)$] Every proper curve $\gamma=(d_n,\dots,d_1)$ on $\D(\Sigma)^*$ is a ribbon curve and
\[
\Hol_\gamma=S_\M\ci\Ophol_{\gamma^\rev}\ci S_{\D^*}=S_\M\ci\Ophol_{\gamma^{\rev -1}}.
\]
\item[$(ii)$] A ribbon opcurve $\rho=(d_1,\dots,d_n)$ on $\D(\Sigma)^*$ is proper if and only if it is a simple opcurve. So proper ribbon and simple ribbon are synonymous.
\end{enumerate}
\end{pro}
\begin{proof}
(i) Notice that an opposite pair $d$, $\tilde{\mathbb{T}}_1d$ is always of the form $a_i^+$, $a_i^-$.
Therefore, proper (op)curves are reduced. Hence, the property of being proper for consecutive arrows $d_k$, $d_{k+1}$ is equivalent to
the property of Lemma~\ref{lem: ribbon}\,(iii). This proves that proper (op)curves are ribbons.
The property of being proper for any pair $d_k$, $d_l$ implies that the $P$ or $Q$ operators associated to them by the opholonomy commute:
$\bigl[\Ophol_{(d_k)}(\Phi),\Ophol_{(d_l)}(\Psi)\bigr]=0$ for every $\Phi,\Psi\in\D(H)^*$ and for $k\neq l$. Therefore, \begin{align*}
\Hol_\gamma(\Psi) =S_\M\bigl(\Ophol_{(d_1)}\bigl(S_{\D^*}\bigl(\Psi_{(n)}\bigr)\bigr)\cdots \Ophol_{(d_n)}\bigl(S_{\D^*}\bigl(\Psi_{(1)}\bigr)\bigr)\bigr) =S_\M\bigl(\Ophol_{\gamma^\rev}(S_{\D^*}(\Psi))\bigr).
\end{align*}
This proves the first equation. The second follows from the first by the observation that for a~proper opcurve $\alpha$ the
$\Ophol_\alpha\ci S_{\D^*}$ is the convolution inverse of $\Ophol_\alpha$ hence equal to the opholonomy of $\alpha^{-1}$.

(ii) Assume $\rho$ is not simple. Then there exists a site $s$ and a minimal closed subcurve $(d_i,d_{i+1},\dots,d_{j-1})\colon s\to s$ of $\rho$. Here $i=1$ or $j=n+1$ is allowed
but not both because this would imply that $\rho$ is closed and simple. If $i>1$, then $d_{i-1}$ and $d_{j-1}$ are incident to $s$ and lie on the boundary of a type~1 face.
If $j\leq n$, then the same holds for the arrows $d_i$ and $d_j$. Anyway, we get contradiction with properness of $\rho$.

If $\rho$ is not proper, then for some $1\leq k<l\leq n$ the $d_k=(a_k)^{\sigma_k}_{i_k}$ and $d_l=(a_l)^{\sigma_l}_{i_l}$ satisfy one of the following: either $a_k=a_l$ or
($a_k=T_1 a_l$ and $i_k\neq i_l$). In both cases $\rho$ is either not simple or not a ribbon.
\end{proof}

En passant, we mention a useful rewriting of the exchange relations of the $P$ and $Q$ operators in terms of length 1 opholonomies:
\begin{align}
\label{ophol exch rel1}
&\Ophol_{(a_0^-)}(\Psi)\Ophol_{(a_2^+)}(\Phi)=\Ophol_{(a_2^+)}(\Phi\sweedr R_1)\Ophol_{(a_0^-)}(\Psi\sweedr R_2),\\
\label{ophol exch rel2}
&\Ophol_{((T_1a)_2^+)}(\Psi) \Ophol_{(a_0^-)}(\Phi)=\Ophol_{(a_0^-)}(R_2\sweedl \Phi)\Ophol_{((T_1a)_2^+)}(R_1\sweedl \Psi),
\end{align}
where $R$ is the $R$-matrix and $\Phi,\Psi\in\D(H)^*$.

The minimal closed ribbon (op)curves are the face loops that encircle a face of $\D(\Sigma)^*$.
Since these loops depend on a base point, i.e., a site $s(a)$, we parametrize
them with the arrow $a\in\Arr(\Sigma)$. The opcurve versions of them are
\begin{align}
\label{alpha_a}
&\text{type 0 loop around }\Ou_0(a)\colon\quad \alpha_a=\bigl((T_0a)_0^+,\bigl(T_0^2a\bigr)_0^+,\dots,(T_0^ma)_0^+\bigr),\\
\label{beta_a}
&\text{type 1 loop around }\Ou_1(a)\colon\quad \beta_a=\bigl(a_0^+,a_2^+,(T_1a)_0^+,(T_1a)_2^+\bigr),\\
\label{gamma_a}
&\text{type 2 loop around }\Ou_2(a)\colon\quad \gamma_a=\bigl(a_2^+,(T_2a)_2^+,\dots,\bigl(T_2^{n-1}a\bigr)_2^+\bigr),
\end{align}
where $m=|\Ou_0(a)|$ and $n=|\Ou_2(a)|$. The type 0 and type 2 loops are oriented counterclockwise and the type 1 loop clockwise, just like in
Figure~\ref{fig: flow}. The $\alpha_a$ and $\gamma_a$ are proper ribbons but~$\beta_a$ is maximally improper. In order to facilitate the computation of the opholonomies
of these opcurves we rewrite the formula (\ref{Ophol}) as follows:
\begin{equation*}
\Ophol_{(a_i^\sigma)}=\theta_a\ci\pi_i\ci S_{\D^*}^{(1-\sigma)/2},
\end{equation*}
where
\begin{equation*} %\label{pi0-pi2}
\pi_0,\pi_2\colon\ \D(H)^*\to \D(H)^*,\qquad
\pi_0(h\ot\varphi):=h\ot\eps\cdot\varphi(1),\qquad
\pi_2(h\ot\varphi):=\eps(h)\cdot 1\ot\varphi,
\end{equation*}
are algebra endomorphisms of the dual of the double and the
\begin{equation*}
\theta_a\colon \ H\ot H^*\to\M(\Sigma),\qquad \theta_a(h\ot\varphi):=P_a(h)Q_a(\varphi)
\end{equation*}
are linear injections.
\begin{lem}\label{lem: theta}
The maps $\theta_a$, $a\in\Arr(\Sigma)$, satisfy the following properties:
\begin{enumerate}\itemsep=0pt
\item[$(i)$] $\theta_a=\Ophol_{(a_0^+,a_2^+)}$.
\item[$(ii)$] $\theta_{T_1a}=S_\M\ci\theta_a$.
\item[$(iii)$] $\theta_a(\Phi)\theta_b(\Psi)=\theta_b(\Psi)\theta_a(\Phi)$ for all $\Phi,\Psi\in\D(H)^*$ if $\Ou_1(a)\neq\Ou_1(b)$.
\item[$(iv)$] The image of $\theta_a$ is the edge algebra $\M_e\cong H\pisharp H^*$ on which $S_\M$ acts as the
$\mathcal{L}^{-1}\ci\mathcal{R}=\mathcal{R}^{-1}\ci\mathcal{L}$ of \eqref{lambinv-rho} and \eqref{rhoinv-lamb}.
The $\theta_a$ identifies this action of $S_\M$ with the action of the antipode of $\D(H)^*$ via
\[
S_\M\ci\theta_a=\theta_a\ci S_{\D^*}.
\]
\item[$(v)$] The $\theta_a$ becomes an antialgebra map
\[
\theta_a(\Phi)\theta_a(\Psi)=\theta_a(\Psi\underset{R}{\bullet}\Phi)
\]
if we endow $\D(H)^*$, as a monoid in the braided monoidal category of right $\D(H)$-modules, with the braided-opposite multiplication
\[
\Psi\underset{R}{\bullet}\Phi:=(\Phi\sweedr R_1)(\Psi\sweedr R_2).
\]
\end{enumerate}
\end{lem}
\begin{proof}
(i) Applying (\ref{Ophol}) and using the expression (\ref{D^* comul}) for the coproduct, it is easy to see that the equation
$\Ophol_{(a_0^+,a_2^+)}(h\ot\varphi)=P_a(h)Q_a(\varphi)$ holds.

(ii) is essentially the definition of $S_\M$.

(iii) holds because $\M$ is the tensor product of the edge algebras $\M_e$ for $e\in\Sigma^1$.

(iv) can be seen by comparing formulas (\ref{lambinv-rho}), (\ref{rhoinv-lamb}) with the antipode formula (\ref{Sdudo}).

(v) The braiding on $(g\ot\psi)\ot(h\ot\varphi)\in\D(H)^*\ot\D(H)^*$ acts by
\begin{gather*}
[(h\ot\varphi)\sweedr R_1]\ot[(g\ot\psi)\sweedr R_2]
=\bigl\bra \eps\ot x_i,h\p\ot\xi_j\varphi\p\xi_k\bigr\ket\bigl(S^{-1}(x_k)h\pp x_j\ot\varphi\pp\bigr)\\ \hphantom{[(h\ot\varphi)\sweedr R_1]\ot[(g\ot\psi)\sweedr R_2]=}{}
\ot\bigl\bra \xi_i\ot 1,g\p\ot\xi_l\psi\p\xi_m\bigr\ket\bigl(S^{-1}(x_m)g\pp x_l\ot\psi\pp\bigr)\\ \hphantom{[(h\ot\varphi)\sweedr R_1]\ot[(g\ot\psi)\sweedr R_2]}{}
=\bigl(S^{-1}(x_k)hx_j\ot\varphi\pp\bigr)\ot\bigl\bra \xi_j\varphi\p\xi_k, g\p\bigr\ket\bigl(g\pp\ot\psi\bigr)\\ \hphantom{[(h\ot\varphi)\sweedr R_1]\ot[(g\ot\psi)\sweedr R_2]}{}
=\bigl(S^{-1}\bigl(g\ppp\bigr)hg\p\ot\bigl(\varphi\sweedr g\pp\bigr)\bigr)\ot\bigl(g\pppp\ot\psi\bigr),
\end{gather*}
therefore
\begin{align*}
&(g\ot\psi)\underset{R}{\bullet}(h\ot\varphi)=g\pppp S^{-1}\bigl(g\ppp\bigr)hg\p\ot\bigl(\varphi\sweedr g\pp\bigr)\psi=h\bigl(\varphi\p\sweedl g\bigr)\ot \varphi\pp\psi
\end{align*}
(even for non-involutive $H$-s) is precisely the multiplication of the smash product applied to $(h\ot\varphi)\ot(g\ot\psi)
\in(H\pisharp H^*)\ot (H\pisharp H^*)$.
\end{proof}

Part (v) of the lemma immediately implies (cf.\ \cite[Theorem 7.6]{Meusburger}).
\begin{cor}
The algebra $\M^\op$ is isomorphic to the tensor product $\bigotimes_{e\in\Sigma^1}\D(H)^*_R$, where the algebra $\D(H)^*_R$ is $\D(H)^*$ endowed with the
braided opposite multiplication.
\end{cor}

\begin{pro} \label{pro: face loop holonomies}
The opholonomy of the type $0$ and $2$ face loops $\alpha_a$ and $\gamma_a$ of~\eqref{alpha_a} and~\eqref{gamma_a} essentially reproduce the Gauss' law and
flux operators of~\eqref{Gauss} and~\eqref{flux}, respectively,
\begin{align*}
\Ophol_{\alpha_a}(h\ot\varphi)=\varphi(1)\cdot G_a(h),\qquad
\Ophol_{\gamma_a}(h\ot\varphi)=\eps(h)\cdot F_a(\varphi).
\end{align*}
The opholonomy of a type $1$ loop is
\begin{equation}\label{edge loop ophol}
\Ophol_{\beta_a}(h\ot\varphi)=\one_\M\cdot\bra R_1S_\D(R_2),h\ot\varphi\ket,
\end{equation}
where $R_1S_\D(R_2)=S_\D(u)$ with $u$ denoting the Drinfeld element $S_\D(R_2)R_1$ of the quasitriangular Hopf algebra $\bra\D(H),R\ket$.
\end{pro}
\begin{proof}
Iterating the formula $(\pi_0\ot\id)\ci\cop_{\D^*}(h\ot\varphi)=\bigl(h\p\ot\eps\bigr)\ot\bigl(h\pp\ot\varphi\bigr)$ for multiple coproducts we arrive to
\begin{align*}
\Ophol_{\alpha_a}(h\ot\varphi)&{}=\theta_{T_0a}\bigl(h\p\ot\eps\bigr)\cdots\theta_{T_0^ma}\bigl(h_{(m)}\ot\eps\bigr)\cdot\varphi(1)\\
&{}=P_{T_0a}\bigl(h\p\bigr)\cdots P_{T_0^ma}\bigl(h_{(m)}\bigr)\cdot\varphi(1)=G_a(h)\cdot \varphi(1).
\end{align*}
Similarly, formula $(\id\ot\pi_2)\ci\cop_{\D^*}(h\ot\varphi)=\bigl(h\ot\varphi\p\bigr)\ot\bigl(1\ot\varphi\pp\bigr)$ leads to
\begin{align*}
\Ophol_{\gamma_a}(h\ot\varphi)&{}=\eps(h)\cdot\theta_{a}\bigl(1\ot\varphi\p\bigr)\cdots\theta_{T_2^{n-1}a}\bigl(1\ot\varphi_{(n)}\bigr)\\
&{}=\eps(h)\cdot Q_{a}\bigl(\varphi\p\bigr)\cdots Q_{T_2^{n-1}a}\bigl(\varphi_{(n)}\bigr)=\eps(h)\cdot F_a(\varphi).
\end{align*}
Using (i), (ii) and (iv) of Lemma~\ref{lem: theta}, the opholonomy of $\beta_a$ can be written as
\[
\Ophol_{\beta_a}=\Ophol_{(a_0^+,a_2^+)}\cv\bigl(S_\M\ci\Ophol_{(a_0^+,a_2^+)}\bigr)=\theta_a\cv(\theta_a\ci S_{\D^*}),
\]
which, by (v) of the lemma, evaluates to
\begin{gather*}
\Ophol_{\beta_a}(h\ot\varphi)=\theta_a\bigl(\bigl((h\ot\varphi)\p\sweedr R_1\bigr)\bigl(S_{\D^*}\bigl((h\ot\varphi)\pp\bigr)\sweedr R_2\bigr)\bigr)\\ \hphantom{\Ophol_{\beta_a}(h\ot\varphi)}{}
=\theta_a\bigl((h\ot\varphi)\pp S_{\D^*}\bigl((h\ot\varphi)\ppp\bigr)\bigr)\bigl\bra (h\ot\varphi)\p,R_1\bigr\ket\bigl\bra S_{\D^*}\bigl((h\ot\varphi)\pppp\bigr),R_2\bigr\ket\\ \hphantom{\Ophol_{\beta_a}(h\ot\varphi)}{}
=\theta_a(1_{\D^*})\cdot \bra h\ot\varphi, R_1S_\D(R_2)\ket.
\end{gather*}
We note that $R_1S_\D(R_2)=(\eps\ot x_i)(S(\xi_i)\ot 1)=(\eps\ot S(x_i))(\xi_i\ot 1)=S_\D(R_1)R_2=S_\D(u)$.
\end{proof}

Since the only non-commuting $P$ and $Q$ operators are lying at corners of type 1 loops, the result (\ref{edge loop ophol}) can be thought of as a holonomy
version of the Heisenberg exchange relations (\ref{QP comm rel}).

Using Proposition~\ref{pro: Hol-Ophol}\,(i), the above results can be translated to holonomies of the basic loop curves $\alpha_a^{\rev -1}$, $\beta_a^{\rev -1}$
and $\gamma_a^{\rev -1}$ as follows:
\begin{align*}
&\Hol_{\alpha_a^{\rev -1}}(h\ot\eps)=S_\M(G_a(h))=P_{T_1T_0a}\bigl(h\p\bigr)P_{T_1T_0^2a}\bigl(h\pp\bigr)\cdots P_{T_1T_0^ma}\bigl(h_{(m)}\bigr),\\
&\Hol_{\beta_a^{\rev -1}}(h\ot\varphi)=\one_\M\cdot\bra R_1S_\D(R_2),h\ot\varphi\ket,\\
&\Hol_{\gamma_a^{\rev -1}}(1\ot\varphi)=S_\M(F_a(\varphi))=Q_{T_1a}\bigl(\varphi\p\bigr)Q_{T_1T_2a}\bigl(\varphi\pp\bigr)\cdots Q_{T_1T_2^{n-1}a}\bigl(\varphi_{(n)}\bigr).
\end{align*}
The curve $\alpha_a^{\rev -1}$, if viewed on $\Sigma$, consists of arrows pointing toward the vertex $v=\Ou_0(a)$ and, as a \smash{$\duP$}-curve, winds around $v$
in clockwise direction. The curve $\gamma_a^{\rev -1}$ consists of arrows on the boundary of the face $f=\Ou_2(a)$ directed in clockwise direction.
These basic loops and their holonomies fit to the conventions used in the literature~\cite{BMCA,Meusburger}. Although the automorphism $S_\M$ provides a direct connection to our conventions, the simplicity of~(\ref{Gauss}) and~(\ref{flux}) can be an excuse for our preference of opholonomies.

The type 1 loops $\beta_a$ and $\beta_a^{\rev -1}$ belong to the following type of (op)curves:
\begin{defi}
A closed (op)curve $\zeta$ on $\D(\Sigma)^*$ is called a central (op)curve if its (op)holonomy is of the form
$(h\ot\varphi)\mapsto \one_\M\cdot\bra Z,h\ot\varphi\ket$ for some central element $Z$ of $\D(H)$.
\end{defi}
Notice that by convolution invertibility of (op)holonomies the $Z$ in the above definition~is necessarily invertible.
Due to Proposition~\ref{pro: Hol-Ophol}\,(i), a proper ribbon opcurve $\zeta$ is central if and only if the proper ribbon curve $\zeta^\rev$ is central.

\begin{lem}\label{lem: central}
Let $s\in\D(\Sigma)^{*\,0}$ be a site.
\begin{enumerate}\itemsep=0pt
\item[$({\rm i})$] The central opcurves $\zeta\colon s\to s$ form a monoid under concatenation and their paths form a~group.
\item[$({\rm ii})$] If $\zeta\colon s\to s$ is a central opcurve and $\alpha\colon s\to s'$ is any opcurve, then $\zeta':=\alpha^{-1}\zeta\alpha\colon s'\to s'$ is a central opcurve, too.
Moreover, $\Ophol_{\zeta'}=\Ophol_\zeta$.
\end{enumerate}
\end{lem}
\begin{proof}
(i) If $\zeta_j\colon s\to s$ have $\Ophol_{\zeta_j}(\Psi)=\one_\M\cdot\bra Z_j,\Psi\ket$ for $j=1,2$, then
$\Ophol_{\zeta_1\zeta_2}(\Psi)=\one_\M\cdot\bra Z_1Z_2,\Psi\ket$. If $\zeta$ is central, then so is $\zeta^{-1}$.

(ii) If we can prove the statement for $\alpha$ having length 1, then the statement will follow by induction on the length of $\alpha$.
Let $\alpha=(d)$. Then
\begin{align*}
\Ophol_{\zeta'}(\Psi)&{}=\Ophol_{(d)}\bigl(S_{\D^*}\bigl(\Psi\p\bigr)\bigr)\Ophol_\zeta\bigl(\Psi\pp\bigr)\Ophol_{(d)}\bigl(\Psi\ppp\bigr)\\
&{}=\Ophol_{(d)}\bigl(S_{\D^*}\bigl(\Psi\p\bigr)\bigr)\Ophol_{(d)}\bigl(\Psi\ppp\bigr)\cdot\bigl\bra Z,\Psi\pp\bigr\ket.
\end{align*}
Evaluating $S_{\D^*}\bigl(\Psi\p\bigr)\ot\bra Z,\Psi\pp\ket\Psi\ppp$ on $Y\ot Y'\in\D(H)\ot\D(H)$, one obtains
$\bra\Psi, S_{\D}(Y)ZY'\ket=\bra\Psi, S_{\D}(Y)Y'Z\ket$, therefore
\[
S_{\D^*}\bigl(\Psi\p\bigr)\ot\bigl\bra Z,\Psi\pp\bigr\ket\Psi\ppp=S_{\D^*}\bigl(\Psi\p\bigr)\ot\Psi\pp\bigl\bra Z,\Psi\ppp\bigr\ket
\]
and
\begin{align*}
\Ophol_{\zeta'}(\Psi)&{}=\Ophol_{(d)}\bigl(S_{\D^*}\bigl(\Psi\p\bigr)\bigr)\Ophol_{(d)}\bigl(\Psi\pp\bigr)\cdot\bigl\bra Z,\Psi\ppp\bigr\ket\\
&{}=\one_\M\cdot\eps_{\D^*}\bigl(\Psi\p\bigr)\bigl\bra Z,\Psi\pp\bigr\ket=\one_\M\cdot\bra Z,\Psi\ket\\
&{}=\Ophol_{\zeta}(\Psi).
\tag*{\qed}
\end{align*}
\renewcommand{\qed}{}
\end{proof}

\begin{defi}
Let $\gamma\colon s\to s'$ be an opcurve on $\D(\Sigma)^*$. A central deformation of $\gamma$ is an opcurve $\delta\colon s\to s'$ such that $\delta\gamma^{-1}$
is central.
\end{defi}

\section{The algebra of ribbon operators}\label{section7}

In this section, we shall work with ribbon operators $M=\Ophol_\rho(\Psi)$ for some ribbon path $\rho$ and some $\Psi\in\D(H)^*$.
Ribbon operators can be multiplied either using the convolution product or the
multiplication of $\M$. The relations of the former type come from the definition of opholonomy,
\[
\Ophol_\alpha\cv\Ophol_\beta=\Ophol_{\alpha\,\beta},\qquad \alpha,\beta\in\Ribb_L^\circ \quad \text{or} \quad \alpha,\beta\in\Ribb_R^\circ.
\]
So we can, for example, change the starting point of the Gauss' law and flux operators by conjugation in the convolution sense,
\begin{align}
\label{rotate G}
&G_{T_0a}(h)=P_{T_0a}(S\bigl(h\p\bigr))G_a\bigl(h\pp\bigr)P_{T_0a}\bigl(h\ppp\bigr),\\
%\label{rotate F}
&F_{T_2a}(\varphi)=Q_a(S\bigl(\varphi\p\bigr))F_a\bigl(\varphi\pp\bigr)Q_a\bigl(\varphi\ppp\bigr).
\end{align}
As for the ordinary multiplication of $\M$, it is quite difficult to find computable relations between ribbon operators.
The rare cases which we are able to compute occupy the present Section.

If two ribbons are sufficiently distant, then the corresponding operators should commute. Thinking over which edges have $P$ or $Q$ operators that do not
commute we arrive to the following:
\begin{defi}\label{def: Usep}
For an arrow $d\in\Arr(\D(\Sigma)^*)$, let $U(d)$ be the set of arrows that are incident to~$d$ and lie on the boundary of the same type 1 face
as $d$ does. (There are 6 such arrows, including~$d$, forming a U-shape of 3 edges.) Two (op)curves, or just sets of arrows,~$\gamma$ and~$\gamma'$
in~$\D(\Sigma)^*$ are called U-separated, written $\gamma\Usep\gamma'$ if $d\nin U(d')$ for all $d\in\gamma$ and $d'\in\gamma'$.
\end{defi}
So we have
\begin{equation}\label{U-sep commute}
\Ophol_\rho(\Phi)\Ophol_\gamma(\Psi)=\Ophol_\gamma(\Psi)\Ophol_\rho(\Phi)\qquad\forall \Phi,\Psi\in\D(H)^*\quad\text{if}\ \rho\Usep\gamma.
\end{equation}

At the other extreme let $\rho=\gamma$. Then the product is computable if $d\nin U(d')$ for any pair of different arrows $d,d'\in\rho$. This property
of being ``self-U-separated'' can be seen to be equivalent to properness (Definition~\ref{def: prop ribb}) of the opcurve. Thus, \begin{align} \label{proper multiply}
\Ophol_\rho(\Phi)\Ophol_\rho(\Psi)&=\Ophol_\rho(\Phi\Psi)\qquad\text{if}\ \rho\in\Ribb^\circ_L\ \text{is proper}\\
\notag&=\Ophol_\rho(\Psi\Phi)\qquad\text{if}\ \rho\in\Ribb^\circ_R\ \text{is proper}
\end{align}
by comparing (\ref{Ophol}) with (\ref{D^* mul}) and using that the coproduct of $\D(H)^*$ is multiplicative. For proper ribbons, we also have the inversion formula
\begin{equation} \label{proper inverse}
\Ophol_{\rho^{-1}}(\Phi)=\Ophol_\rho(S_{\D^*}(\Phi))\qquad\text{if}\ \rho\ \text{is proper}
\end{equation}
and the cyclic permutation formula
\begin{equation} \label{proper cyclic}
\Ophol_{C\rho}(\Psi)=\Ophol_\rho(\Psi)\qquad\text{if}\ \rho \ \text{is closed, proper and}\ \Psi\in\Cocom\D(H)^*,
\end{equation}
where $C\rho=(d_2,\dots,d_n,d_1)$ whenever $\rho=(d_1,\dots,d_n)$.

Now consider the length 2 opcurves that are not straight, the ``elbows''.
Up to taking inverses, there are 4 such opcurves with common middle point being the site $s(a)$:
\begin{equation*}
\bigl(a_0^+,a_2^+\bigr),\qquad \bigl((T_2^{-1}a)_2^+,a_2^+\bigr),\qquad \bigl((T_2^{-1}a)_2^+,(T_0a)_0^+\bigr),\qquad \bigl(a_0^+,(T_0a)_0^+\bigr).
\end{equation*}
The following lemma is known in the context of Hopf algebra gauge theory \cite{Meusburger}.
\begin{lem}
The opholonomies of an ``opposite pair of elbows'' commute. That is to say,
\begin{align}
\label{PQ-QP elbows}
&\big[\Ophol_{(a_0^+,a_2^+)}(\Phi),\Ophol_{((T_2^{-1}a)_2^+,(T_0a)_0^+)}(\Psi)\big]=0,\\
\label{QQ-PP elbows}
&\big[\Ophol_{((T_2^{-1}a)_2^+,a_2^+)}(\Phi),\Ophol_{(a_0^+,(T_0a)_0^+)}(\Psi)\big]=0 %these are the ribbons
\end{align}
for all $\Phi,\Psi\in\D(H)^*$.
If we replace any one or both of the elbows in \eqref{PQ-QP elbows} or \eqref{QQ-PP elbows} by their inverses commutativity of their opholonomies
remains true.
\end{lem}
\begin{proof}
(\ref{PQ-QP elbows}) follows from (\ref{U-sep commute}) since the two curves are U-separated.
(\ref{QQ-PP elbows}) can be shown by checking that
\begin{align*}
&\Ophol_{((T_2^{-1}a)_2^+,a_2^+)}(g\ot\varphi)=\eps(g)\cdot Q_{T_2^{-1}a}\bigl(\varphi\p\bigr)Q_a\bigl(\varphi\pp\bigr),\\
&\Ophol_{(a_0^+,(T_0a)_0^+)}(h\ot\psi)=\psi(1)\cdot P_a\bigl(h\p\bigr) P_{T_0a}\bigl(h\pp\bigr)
\end{align*}
and then computing
\begin{gather*}
P_a\bigl(h\p\bigr) P_{T_0a}\bigl(h\pp\bigr)\cdot Q_{T_2^{-1}a}\bigl(\varphi\p\bigr)Q_a\bigl(\varphi\pp\bigr)=\\
\quad \  {}\eqby{QPbar inv-comm rel}
P_a\bigl(h\p\bigr)Q_{T_2^{-1}a}\bigl(\varphi\p\bigr)P_{T_0a}\bigl(h\pp\sweedr\varphi\pp\bigr) Q_a\bigl(\varphi\ppp\bigr)\\
\qquad {}=Q_{T_2^{-1}a}\bigl(\varphi\p\bigr)P_a\bigl(h\p\bigr)Q_a\bigl(\varphi\ppp\bigr)P_{T_0a}\bigl(h\pp\sweedr\varphi\pp\bigr)\\
\quad \  {}\eqby{QP inv-comm rel}
Q_{T_2^{-1}a}\bigl(\varphi\p\bigr)Q_a\bigl(\varphi\pppp\bigr)P_a(S\bigl(\varphi\ppp\bigr)\sweedl h\p)P_{T_0a}\bigl(h\pp\sweedr\varphi\pp\bigr)\\
\qquad {}=Q_{T_2^{-1}a}\bigl(\varphi\p\bigr)Q_a\bigl(\varphi\pp\bigr)\cdot P_a\bigl(h\p\bigr) P_{T_0a}\bigl(h\pp\bigr).
\end{gather*}
Since the elbows in (\ref{QQ-PP elbows}) are proper ribbons, replacing any one of them with its inverse amounts only to applying $S_{\D^*}$ on $\Phi$
or $\Psi$, see (\ref{proper inverse}).
Taking inverses of the elbows of (\ref{PQ-QP elbows}) requires more work,
\begin{align*}
&\Ophol_{(a_0^+,a_2^+)^{-1}}=\Ophol_{((T_1a)_0^+,(T_1a)_2^+)}\cv\Ophol_{\beta_a^{-1}},\\
&\Ophol_{((T_2^{-1}a)_2^+,(T_0a)_0^+)^{-1}}=\Ophol_{((T_1T_2^{-1}a)_2^+,(T_1T_0a)_0^+)}\cv\Ophol_{{\beta'}_b^{-1}},
\end{align*}
where $b=T_2^{-1}a$ and $\beta'_b:=\bigl(b_0^-\bigr)\beta_b\bigl(b_0^+\bigr)$. The appearance of $T_1$ in the first terms does not change U-separatedness of the curves and
the $\beta_a$ and $\beta'_b$ being central the second terms do not change commutativity, either.
\end{proof}

\begin{cor}\label{cor: elbow-loop}
Let $\lambda$ be either a type $0$ or a type $2$ face loop. Let $\epsilon=(d,b)$ be an elbow opcurve with midpoint $\nabla_1d=\nabla_0b=s$ which intersects
$\lambda$ only at the site $s$. If $s$ is not the base point of~$\lambda$, then
$\Ophol_\epsilon(\Phi)$ commutes with $\Ophol_\lambda(\Psi)$ for every $\Phi,\Psi\in\D(H)^*$.
\end{cor}

The above corollary is the first instance of a more general commutativity theorem of holo\-no\-mies of type~0 or type~2 face loops with certain ribbon
operators; expressing a sort of ``gauge invariance'' of the ribbon operator. Before formulating the statement, let us recall some Hopf algebra.

If $h\in\Cocom H$, then the iterated coproducts $h\p\ot\cdots\ot h_{(n)}$ are invariant under cyclic permutations.
In particular, the Gauss' law operators $G_a(h)$ are independent of the choice of the base point, $G_{T_0a}(h)=G_a(h)$, if $h\in\Cocom H$.
The same holds for the flux operators~$F_a(\varphi)$ if $\varphi\in\Cocom H^*$. A cocommutative element $\varphi$ defines a trace on $H$ by the
evaluation, ${\bra\varphi,hg\ket=\bra\varphi,gh\ket}$. All traces on $H$ are of this form. Since the Haar integral $\iota$ is cocommutative and
the map $H\to H^*$, $h\mapsto\iota\sweedr h$ is an isomorphism \cite{Sweedler}, every $\varphi\in\Cocom H^*$ is of the form $\varphi=\iota\sweedr z$
for a unique central element $z$ of $H$.

The relation of cocommutative elements of $H$ and $H^*$ with that of $\D(H)^*$ is intricate. Suppose we want the holonomy of the type 0 face loop $\alpha_a$
to be independent of the base point. For which elements $\Psi\in\D(H)^*$ do we have the relation $\Ophol_{\alpha_{T_0a}}(\Psi)=\Ophol_{\alpha_a}(\Psi)$?
One answer is provided by (\ref{proper cyclic}) saying that $\Psi\in\Cocom\D(H)^*$. Another answer is obtained by Proposition~\ref{pro: face loop holonomies}
and by equation~(\ref{rotate G}) saying that $\Psi=h\ot \eps$ with $h\in\Cocom H$. Which answer implies the other? None of them. The point is that the surjective
coalgebra maps
\[
H \quad \longlarr{\id\ot\eps}\quad \D(H)^*\quad\longrarr{\eps\ot\id}\quad H^*
\]
if restricted to $\Cocom\D(H)^*$ become mere linear maps
\begin{equation*}%\label{cocom 6}
\Cocom H \quad \longlarr{\id\ot\eps}\quad \Cocom\D(H)^*\quad\longrarr{\eps\ot\id}\quad \Cocom H^*,
\end{equation*}
which are no longer surjective, unless $\Cocom H$, resp.\ $\Cocom H^*$, is commutative. This is why condition~A of the theorem below contains a list of 3 possibilities.

\begin{thm}\label{thm: ribbon-loop commute}
Let $\rho\colon s_0\to s_1$ be a ribbon opcurve $($Definition~$\ref{def: ribbon})$ and let $\lambda$ be either a type $0$ face loop or a type $2$ face loop which
is disjoint from both $s_0$ and $s_1$. Then
\begin{equation*}
\Ophol_\rho(\Phi)\ \text{and }\Ophol_\lambda(\Psi)\ \text{commute for all }\Phi\in\D(H)^*%\ \text{and for }\Psi\in\Cocom\D(H)^*.
\end{equation*}
provided one of the following conditions holds:
\begin{enumerate}\itemsep=0pt
\item[A.] $\Psi$ belongs to either $\Cocom\D(H)^*$ or $\Cocom H\ot \eps$ or $1\ot\Cocom H^*$.
\item[B.] $\Psi\in\D(H)^*$ is arbitrary but the base point of $\lambda$ is not lying on $\rho$.
\end{enumerate}
\end{thm}
\begin{proof}
First we assume that $\rho$ is a left ribbon and $\lambda$ is a type 2 face loop. Let $w_\rho\in\bigl\{T_0^{-1},T_2\bigr\}^*$ be the code word of $\rho$
(Definition~\ref{def: code word}).
Let $\alpha_1,\dots,\alpha_k$ be the list of generalized subcurves of $\rho$ such that (1) all arrows and sites of $\alpha_j$ belongs to $\lambda$ and
(2) $\alpha_j$ is maximal among the subcurves satisfying (1). Mentioning sites under (1) refers to the possibility that $\alpha_j$ consists of a single
site. It is maximal if and only if none of its neighbour arrows in $\rho$ belongs to $\lambda$.
The arrow $d_j^\inn$ of~$\rho$ preceding~$\alpha_j$ and the arrow $d_j^\out$ succeeding it must belong to $U(\lambda)$ but not to $\lambda$. These arrows
are uniquely determined by the arc~$\alpha_j$ as follows. If $\alpha_j=(s)$ is trivial, then \smash{$\bigl(d_j^\inn,d_j^\out\bigr)$} is the unique elbow which touches
$\lambda$ only at the site $s$. If \smash{$\alpha_j=\bigl(c_{j,1},\dots,c_{j,l_j}\bigr)$} is non-trivial, then the \smash{$\bigl(d_j^\inn,c_{j,1}\bigr)$} and~\smash{$\bigl(c_{j,l_j},d_j^\out\bigr)$}
are straight (do not turn either left or right). The existence of~such~\smash{$d_j^{\inn/\out}$} in $\rho$ for all $j$ follows either from maximality of
the arcs $\alpha_j$ or \big(for $d_1^\inn$ and $d_k^\out$\big) from the assumption that $s_0$, $s_1$ do not lie on $\lambda$. Let \smash{$\gamma_j:=
\bigl(d_j^\inn\bigr)\alpha_j\bigl(d_j^\out\bigr)$}. Then every $\gamma_j$ corresponds to a subword $T_0^{-1}T_2^{l_j}T_0^{-1}$ of $w_\rho$ and the whole ribbon can be written
as the composite
\[
\rho=\delta_0\gamma_1\delta_1\cdots\gamma_k\delta_k
\]
of possibly empty\footnote{$\delta_j$ is empty if $\bigl(d_j^\out,d_{j+1}^\inn\bigr)$ is a type 0 loop of length~2. Such double edges in $\D(\Sigma)^*$ correspond to
a 2-valent vertex in~$\Sigma$.} curves $\delta_j$ U-separated from $\lambda$ and of ``collars'' $\gamma_j$ of $\lambda$. These collars need not be
as simple as the one of Figure~\ref{fig: collar}. They can wind around $\lambda$ arbitrary many times, although always counterclockwise. For different indices, $\gamma_i$ and~$\gamma_j$ may overlap arbitrarily or can even coincide.

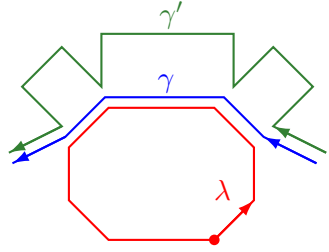
\begin{figure}[t]\centering
\parbox{200pt}{
\begin{picture}(200,100)(0,-10)
\thicklines
{\color{red}
\put(120,5){\circle*{4}}
\polyline(120,5)(135,20)(135,40)(120,55)(80,55)(65,40)(65,20)(80,5)(120,5)
\put(120,5){\vector(1,1){15}}
\put(120,20){$\lambda$}
}
{\color{blue}
%\put(155,34){\circle*{4}}
\polyline(155,34)(135,44)(120,59)(75,59)(60,44)(40,34)
\put(155,34){\vector(-2,1){20}}
\put(60,44){\vector(-2,-1){20}}
\put(95,63){$\gamma$}
}
{\color{melegzold}
%\put(155,38){\circle*{4}}
\polyline(155,38)(135,48)(150,63)(135,78)(120,63)(120,83)(70,83)(70,63)(55,78)(40,63)(55,48)(35,38)
\put(155,38){\vector(-2,1){20}}
\put(55,48){\vector(-2,-1){20}}
\put(92,87){$\gamma'$}
}
\end{picture}
}
\caption{To the proof of Theorem~\ref{thm: ribbon-loop commute}: A subcurve $\gamma$ of a left ribbon which collars the face loop $\lambda$ of type~2
and its central deformation $\gamma'$.}
\label{fig: collar}
\end{figure}

It is now clear that commutativity of the opholonomies of $\rho$ and $\lambda$ will follow from that of $\gamma$ and $\lambda$ for a single collar
$\gamma=\bigl(d^\inn,c_1,\dots,c_l,d^\out\bigr)$. A central deformation $\gamma'$ of $\gamma$ can be defined in the code word representation
(Lemma~\ref{def: code word}) by means of a homomorphism $\smash{\mathfrak{h}\colon \bigl\{T_0^{-1},T_2\bigr\}^*}\to\smash{\bigl\{T_0,T_0^{-1},T_2,T_2^{-1}\bigr\}^*}$ of free monoids as
\[
\gamma':=\bra a_0,\mathfrak{h}(w_\gamma)\ket,\qquad\text{where}\quad\mathfrak{h}\bigl(T_0^{-1}\bigr):=T_0^{-1},\ \mathfrak{h}(T_2):=T_0^{-1}T_2^{-1}T_0^{-1},
\]
where $\gamma=\bra a_0,w_\gamma\ket$. This means that $\gamma'$ differs from $\gamma$ by replacing each $c_j$ of color $T_2$ by a~detour
$\bigl(b_j,c'_j,d_j\bigr)$ of color $\mathfrak{h}(T_2)$ such that the difference $T_2\mathfrak{h}(T_2)^{-1}=T_2T_0T_2T_0$ is a type 1 loop, a conjugate of some
$\beta_a$, hence a central curve. To see that $\gamma'$ is a central deformation of~$\gamma$ introduce the opcurves $\sigma_j:=\bigl(d^\inn,c_1,\dots,c_j\bigr)$
and the lassos $\vartheta_j:=\sigma_j\beta_j\sigma_j^{-1}$ with the central curve \smash{$\beta_j=\bigl(\tilde{\mathbb{T}}_1c_j,b_j,c'_j,d_j\bigr)^{-1}$}. Then we have equality of paths
\[
\gamma{\gamma'}^{-1}=\vartheta_l\cdots\vartheta_1
\]
proving that $\gamma{\gamma'}^{-1}$ is a central opcurve by Lemma~\ref{lem: central}. Thus, \[
\Ophol_\gamma=\Ophol_{\gamma{\gamma'}^{-1}}\cv\Ophol_{\gamma'}=\one_\M\cdot\bra Z,\under\ket\cv\Ophol_{\gamma'}
\]
(for some $Z\in\Center(\D(H))$) commutes with $\Ophol_\lambda(\Psi)$ if and only if $\Ophol_{\gamma'}$ does. We write $\gamma'$ as the composite
\[
(d_0,b_1)(c'_1)(d_1,b_2)(c'_2)\cdots(c'_l)(d_l,b_{l+1}),\qquad\text{where}\quad d_0=d^\inn,\quad b_{l+1}=d^\out
\]
of elbows \smash{$\bigl(d_j,b_{j+1}\bigr)$} touching $\lambda$ at a single site and of length 1 pieces $\bigl(c'_j\bigr)$. The elbows commute with $\Ophol_\lambda(\Psi)$
because -- in case A -- the assumption on $\Psi$ ensures $\Ophol_\lambda(\Psi)$ is independent of the base point of $\lambda$ so that it can be
shifted to become different from the contact point of the elbow and Corollary~\ref{cor: elbow-loop} to apply -- and in case~B -- the base point of
$\lambda$ is different from the contact point and Corollary~\ref{cor: elbow-loop} directly applies.
We claim that $\bigl(c'_j\bigr)$ is U-separated from $\lambda$. Assume it is not. Then at least one endpoint of $c'_j$ lies
on $\lambda$ so either $b_j$ or $d_j$ connects 2 sites of $\lambda$. Let $s(a_1)$ and~$s(a_2)$ be these two sites. Since they lie on the face loop
$\lambda$, $a_1$ and $a_2$ must belong to the same $T_2$-orbit on $\Arr(\Sigma)$. But $s(a_1)$ and $s(a_2)$ are also connected by $b_j$ or $d_j$ which
have color~$T_0^{-1}$, so~$a_1$ and~$a_2$ belong to the same $T_0$-orbit. By Lemma~\ref{lem: OCPM}\,(iii), this is possible only if $a_1=a_2$ so~$b_j$
or~$d_j$ is a loop which cannot exist on $\D(\Sigma)^*$. This contradiction refutes the assumption that $\bigl(c'_j\bigr)$ is not U-separated from $\lambda$.
Hence, $\Ophol_\lambda(\Psi)$ commutes also with the $\Ophol_{(c'_j)}(\Phi)$ for all~$j$ and finally also with $\Ophol_{\gamma'}$.
This finishes the proof of the theorem for left ribbons $\rho$ and type~2 face loops $\lambda$.

If $\rho$ is a left ribbon and $\lambda$ is a type 0 face loop, then the proof goes exactly as above except that the collars $\gamma$ now, having color code
\smash{$T_2\bigl(T_0^{-1}\bigr)^lT_2$}, wind around $\lambda$ in clockwise direction. The deformation $\gamma'$ is given by the substitution
$T_0^{-1}\mapsto T_2T_0T_2$, $T_2\mapsto T_2$.

If $\rho$ is a right ribbon, then the collars go clockwise around type~2 face loops and counterclockwise around type 0 loops.
Otherwise, the proof is the same.
\end{proof}

\begin{cor}\label{cor: ribop gauge invar}
For every ribbon $\rho\colon\bra v_0,f_0\ket\to \bra v_1,f_1\ket$ and for every $\Phi\in\D(H)^*$, the operator $M=\Ophol_\rho(\Phi)$ commutes with all the
operators $A_v$ and $B_f$ except for $v\in\{v_0,v_1\}$ and ${f\in\{f_0,f_1\}}$.
\end{cor}

For a given ribbon, there are four type 0 or type 2 face loops for which the above theorem does not apply.
The next theorem is to deal with these remaining cases.

\begin{thm} \label{thm: ribb bim}
Let $\rho\colon s(a)\to s(b)$ be a proper left ribbon, where $s(a)$ and $s(b)$ are disjoint sites. Then
\begin{gather}
\label{source gauge transf}
\Ophol_\rho(\Phi)\ra{a} X=\Ophol_\rho(\Phi\sweedr X),\\
\label{target gauge transf}
X\la{b}\Ophol_\rho(\Phi)=\Ophol_\rho(X\sweedl \Phi)
\end{gather}
hold true for every $X\in\D(H)$ and $\Phi\in\D(H)^*$.
\end{thm}
\begin{proof}
The equations to be proven are equivalent to the exchange relations
\begin{gather}
\label{source gauge rel}
\Ophol_\rho(\Phi)\mathbf{D}_a(X)=\mathbf{D}_a\bigl(X\p\bigr)\Ophol_\rho\bigl(\Phi\sweedr X\pp\bigr),\\
\label{target gauge rel}
\mathbf{D}_b(X)\Ophol_\rho(\Phi)=\Ophol_\rho\bigl(X\p\sweedl \Phi\bigr)\mathbf{D}_b\bigl(X\pp\bigr).
\end{gather}

Every arrow of $\rho$ is either a $c_0^-$ or a $c_2^+$ for some $c\in\Arr(\Sigma)$. The consecutive $c_0^-$ arrows form a non-trivial arc on a type 0 face loop
and the consecutive $c_2^+$ arrows do the same on a type 2 face loop. So there is a unique decomposition $\rho=\upsilon_1\cdots\upsilon_m$ into a concatenation
of arcs which are alternatingly type 0 or type 2. \big(This decomposition corresponds to the decomposition of a word \smash{$w\in\bigl\{T_0^{-1},T_2\bigr\}^*$} into an alternating product
of powers \smash{$T_0^{-k}$ and $T_2^l$.\big)} Since $\rho$ is proper, two arcs can intersect only at some of their common endpoints and in this case the two arcs are consecutive in the
sequence $(\upsilon_1,\dots,\upsilon_m)$. The number of arcs $m\geq 2$ because we assumed that~$s(a)$ and $s(b)$ are disjoint.

The face loops $\alpha_a$, $\gamma_a$ and $\alpha_b$, $\gamma_b$ involved in the gauge transformations $\mathbf{D}_a(X)$ and $\mathbf{D}_b(X)$ fail to commute only with a
few initial and final arcs. More precisely, let $\upsilon_j\colon s_{j-1}\to s_j$ and assume that $\upsilon_1$ is type 2. Then $s_1$ does not lie on $\alpha_a$,
otherwise two different sites, $s_0$ and $s_1$, belonged to both a type 0 and a type 2 face loop which is impossible in view of Figure~\ref{fig: flow}.
So neither endpoints of $\upsilon_2\cdots\upsilon_m$ belongs to $\alpha_a$. Furthermore, $\upsilon_2\cdots\upsilon_m$ never visits the site $s(a)=s_0$ because $\upsilon_1$ does
but $\rho$ is proper, hence a simple curve by Proposition~\ref{pro: Hol-Ophol}\,(ii). It follows that
\begin{equation}\label{G_a and v2...vm if v1 is type 2}
\upsilon_1\ \text{is type 2}\quad\Rightarrow\quad G_a(H)\ \text{commutes with} \ \Ophol_{ \upsilon_2\dots\upsilon_m}(\D(H)^*)
\end{equation}
by Theorem~\ref{thm: ribbon-loop commute}\,(B).
Now assume that $\upsilon_1$ is type 0. Then $s_1$ lies on $\alpha_a$ and $s_2$ does not because~$\upsilon_2$ is type 2. Therefore, \begin{equation}\label{G_a and v3...vm if v1 is type 0}
\upsilon_1\ \text{is type 0}\quad\Rightarrow\quad G_a(H)\ \text{commutes with} \ \Ophol_{ \upsilon_3\dots\upsilon_m}(\D(H)^*)
\end{equation}
again by Theorem~\ref{thm: ribbon-loop commute}\,(B). Similar arguments yield
\begin{gather}
\label{F_a and v2...vm if v1 is type 0}
\upsilon_1\ \text{is type 0}\quad\Rightarrow\quad F_a(H^*)\ \text{commutes with} \ \Ophol_{ \upsilon_2\dots\upsilon_m}(\D(H)^*),\\
\label{F_a and v3...vm if v1 is type 2}
\upsilon_1\ \text{is type 2}\quad\Rightarrow\quad F_a(H^*)\ \text{commutes with} \ \Ophol_{ \upsilon_3\dots\upsilon_m}(\D(H)^*).
\end{gather}
Similar conclusions can be drawn for commutativity of $F_b(H^*)$ and $G_b(H)$ with $\Ophol_{\upsilon_1\dots\upsilon_{m-1}}$ or $\Ophol_{\upsilon_1\dots\upsilon_{m-2}}$
but we shall not need them.

Concentrating on the relation (\ref{source gauge rel}) we have 6 different exchange relations to compute:

\textbf{Case $\boldsymbol{G_a}$ with $\boldsymbol{\upsilon_1}$ if $\boldsymbol{\upsilon_1}$ is type 0.} Let $n=|\Ou_0(a)|$. Then the arc is $\smash{\upsilon_1=\bigl(a_0^-,\bigl(T_0^{-1}a\bigr)_0^-,\dots,}\allowbreak\smash{\bigl(T_0^{-j+1}a\bigr)_0^-\bigr)}$
for some $0<j<n$. Since $\Ophol_{\upsilon_1}(h\ot\varphi)$ is proportional to $\varphi(1)$, it suffices to work with
$\Ophol_{\upsilon_1}(h\ot\eps)=P^S_{a,\dots,T_0^{-j+1}a}(h)$ where we use the notation $P_{c_1,\dots,c_i}(h):=P_{c_1}\bigl(h\p\bigr)\cdots P_{c_i}\bigl(h_{(i)}\bigr)$ and similarly for
$P^S_c=P_c\ci S$,
\begin{align*}
P^S_{a,\dots,T_0^{-j+1}a}(h)G_a(g)&{}
=P_{T_0^{n-j+1}a,\dots,T_0^na}(S(h))P_{T_0a,\dots,T_0^{n-j}a}\bigl(g\p\bigr)P_{T_0^{n-j+1}a,\dots,T_0^na}\bigl(g\pp\bigr)\\
&{}=P_{T_0a,\dots,T_0^{n-j}a}\bigl(g\p\bigr)P_{T_0^{n-j+1}a,\dots,T_0^na}\bigl(S(h)g\pp\bigr)\\
&{}=G_a\bigl(g\p\bigr)P_{T_0^{n-j+1}a,\dots,T_0^na}\bigl(S\bigl(g\pp\bigr)S(h)g\ppp\bigr)\\
&{}=G_a\bigl(g\p\bigr)P^S_{a,\dots,T_0^{-j+1}a}\bigl(S\bigl(g\ppp\bigr)hg\pp\bigr).
\end{align*}

\textbf{Case $\boldsymbol{G_a}$ with $\boldsymbol{\upsilon_2}$ if $\boldsymbol{\upsilon_1}$ is type 0.} Given $\upsilon_1$ as in the previous case the next arc~$\upsilon_2$ starts at \smash{$s\bigl(T_0^{-j}a\bigr)$} therefore
\smash{$\upsilon_2=\bigl(\bigl(T_0^{-j}a\bigr)_2^+,\bigl(T_2T_0^{-j}a\bigr)_2^+,\dots,\bigl(T_2^{k-1}T_0^{-j}a\bigr)_2^+\bigr)$} for some \smash{$0<k<\!\big|\Ou_2\bigl(T_0^{-j}a\bigr)\big|$}. Nontriviality of the commutator of
$\Ophol_{\upsilon_2}(1\ot\varphi)=Q_{T_0^{-j}a,\dots,T_2^{k-1}T_0^{-j}a}(\varphi)$ and $G_a$ comes from the first arrow of $\upsilon_2$ and from the
$(n-j)$-th arrow of $G_a$,
\begin{gather*}
Q_{T_0^{-j}a,\dots,T_2^{k-1}T_0^{-j}a}(\varphi)G_a(g)=Q_{T_0^{-j}a}\bigl(\varphi\p\bigr)G_a(g)Q_{T_2T_0^{-j}a,\dots,T_2^{k-1}T_0^{-j}a}\bigl(\varphi\pp\bigr)\\[1mm]
 \qquad{}=P_{T_0a,\dots,T_0^{n-j-1}a}\bigl(g\p\bigr) Q_{T_0^{-j}a}\bigl(\varphi\p\bigr) P_{T_0^{n-j}a}\bigl(g\pp\bigr)\\[1mm]
 \qquad\quad{}\times P_{T_0^{n-j+1}a,\dots,T_0^na}\bigl(g\ppp\bigr)Q_{T_2T_0^{-j}a,\dots,T_2^{k-1}T_0^{-j}a}\bigl(\varphi\pp\bigr)\\[1mm]
 \quad \ \ {}\stackrel{\eqref{Bd cap Cb}}{=} P_{T_0a,\dots,T_0^{n-j-1}a}\bigl(g\p\bigr)P_{T_0^{n-j}a}(\varphi\p\sweedl g\pp)Q_{T_0^{-j}a}\bigl(\varphi\pp\bigr)\\[1mm]
 \qquad\quad{}\times P_{T_0^{n-j+1}a,\dots,T_0^na}\bigl(g\ppp\bigr)Q_{T_2T_0^{-j}a,\dots,T_2^{k-1}T_0^{-j}a}\bigl(\varphi\ppp\bigr)\\[1mm]
 \qquad{}=P_{T_0a,\dots,T_0^{n-j}a}\bigl(g\p\bigr) P_{T_0^{n-j+1}a,\dots,T_0^na}\bigl(g\ppp\bigr) Q_{T_0^{-j}a,\dots,T_2^{k-1}T_0^{-j}a}\bigl(\varphi\sweedr g\pp\bigr).
\end{gather*}
This is not yet in the desired form, there is no $G_a$ on the left-hand side, but we wish to complete it in the presence of the holonomy of $\upsilon_1$:
\begin{gather*}
 P^S_{a,\dots,T_0^{-j+1}a}(h)Q_{T_0^{-j}a,\dots,T_2^{k-1}T_0^{-j}a}(\varphi)G_a(g)\\
 =P_{T_0^{n-j+1}a,\dots,T_0^na}(S(h)) P_{T_0a,\dots,T_0^{n-j}a}\bigl(g\p\bigr)
 P_{T_0^{n-j+1}a,\dots,T_0^na}\bigl(g\ppp\bigr)Q_{T_0^{-j}a,\dots,T_2^{k-1}T_0^{-j}a}\bigl(\varphi\sweedr g\pp\bigr)\\
 =P_{T_0a,\dots,T_0^{n-j}a}\bigl(g\p\bigr)P_{T_0^{n-j+1}a,\dots,T_0^na}\bigl(S(h)g\ppp\bigr)Q_{T_0^{-j}a,\dots,T_2^{k-1}T_0^{-j}a}\bigl(\varphi\sweedr g\pp\bigr)\\
 =G_a\bigl(g\p\bigr)P_{T_0^{n-j+1}a,\dots,T_0^na}\bigl(S\bigl(g\pp\bigr)S(h)g\pppp\bigr)Q_{T_0^{-j}a,\dots,T_2^{k-1}T_0^{-j}a}\bigl(\varphi\sweedr g\ppp\bigr)\\
 =G_a\bigl(g\p\bigr)P^S_{a,\dots,T_0^{-j+1}a}\bigl(S\bigl(g\pppp\bigr)hg\pp\bigr)Q_{T_0^{-j}a,\dots,T_2^{k-1}T_0^{-j}a}\bigl(\varphi\sweedr g\ppp\bigr).
\end{gather*}
Since $\Ophol_{\upsilon_1\upsilon_2}(\Phi)=\Ophol_{\upsilon_1}(h\ot\eps)\Ophol_{\upsilon_2}(1\ot\varphi)$, where $\Phi=h\ot\varphi$ and
\begin{gather*}
S\bigl(g\ppp\bigr)hg\p\ot\varphi\sweedr g\pp=\bigl\bra h\p\ot\xi_i\varphi\p\xi_j,\eps\ot g\bigr\ket S(x_j)h\pp x_i\ot\varphi\pp=\Phi\sweedr(\eps\ot g),
\end{gather*}
we obtain $\Ophol_{\upsilon_1\upsilon_2}(\Phi)G_a(g)=G_a\bigl(g\p\bigr)\Ophol_{\upsilon_1\upsilon_2}\bigl(\Phi\sweedr\bigl(\eps\ot g\pp\bigr)\bigr)$ which, together with
(\ref{G_a and v3...vm if v1 is type 0}), gives the result
\begin{equation} \label{Hol_rho G_a}
\Ophol_\rho(\Phi)G_a(g)=G_a\bigl(g\p\bigr)\Ophol_\rho\bigl(\Phi\sweedr\bigl(\eps\ot g\pp\bigr)\bigr),
\end{equation}
whenever $\upsilon_1$ is a type 0 arc.

\textbf{Case $\boldsymbol{F_a}$ with $\boldsymbol{\upsilon_1}$ if $\boldsymbol{\upsilon_1}$ is type 0.}
For $\upsilon_1$ as above, we can write
\begin{gather*}
 P^S_{a,\dots,T_0^{-j+1}a}(h)F_a(\psi)=P_a\bigl(S\bigl(h\p\bigr)\bigr)F_a(\psi)P_{T_0^{-1}a}\bigl(S\bigl(h\pp\bigr)\bigr)\cdots P_{T_0^{-j+1}a}\bigl(S\bigl(h_{(j)}\bigr)\bigr)\\
\quad \ {}\stackrel{\eqref{rhoinv-lamb}}{=}F_a\bigl(\psi\pp\bigr)P_a\bigl(S\bigl(\psi\p\bigr)\sweedl S\bigl(h\p\bigr)\bigr)P_{T_0^{-1}a}\bigl(S\bigl(h\pp\bigr)\bigr)\cdots P_{T_0^{-j+1}a}\bigl(S\bigl(h_{(j)}\bigr)\bigr)\\
 \qquad {}=F_a\bigl(\psi\pp\bigr)P^S_{a,\dots,T_0^{-j+1}a}\bigl(h\sweedr \psi\p\bigr).
\end{gather*}
Since
\begin{gather*}
\Phi\sweedr(\psi\ot 1)=\bigl\bra\psi,h\p\bigr\ket h\pp\ot\varphi=h\sweedr\psi\ot\varphi,\\
\Ophol_{\upsilon_1}(\Phi\sweedr(\psi\ot 1))= P^S_{a,\dots,T_0^{-j+1}a}(h\sweedr \psi)\cdot\bra\varphi,1\ket,
\end{gather*}
 using also (\ref{F_a and v2...vm if v1 is type 0}), we immediately obtain
\begin{equation}\label{Hol_rho F_a}
\Ophol_\rho(\Phi)F_a(\psi)=F_a\bigl(\psi\pp\bigr)\Ophol_\rho\bigl(\Phi\sweedr\bigl(\psi\p\ot 1\bigr)\bigr),
\end{equation}
which together with (\ref{Hol_rho G_a}) proves (\ref{source gauge rel}) at least if the first arc of $\rho$ is type 0.

\textbf{Case $\boldsymbol{G_a}$ with $\boldsymbol{\upsilon_1}$ if $\boldsymbol{\upsilon_1}$ is type 2.} Then $\upsilon_1=\bigl(a_2^+,\dots,\bigl(T_2^{j-1}a\bigr)_2^+\bigr)$ for some $0<j<n$ where $n=|\Ou_2(a)|$. Thus, \begin{align*}
\Ophol_{\upsilon_1}(\Phi)G_a(g)={}&\eps(h) Q_{a,\dots,T_2^{j-1}a}(\varphi)G_a(g)\\
={}&\eps(h) Q_a\bigl(\varphi\p\bigr)G_a(g)Q_{T_2a}\bigl(\varphi\pp\bigr)\cdots Q_{T_2^{j-1}a}\bigl(\varphi_{(j)}\bigr)\\
\stackrel{\eqref{QP comm rel}}{=}{}&\eps(h) G_a\bigl(g\p\bigr)Q_a\bigl(\varphi\p\sweedr g\pp\bigr)Q_{T_2a}\bigl(\varphi\pp\bigr)\cdots Q_{T_2^{j-1}a}\bigl(\varphi_{(j)}\bigr)\\
={}&G_a\bigl(g\p\bigr) \eps(h) Q_{a,\dots,T_2^{j-1}a}\bigl(\varphi\sweedr g\pp\bigr)\\
={}&G_a\bigl(g\p\bigr)\Ophol_{\upsilon_1}\bigl(\Phi\sweedr \bigl(\eps\ot g\pp\bigr)\bigr),
\end{align*}
which, together with (\ref{G_a and v2...vm if v1 is type 2}), is sufficient to conclude that (\ref{Hol_rho G_a}) holds also if $\upsilon_1$ is a type 2 arc.

\textbf{Case $\boldsymbol{F_a}$ with $\boldsymbol{\upsilon_1}$ if $\boldsymbol{\upsilon_1}$ is type 2.} With the same $\upsilon_1$ as in the previous case, we can write
\begin{align*}
Q_{a,\dots,T_2^{j-1}a}(\varphi)F_a(\psi)&{}=Q_{a,\dots,T_2^{j-1}a}\bigl(\varphi\psi\p\bigr) Q_{T_2^ja,\dots,T_2^{n-1}a}\bigl(\psi\pp\bigr)\\
&{}=Q_{T_2^ja,\dots,T_2^{n-1}a}\bigl(\psi\pppp\bigr)Q_{a,\dots,T_2^{j-1}a}\bigl(\psi\ppp\bigr)Q_{a,\dots,T_2^{j-1}a}\bigl(S\bigl(\psi\pp\bigr)\varphi\psi\p\bigr)\\
&{}=F_a\bigl(\psi\ppp\bigr)Q_{a,\dots,T_2^{j-1}a}\bigl(S\bigl(\psi\pp\bigr)\varphi\psi\p\bigr).
\end{align*}

\textbf{Case $\boldsymbol{F_a}$ with $\boldsymbol{\upsilon_2}$ if $\boldsymbol{\upsilon_1}$ is type 2.} With $\upsilon_1$ as in the previous two cases the next arc is
\smash{$\upsilon_2=\bigl(\bigl(T_2^ja\bigr)_0^-,\bigl(T_0^{-1}T_2^ja\bigr)_0^-,\dots,\bigl(T_0^{-k+1}T_2^ja\bigr)_0^-\bigr)$}, where $0<k<\big|\Ou_0\bigl(T_2^ja\bigr)\big|$. Non-commutativity of $F_a$ and $\Ophol_{\upsilon_2}$
is due to the first arrow of $\upsilon_2$ and to the arrow $\bigl(T_2^ja\bigr)_2^+$ of $\gamma_a$,
\begin{align*}
&P^S_{T_2^ja,\dots,T_0^{-k+1}T_2^ja}(h)F_a(\psi)=P_{T_0^{-k+1}T_2^ja,\dots,T_2^ja}(S(h)) Q_{a,\dots,T_2^{j-1}a}\bigl(\psi\p\bigr)Q_{T_2^ja,\dots,T_2^{n-1}a}\bigl(\psi\pp\bigr)\\
&\qquad {}=Q_{a,\dots,T_2^{j-1}a}\bigl(\psi\p\bigr)P_{T_0^{-k+1}T_2^ja,\dots,T_0^{-1}T_2^ja}(S\bigl(h\pp\bigr))P_{T_2^ja}\bigl(S\bigl(h\p\bigr)\bigr) Q_{T_2^ja}\bigl(\psi\pp\bigr)\\
&\qquad\quad{}\times Q_{T_2^{j+1}a,\dots,T_2^{n-1}a}\bigl(\psi\ppp\bigr)\\
&\quad \ \ {}\stackrel{\eqref{PQ lin}}{=} Q_{a,\dots,T_2^{j-1}a}\bigl(\psi\p\bigr)P_{T_0^{-k+1}T_2^ja,\dots,T_0^{-1}T_2^ja}\bigl(S\bigl(h\pp\bigr)\bigr)\\
&\qquad\quad{}\times Q_{T_2^ja}\bigl(\psi\pp\sweedr S\bigl(S\bigl(h\p\bigr)\pp\bigr)\bigr)P_{T_2^ja}\bigl(S\bigl(h\p\bigr)\p\bigr) Q_{T_2^{j+1}a,\dots,T_2^{n-1}a}\bigl(\psi\ppp\bigr)\\
&\qquad {}=Q_{a,\dots,T_2^{j-1}a}\bigl(\psi\p\bigr)Q_{T_2^ja,\dots,T_2^{n-1}a}\bigl(\psi\pp\sweedr h\p\bigr) P_{T_0^{-k+1}T_2^ja,\dots,T_2^ja}\bigl(S\bigl(h\pp\bigr)\bigr).
\end{align*}
Combining this with the result of the previous case, we have
\begin{gather*}
 \Ophol_{\upsilon_1\upsilon_2}(\Phi)F_a(\psi)=\Ophol_{\upsilon_1}\bigl(h\p\ot\xi_i\varphi\p\xi_j\bigr)\Ophol_{\upsilon_2}\bigl(S(x_j)h\pp x_i\ot\varphi\pp\bigr)F_a(\psi)\\
 \qquad {}=Q_{a,\dots,T_2^{j-1}a}(\xi_i\varphi\xi_j)P^S_{T_2^ja,\dots,T_0^{-k+1}T_2^ja}(S(x_j)hx_i)F_a(\psi)\\
 \qquad {}=Q_{a,\dots,T_2^{j-1}a}(\xi_i\varphi\xi_j)Q_{a,\dots,T_2^{j-1}a}\bigl(\psi\p\bigr)Q_{T_2^ja,\dots,T_2^{n-1}a}\bigl(\psi\pp\sweedr [S(x_j)hx_i]\p\bigr) \\
 \qquad\quad{}\times P_{T_0^{-k+1}T_2^ja,\dots,T_2^ja}\bigl(S\bigl([S(x_j)hx_i]\pp\bigr)\bigr)\\
 \qquad {}=Q_{a,\dots,T_2^{j-1}a}\bigl(\xi_i\varphi\xi_j\psi\p\bigr)Q_{T_2^ja,\dots,T_2^{n-1}a}\bigl(\psi\ppp\bigr)P^S_{T_2^ja,\dots,T_0^{-k+1}T_2^ja}\bigl([S(x_j)hx_i]\sweedr \psi\pp\bigr)
\end{gather*}
The Hopf algebraic identity
\[
\xi_i\varphi\xi_j\psi\p\ot[S(x_j)hx_i]\sweedr\psi\pp=\psi\pp\xi_i\varphi\xi_j\ot S(x_j)(h\sweedr\psi\p)x_i
\]
shows that the last line above equals to
\begin{align*}
&Q_{a,\dots,T_2^{j-1}a}\bigl(\psi\pp\xi_i\varphi\xi_j\bigr)Q_{T_2^ja,\dots,T_2^{n-1}a}\bigl(\psi\ppp\bigr)P^S_{T_2^ja,\dots,T_0^{-k+1}T_2^ja}\bigl(S(x_j)\bigl(h\sweedr\psi\p\bigr)x_i\bigr)\\
& \qquad{}=F_a\bigl(\psi\pp\bigr)Q_{a,\dots,T_2^{j-1}a}(\xi_i\varphi\xi_j)P^S_{T_2^ja,\dots,T_0^{-k+1}T_2^ja}\bigl(S(x_j)\bigl(h\sweedr\psi\p\bigr)x_i\bigr)\\
& \qquad{}=F_a\bigl(\psi\pp\bigr)\Ophol_{\upsilon_1}\bigl(\bigl(h\sweedr\psi\p\bigr)\p\ot \xi_i\varphi\p\xi_j\bigr)\Ophol_{\upsilon_2}\bigl(S(x_j)\bigl(h\sweedr\psi\p\bigr)\pp x_i\ot\varphi\pp\bigr)\\
& \qquad{}=F_a\bigl(\psi\pp\bigr)\Ophol_{\upsilon_1\upsilon_2}\bigl(h\sweedr\psi\p\ot\varphi\bigr)\\
& \qquad{}=F_a\bigl(\psi\pp\bigr)\Ophol_{\upsilon_1\upsilon_2}\bigl(\Phi\sweedr\bigl(\psi\p\ot 1\bigr)\bigr).
\end{align*}
Taking into account (\ref{F_a and v3...vm if v1 is type 2}), this proves (\ref{Hol_rho F_a}) for
$\rho$ having first arc of type~2.

The above cases together prove formula (\ref{source gauge rel}) for all proper left ribbons with disjoint source and target.

Formula (\ref{target gauge rel}) requires to consider another 6 cases but there is a shorter and more instructive way to prove it.
We can apply duality to view the holonomy of the
left ribbon $\rho$ as the holonomy of a right ribbon in another Kitaev model based on the Hopf algebra $H^*$. This duality has two ingredients: an algebra isomorphism
$I\colon \M(\Sigma,H)\to\M(\Sigma^*,H^*)$ and an isomorphism $J\colon \Qv(\D(P)^\sim)\to\Qv(\D(P^\backsim)^\sim)$ of quivers.
The $I$ is defined by $I(P_a(h)):=Q^*_{T_1a}(h)$, $I(Q_a(\varphi)):=P^*_a(\varphi)$ where the $P_a^*$ and $Q_a^*$ are the standard generators of $\M(\Sigma^*,H^*)$
like the $P_a$ and $Q_a$ have been for $\M(\Sigma,H)$ in Section~\ref{section3}. Comparing equations~\eqref{QP comm rel} and~\eqref{QPbar inv-comm rel},
it is easy to see that $I$ is an algebra isomorphism.
It maps Gauss' law operators to flux operators and vice versa,
\begin{align*}
I(G_a(h))=F^*_{T^{-1}_2a}(h),\qquad
I(F_a(\varphi))=G^*_{T^{-1}_2a}(\varphi),\qquad
I(\mathbf{D}_a(X))=\mathbf{D}^*_{T^{-1}_2a}(\Upsilon(X)),
\end{align*}
where
\begin{align*}
\Upsilon\colon\ \D(H)&\to\D(H^*),\\
\varphi\ot h&\mapsto h\pp\ot\varphi\pp\cdot\bigl\bra h\ppp,\varphi\p\bigr\ket\bigl\bra S\bigl(h\p\bigr),\varphi\ppp\bigr\ket
\end{align*}
is an isomorphism of algebras and antiisomorphism of coalgebras.

In order to construct $J$, we first define a bijection $j$ between the OCPMs associated to the arrow presentations $P$ and its dual $P^\backsim$ of (\ref{iduP}) by
\begin{align*}
j\colon \ \Sigma(\mathrm{P}) \to\Sigma(\iduP),\qquad
\Ou_0(a) \mapsto\, \stackrel{\raisebox{-1pt}{\scriptsize $\backsim$}}{\Ou}_2(T_1a),\qquad
\Ou_1(a) \mapsto\,\stackrel{\raisebox{-1pt}{\scriptsize $\backsim$}}{\Ou}_1(a),\qquad
\Ou_2(a) \mapsto\, \stackrel{\raisebox{-1pt}{\scriptsize $\backsim$}}{\Ou}_0(a),
\end{align*}
which satisfies
\begin{gather*}
 j\ci\del_0= \, \stackrel{\raisebox{-1pt}{\scriptsize $\backsim$}}{d}_1,\qquad j\ci d_0=\,\stackrel{\raisebox{-1pt}{\scriptsize $\backsim$}}{\del}_0,\qquad
 j\ci\del_1=\, \stackrel{\raisebox{-1pt}{\scriptsize $\backsim$}}{d}_0,\qquad j\ci d_1=\,\stackrel{\raisebox{-1pt}{\scriptsize $\backsim$}}{\del}_1.
\end{gather*}
The definition of $J$ is this
\begin{align*}
J\colon \ \Qv(\D(\mathrm{P})^\sim) \to\Qv\bigl(\D\bigl(\iduP\bigr)^\sim\bigr),\qquad\!\!
\bra v,f\ket \mapsto\bra j(f),j(v)\ket,\qquad\!\!
a_0^\pm \mapsto(T_1a)_2^\pm,\qquad\!\!
a_2^\pm \mapsto a_0^\pm.
\end{align*}
This can be shown to be a map of quivers by reading off the definition of $\nabla_i$ from (\ref{nabla}). For example,
\begin{align*}
\stackrel{\raisebox{-1pt}{\scriptsize $\backsim$}}{\nabla}_0\ci J\bigl(a_0^+\bigr)&{}=\, \stackrel{\raisebox{-1pt}{\scriptsize $\backsim$}}{\nabla}_0(T_1a)_2^+=\bigl\bra\stackrel{\raisebox{-1pt}{\scriptsize $\backsim$}}{\del}_0T_1a,\stackrel{\raisebox{-1pt}{\scriptsize $\backsim$}}{d}_0T_1a\bigr\ket=
\bigl\bra\stackrel{\raisebox{-1pt}{\scriptsize $\backsim$}}{\del}_1a,\stackrel{\raisebox{-1pt}{\scriptsize $\backsim$}}{d}_1a\bigr\ket\\
&{}=\bra j(d_1a),j(\del_0a)\ket=J(\bra\del_0a,d_1a\ket)=J\ci\nabla_0\bigl(a_0^+\bigr).
\end{align*}
This $J$ is one of the maps that realizes the intuitively clear idea that the complexes~$\D(\Sigma)^*$ and~$\D(\Sigma^*)^*$ are the ``same'' except relabelling the arrows
and repainting the blue and red faces to red and blue, respectively.

Now $J$ induces an isomorphism $J^*\colon \Path \D(\mathrm{P})^\sim\to\Path\D(\iduP)^\sim$ of groupoids which sends a~left ribbon $\rho\colon s(a)\to s(b)$ to a right ribbon
\smash{$J^*\rho\colon \stackrel{\raisebox{-1pt}{\scriptsize $\backsim$}}{s}\!\bigl(T_2^{-1}a\bigr)\to \, \stackrel{\raisebox{-1pt}{\scriptsize $\backsim$}}{s}\!\bigl(T_2^{-1}b\bigr)$}.
\big(The appearance of~$T_2^{-1}$ is due to the fact that the site labelling function
$a\mapsto s(a)$ satisfies \smash{$J\ci s=\,\stackrel{\raisebox{-1pt}{\scriptsize $\backsim$}}{s}\ci T_2^{-1}$}. As such, it belongs to the object map of the functor $J^*$.\big)
If $\Hol^{\circ\,\backsim}$ denotes the opholonomy of the dual model, then the duality formula says that for all curves $\rho$ of the original model
\begin{equation}\label{duality formula}
I\bigl(\Ophol_\rho(\Phi)\bigr)=\Hol^{\circ\,\backsim}_{J^*\rho}(\Upsilon^*(\Phi)),
\end{equation}
where $\Upsilon^*= \bigl(\Upsilon^{\mathsf{T}}\bigr)^{-1} \colon\D(H)^*\to\D(H^*)^*$ is the mapping % RS: \,?
\[
\Upsilon^*(h\ot\varphi)=\xi_k\varphi\xi_l\ot S(x_l)hx_k,
\]
which is an antiisomorphism of algebras and an isomorphism of coalgebras.

Applying the duality formula to (\ref{source gauge rel}), we obtain
\begin{align*}
\Hol^{\circ\,\backsim}_{J^*\rho}(\Upsilon^*(\Phi))\mathbf{D}^*_{T^{-1}_2a}(\Upsilon(X))
&{}=\mathbf{D}^*_{T^{-1}_2a}\bigl(\Upsilon\bigl(X\p\bigr)\bigr)\Hol^{\circ\,\backsim}_{J^*\rho}\bigl(\Upsilon^*\bigl(\Phi\pp\bigr)\bigr)\bigl\bra\Upsilon^*\bigl(\Phi\p\bigr),\Upsilon\bigl(X\pp\bigr)\bigr\ket\\
&{}=\mathbf{D}^*_{T^{-1}_2a}\bigl(\Upsilon(X)\pp\bigr)\Hol^{\circ\,\backsim}_{J^*\rho}\bigl(\Upsilon^*(\Phi)\pp) \bra\Upsilon^*(\Phi)\p,\Upsilon(X)\p\bigr\ket\\
&{}=\mathbf{D}^*_{T^{-1}_2a}\bigl(\Upsilon(X)\pp\bigr)\Hol^{\circ\,\backsim}_{J^*\rho}\bigl(\Upsilon^*(\Phi)\sweedr \Upsilon(X)\p\bigr).
\end{align*}
Therefore, \begin{equation*}
\Hol^{\circ\,\backsim}_{\gamma}(\Psi)\mathbf{D}^*_a(Y)=\mathbf{D}^*_a\bigl(Y\pp\bigr)\Hol^{\circ\,\backsim}_{\gamma}\bigl(\Psi\sweedr Y\p\bigr)
\end{equation*}
for all proper right ribbons $\gamma\colon \stackrel{\raisebox{-1pt}{\scriptsize $\backsim$}}{s}\!(a)\to$$\stackrel{\raisebox{-1pt}{\scriptsize $\backsim$}}{s}\!(b)$, for $Y\in\D(H^*)$ and for $\Psi\in\D(H^*)^*$. For the left ribbon
$\gamma^{-1}$, this implies
\begin{align*}
\mathbf{D}^*_a(Y)\Hol^{\circ\,\backsim}_{\gamma^{-1}}(\Psi)
&{}=\Hol^{\circ\,\backsim}_{\gamma^{-1}}\bigl(\Psi\p\bigr)\Hol^{\circ\,\backsim}_{\gamma}\bigl(\Psi\pp\bigr)\mathbf{D}^*_a(Y)\Hol^{\circ\,\backsim}_{\gamma^{-1}}\bigl(\Psi\ppp\bigr)\\
&{}=\Hol^{\circ\,\backsim}_{\gamma^{-1}}\bigl(\Psi\p\bigr)\mathbf{D}^*_a\bigl(Y\pp\bigr)\Hol^{\circ\,\backsim}_{\gamma}\bigl(\Psi\pp\sweedr Y\p\bigr)\Hol^{\circ\,\backsim}_{\gamma^{-1}}\bigl(\Psi\ppp\bigr)\\
&{}=\Hol^{\circ\,\backsim}_{\gamma^{-1}}\bigl(Y\p\sweedl \Psi\bigr)\mathbf{D}^*_a\bigl(Y\pp\bigr).
\end{align*}
Since this is true for all left ribbons $\gamma^{-1}$ with target $\stackrel{\raisebox{-1pt}{\scriptsize $\backsim$}}{s}\!(a)$ not only on $\D(\Sigma^*)^*$ and not only for~$H^*$ but on the dual
of the double of any OCPM $\Sigma$ and for any involutive f.d.\ Hopf algebra~$H$, we can apply it to the original $\Sigma$ and $H$ and we get the proof of
(\ref{target gauge rel}).
\end{proof}

For sake of completeness, we note here the formulas that hold instead of (\ref{source gauge transf}) and (\ref{target gauge transf})
if $\rho\colon s(a)\to s(b)$ is a proper right ribbon,
\begin{gather}
%\label{source gauge transf - right ribb}
\Ophol_\rho(\Phi)\opra{a} X=\Ophol_\rho(\Phi\sweedr X),\\
%\label{target gauge transf - right ribb}
X\opla{b}\Ophol_\rho(\Phi)=\Ophol_\rho(X\sweedl \Phi).
\end{gather}

\begin{cor}%\label{cor: ribb bim}
Let $s(a)$ and $s(b)$ be disjoint sites and $\rho\colon s(a)\to s(b)$ be a proper left ribbon.
Then the space $\Ophol_\rho(\D(H)^*)$ of ribbon operators over $\rho$ equipped with
the actions $\la{b}$ and $\ra{a}$ is a~bimodule over $\D(H)$ such that
\[
\Ophol_\rho\colon\ \D(H)^*\to\Ophol_\rho(\D(H)^*)
\]
is a homomorphism of $\D(H)$-$\D(H)$ bimodules if we consider $\D(H)^*$ as the regular bimodule $\bra\D(H)^*,\sweedl,\sweedr\ket$.
This map is also a homomorphism of left $\D(H)$-module algebras and of right $\D(H)$-module algebras.

Similar holds for proper right ribbons $\rho\colon s(a)\to s(b)$ but then the bimodule structure is given by $\opla{b}$ and $\opra{a}$
and the algebra structure by $\D(H)^{*\,\op}$.
\end{cor}
\begin{proof}
That $\Ophol_\rho(\D(H)^*)$ is a left and right $\D(H)$-module and that the map is a left and right $\D(H)$-module map follow from
Theorem~\ref{thm: ribb bim}. Due to disjointness of the source and target sites, the left and right actions commute so
$\Ophol_\rho(\D(H)^*)$ is a bimodule. The last sentence follows from (\ref{proper multiply}) which shows
that $\Ophol_\rho(\D(H)^*)\subset\M$ is a left and right module subalgebra.
\end{proof}

Similar results can be obtained \cite{Cowtan-Majid} in a context of non-semisimple Hopf algebras.

We close the section with two more results on closed ribbon operators.
An opcurve $\gamma\colon s_0\to s_1$ is called closed if $s_0=s_1$.
If $\gamma=(d_1,\dots,d_n)$ is closed, then every cyclic permutation
$C^k\gamma=(d_{k+1},\dots,d_n,d_1,\dots,d_k)$ is also closed. It follows from the code word representation of ribbons that if the closed curve $\gamma$ happens to be a ribbon then the $C^k\gamma$ are also ribbons of the same type (left/right). Similarly, the cyclic permutations of proper closed ribbons are again proper as it is clear from
Definition~\ref{def: prop ribb}.

\begin{pro}
Let $\delta\colon s(a)\to s(a)$ and $\rho\colon s(a)\to s(a)$ be closed left ribbons of the following form \big(for $\Usep$ see Definition~$\ref{def: Usep}$\big):
\begin{gather*}
\delta=\bigl(a_0^-\bigr)\delta'\bigl(\bigl(T_2^{-1}a\bigr)_2^+\bigr),\qquad\text{where}\quad\delta'\Usep\rho,\\
\rho=\bigl(a_2^+\bigr)\rho'\bigl((T_0a)_0^-\bigr),\qquad\text{where}\quad\rho'\Usep\lambda,
\end{gather*}
so that the only intersection of $\delta$ and $\rho$ is the site $s(a)$. $($This is possible if the genus of $[\Sigma]$ is at least $1.)$ Then
\begin{equation}\label{longi-merid}
\Ophol_\delta(\Psi)\Ophol_\rho(\Phi)=\Ophol_\rho(R_{2'}\sweedl \Phi\sweedr R_1)\Ophol_\delta(R_{1'}\sweedl\Psi\sweedr R_2)
\end{equation}
for all $\Phi,\Psi\in\D(H)^*$ where $R=R_1\ot R_2\equiv R_{1'}\ot R_{2'}$ is the $R$-matrix of $\D(H)$.
\end{pro}
\begin{proof}
For any curve $\gamma$ we abbreviate $\Ophol_\gamma$ by $H_\gamma$.
The strategy of the proof is similar to that of Lemma~\ref{lem: doubles in M}. We decompose the opholonomies into 3 terms,
\begin{align*}
&H_\delta(\Psi)=H_{(a_0^-)}\bigl(\Psi\p\bigr)\,H_{\delta'}\bigl(\Psi\pp\bigr)\,H_{((T_2^{-1}a)_2^+)}\bigl(\Psi\ppp\bigr),\\
&H_\rho(\Phi)=H_{(a_2^+)}\bigl(\Phi\p\bigr)\,H_{\rho'}\bigl(\Phi\pp\bigr)\,H_{((T_0a)_0^-)}\bigl(\Phi\ppp\bigr)
\end{align*}
with the middle terms commuting with all terms of the other line. Then we exchange~$H_\delta$ and~$H_\rho$ in such a way that the order of terms
of the same line remains intact in every step. Using the identities (\ref{ophol exch rel1}) and (\ref{ophol exch rel2}), we can write
\begin{align*}
H_\delta(\Psi)H_\rho(\Phi)&{}=H_{(a_0^-)}\bigl(\Psi\p\bigr)H_{(a_2^+)}\bigl(\Phi\p\bigr)H_{\rho'}\bigl(\Phi\pp\bigr)\\
&\quad{}\cdot H_{\delta'}\bigl(\Psi\pp\bigr)H_{((T_2^{-1}a)_2^+)}\bigl(\Psi\ppp\bigr)H_{((T_0a)_0^-)}\bigl(\Phi\ppp\bigr)\\
&{}=H_{(a_2^+)}\bigl(\Phi\p\sweedr R_1\bigr)H_{(a_0^-)}\bigl(\Psi\p\sweedr R_2\bigr) H_{\rho'}\bigl(\Phi\pp\bigr)H_{\delta'}\bigl(\Psi\pp\bigr)\\
&\quad{}\cdot H_{((T_0a)_0^-)}\bigl(R_{2'}\sweedl\Phi\ppp\bigr) H_{((T_2^{-1}a)_2^+)}\bigl(R_{1'}\sweedl\Psi\ppp\bigr)\\
&{}=H_{(a_2^+)}\bigl(\Phi\p\sweedr R_1\bigr) H_{\rho'}\bigl(\Phi\pp\bigr) H_{((T_0a)_0^-)}\bigl(R_{2'}\sweedl\Phi\ppp\bigr)\\
&\quad{}\cdot H_{(a_0^-)}\bigl(\Psi\p\sweedr R_2\bigr)H_{\delta'}\bigl(\Psi\pp\bigr) H_{((T_2^{-1}a)_2^+)}\bigl(R_{1'}\sweedl\Psi\ppp\bigr)\\
&{}=H_\rho\bigl(R_{2'}\sweedl \Phi\sweedr R_1\bigr) H_\delta\bigl(R_{1'}\sweedl\Psi\sweedr R_2\bigr),
\end{align*}
which completes the proof of (\ref{longi-merid}).
\end{proof}

A closed ribbon version of (\ref{source gauge transf}) and (\ref{target gauge transf}) can be obtained as follows. First, we need a lemma.
\begin{lem} \label{lem: closed rib=rib rib}
Let $\rho\colon s(a)\to s(a)$ be a non-empty closed proper left ribbon which is not a face loop. Then there exist proper left ribbons $\rho_1\colon s(a)\to s(b)$ and
$\rho_2\colon s(b)\to s(a)$ with disjoint $s(a)$ and $s(b)$ such that $\rho=\rho_1\rho_2$.
\end{lem}
\begin{proof}
Let $\rho=\upsilon_1\cdots\upsilon_m$ be the arc decomposition as in the prof of Theorem~\ref{thm: ribb bim}. Since $\rho$ is not a face loop, $m\geq 2$. Let
$s(a_0)\rarr{\upsilon_1}s(a_1)\rarr{\upsilon_2}s(a_2)$, where $a_0=a$. Then $s(a_1)$ is neighbour to~$s(a_0)$ and~$s(a_2)$ is disjoint from $s(a_0)$. Similarly, the sites of
\smash{$s(a_{m-2})\rarr{\upsilon_{m-1}}s(a_{m-1})\rarr{\upsilon_m}s(a_m)$}, where $a_m=a$, satisfy that $s(a_{m-1})$ is neighbour to $s(a_m)$ and $s(a_{m-2})$ is disjoint from
$s(a_m)$. This already excludes $m=2$ and $m=3$. So we can take $\rho_1=\upsilon_1\upsilon_2$ and $\rho_2=\upsilon_3\cdots\upsilon_m$.
\end{proof}

\begin{pro}
Let $\rho\colon s(a)\to s(a)$ be a non-empty closed proper left ribbon. Then
\begin{equation}
\label{D(X) on closed ribbon}
\mathbf{D}_a(X)\Hol_\rho(\Phi)=\Hol_\rho\bigl(X\p\sweedl\Phi\sweedr S_\D\bigl(X\ppp\bigr)\bigr)\mathbf{D}_a\bigl(X\pp\bigr)
\end{equation}
holds true for every $X\in\D(H)$ and $\Phi\in\D(H)^*$.
\end{pro}
\begin{proof}
First, we assume that $\rho$ is not a face loop.
Take a decomposition
\[
\rho=s(a)\rarr{\rho_1}s(b)\rarr{\rho_2} s(a)
\]
 provided by Lemma~\ref{lem: closed rib=rib rib} with disjoint sites $s(a)$ and $s(b)$.
Using Theorem~\ref{thm: ribb bim},{\samepage
\begin{align*}
\mathbf{D}_a(X)\Hol_\rho(\Phi)&{}=\mathbf{D}_a(X)\Hol_{\rho_1}\bigl(\Phi\p\bigr)\Hol_{\rho_2}\bigl(\Phi\pp\bigr)\\
&{}=\Hol_{\rho_1}\bigl(\Phi\p\sweedr S_\D\bigl(X\pp\bigr)\bigr)\mathbf{D}_a\bigl(X\p\bigr)\Hol_\rho\bigl(\Phi\pp\bigr)\\
&{}=\Hol_{\rho_1}\bigl(\Phi\p\sweedr S_\D\bigl(X\ppp\bigr)\bigr)\Hol_\rho\bigl(X\p\sweedl\Phi\pp\bigr)\mathbf{D}_a\bigl(X\pp\bigr)\\
&{}=\Hol_\rho\bigl(X\p\sweedl\Phi\sweedr S_\D\bigl(X\ppp\bigr)\bigr)\mathbf{D}_a\bigl(X\pp\bigr)
\end{align*}
the statement is proven.}

If $\rho$ is a face loop, then $\rho$ is either $\gamma_a$ or $\alpha^{-1}_a$. Assume $\rho=\gamma_a$. Then by Proposition~\ref{pro: face loop holonomies}
\[
\Hol_\rho(h\ot\varphi)=\eps(h)F_a(\varphi)=\mathbf{D}_a(\varphi\ot 1)\,\eps(h),
\]
so the left-hand side of (\ref{D(X) on closed ribbon}) with $X=\psi\ot g$ and $\Phi=h\ot\varphi$ can be written as
\begin{gather*}
\mathbf{D}_a(X)\Hol_\rho(\Phi)=D_a(X(\varphi\ot 1)\eps(h)\\ \hphantom{\mathbf{D}_a(X)\Hol_\rho(\Phi)}{}
=\mathbf{D}_a\bigl(\psi\varphi\pp\ot g\pp\bigr)\cdot\bigl\bra\varphi\ppp,g\p\bigr\ket \bigl\bra\varphi\p,S\bigl(g\ppp\bigr)\bigr\ket\eps(h).
\end{gather*}
For the right-hand side, a much longer computation is necessary,
\begin{gather*}
 \Hol_\rho\bigl(X\p\sweedl\Phi\sweedr S_\D\bigl(X\ppp\bigr)\bigr)\mathbf{D}_a\bigl(X\pp\bigr)\\
\qquad{}= \Hol_\rho\bigl(\big[\underset{h_{ji}}{\underbrace{S(x_j)h\pp x_i}} \ot\varphi\pp\big]\p\bigr)\bigl\bra\big[S(x_j)h\pp x_i \ot\varphi\pp\big]\pp,X\p\bigl\ket\\
\qquad\quad{}\times \bigl\bra h\p\ot\xi_i\varphi\p\xi_j,S_\D\bigl(X\ppp\bigr)\bigl\ket\mathbf{D}_a\bigl(X\pp\bigr)\\
\qquad{}= \Hol_\rho\bigl(h_{ji(1)}\ot\xi_k\varphi\pp\xi_l\bigr)\bigl\bra S(x_l)h_{ji(2)}x_k\ot\varphi\ppp,X\p\bigl\ket\\
\qquad\quad{}\times \bigl\bra h\p\ot\xi_i\varphi\p\xi_j,S_\D\bigl(X\ppp\bigr)\bigl\ket\mathbf{D}_a\bigl(X\pp\bigr) \\
\qquad{}= \mathbf{D}_a\bigl(\xi_k\varphi\pp\xi_l\ot 1\bigr)\bigl\bra S(x_l)h_{ji}x_k\ot\varphi\ppp,\psi\ppp\ot g\p\bigl\ket\\
\qquad\quad{}\times \bigl\bra x_n S\bigl(h\p\bigr)S(x_m)\ot\xi_mS(\xi_j)S\bigl(\varphi\p\bigr)S(\xi_i)\xi_n,\psi\p\ot g\ppp\bigl\ket\mathbf{D}_a\bigl(\psi\pp\ot g\pp\bigr)\\
\qquad{}= \bigl\bra\psi\ppp,S(x_l)S(x_j)h\pp x_ix_k\bigl\ket\bigl\bra\varphi\ppp,g\p\bigl\ket\bigl\bra\psi\p,x_nS\bigl(h\p\bigr)S(x_m)\bigl\ket\\
\qquad\quad{}\times \bigl\bra\xi_mS(\xi_j)S\bigl(\varphi\p\bigr)S(\xi_i)\xi_n,g\ppp\bigl\ket\mathbf{D}_a\bigl(\xi_k\varphi\pp\xi_l\psi\pp\ot g\pp\bigr)\\
\qquad {}= \bigl\bra\psi_{(7)},h\pp\bigl\ket\bigl\bra\varphi\ppp,g\p\bigl\ket\bigl\bra\psi\pp,S\bigl(h\p\bigr)\bigl\ket\\
\qquad\quad{}\times \bigl\bra S\bigl(\psi\ppp\bigr)\psi_{(6)}S\bigl(\varphi\p\bigr)S\bigl(\psi_{(8)}\bigr)\psi\p,g\ppp\bigl\ket\mathbf{D}_a\bigl(\psi_{(9)}\varphi\pp S\bigl(\psi\ppppp\bigr)\psi\pppp\ot g\pp\bigr)\\
\qquad {}= \eps(h)\bigl\bra\varphi\ppp,g\p\bigl\ket\bigl\bra S\bigl(\varphi\p\bigr),g\ppp\bigl\ket \mathbf{D}_a\bigl(\psi\varphi\pp\ot g\pp\bigr).
\end{gather*}
The case of $\rho=\alpha^{-1}_a$ can be proven by a similar calculation.
\end{proof}

Note that (\ref{D(X) on closed ribbon}) is identical to the multiplication rule in the double of the double, $\D(\D(H))$.

\section{Homotopy}\label{section8}

In order to explain ``topological invariance'' of the Kitaev model, the first step is to look for an appropriate notion of homotopy of curves.
Eventually, we want to use homotopy on the complex~$\D(\Sigma)^*$ but the very notion should be given for an arbitrary $\Sigma$.
\begin{defi}\label{def: htp}
Let $\Sigma$ be a connected OCPM with arrow presentation $\bra \Arr, T_0,T_2\ket$.
%Let $\Sigma\,{::}\,P$ be a connected OCPM with a fixed arrow presentation
%\[
%P=\bra\Arr,T_0,T_2\ket.
%\]
\begin{enumerate}\itemsep=0pt
\item Let $f$ be a face of $\Sigma$. An $f$-loop $\lambda$ is a closed opcurve $\bigl(a,T_2a,T_2^2a,\dots\bigr)$ or its inverse where $a\in\Arr$ is such that
$\Ou_2(a)=f$.
The set of $f$-loops for all $f\in\Sigma^2$ are called face loops.
\item A lasso $\vartheta$ is a conjugate of a face loop, i.e., $\vartheta=\sigma\lambda\sigma^{-1}$, where $\lambda$ is a face loop and $\sigma$ is
arbitrary, provided $\del_1\sigma=\del_0\lambda$.
If $\lambda$ is an $f$-loop then $\vartheta$ is called an $f$-lasso.
\item Let $\alpha$, $\beta$ be parallel opcurves, i.e., $\del_i\alpha=\del_i\beta$, $i=0,1$. A (combinatorial) homotopy from~$\alpha$ to~$\beta$ is a
sequence of lassos
$(\vartheta_N,\dots,\vartheta_1)$ such that
\[
\beta\vartheta_N\dots\vartheta_2\vartheta_1=\alpha\qquad\text{as paths.}
\]
Two opcurves are called homotopic, denoted by $\alpha\sim\beta$, if they are parallel and there is a~homotopy from one of them to the other.
\item A contraction of a closed curve $\alpha\colon v\to v$ is a homotopy from $\alpha$ to the trivial curve $(v)$. A~closed curve is contractible if has a contraction.
\item If $(\vartheta_N,\dots,\vartheta_1)\colon \alpha\to\beta$ is a homotopy with $\vartheta_i$ being an $f_i$-lasso, then the support of this homotopy is the set of faces occurring in the list~$(f_i)$.
\end{enumerate}
\end{defi}

\begin{rmk}
There are op-versions of the above definition of homotopy of opcurves:
\begin{itemize}\itemsep=0pt
\item In ophomotopy of opcurves $\alpha\to\beta$ means $\alpha\vartheta_1\cdots\vartheta_N=\beta$.
\item In homotopy of curves $\alpha\to\beta$ means $\vartheta_N\ci\cdots\ci\vartheta_1\ci\alpha=\beta$.
\item In ophomotopy of curves $\alpha\to\beta$ means $\vartheta_1\ci\cdots\ci\vartheta_N\ci\beta=\alpha$.
\end{itemize}
\end{rmk}

\begin{sch}
Let $\Sigma$ be a connected OCPM. Homotopy of opcurves on $\Sigma$ satisfies the following elementary properties:
\begin{enumerate}\itemsep=0pt
\item If $\boldsymbol{\vartheta}=(\vartheta_N,\dots,\vartheta_1)$ is a homotopy $\alpha\to\beta$ then
\smash{$\boldsymbol{\vartheta}^{-1}=\bigl(\vartheta_1^{-1},\dots,\vartheta_N^{-1}\bigr)$} is a homotopy $\beta\to\alpha$.
\item If $\alpha\rarr{\boldsymbol{\vartheta}}\beta\rarr{\boldsymbol{\vartheta}'}\gamma$ are homotopies, then the composition
\[
\boldsymbol{\vartheta}'\ci\boldsymbol{\vartheta}:=\bigl(\vartheta'_M,\dots,\vartheta'_1,\vartheta_N,\dots,\vartheta_1\bigr)
\]
is a homotopy $\alpha\to\gamma$.
\item If two opcurves are equal as paths, then they are homotopic.
\item Homotopy is an equivalence relation on the set of opcurves and each path is the subset of a unique homotopy class.
\item A homotopy $\alpha\to\beta$ is the same as a contraction of $\beta^{-1}\alpha$.
\item If $\boldsymbol{\vartheta}\colon\alpha\to\beta$ is a homotopy, then $\bigl(\beta\vartheta_N\beta^{-1},\dots\beta\vartheta_1\beta^{-1}\bigr)$ is a
homotopy $\beta^{-1}\to\alpha^{-1}$.
\item Let
\[
u\rarr{\alpha}v\pair{\beta}{\gamma}w\rarr{\delta}z
\]
 be opcurves and let $\boldsymbol{\vartheta}\colon \beta\to\gamma$ be a homotopy. Then
 $\bigl(\dots,\delta^{-1}\vartheta_i\delta,\dots\bigr)$ is a homotopy $\alpha\beta\delta\to\alpha\gamma\delta$.
\item If $(\vartheta_N,\dots,\vartheta_1)\colon\alpha\to\beta$ is a homotopy and $p\in S_N$ is any permutation, then there exist closed curves $\pi_i$ such that
$\vartheta'_i:=\pi_i\vartheta_{p(i)}\pi_i^{-1}$ defines a homotopy $(\vartheta'_N,\dots,\vartheta'_1)\colon\alpha\to\beta$.
\end{enumerate}
\end{sch}

The set of homotopy classes of opcurves together with concatenation is a groupoid which is the opposite of the traditional fundamental groupoid
$\pi_1(\Sigma)$. It is evident from Definition~\ref{def: htp}, although its formalization could be tedious, that this groupoid is the same as the
fundamental groupoid $\pi_1([\Sigma])$ of the geometric realization of the complex $\Sigma$.

The notion of contraction defined above does not reflect the idea that a contractible curve bounds a disk and the support of the contraction is the disk itself. As a matter of fact, any finite set of faces can be the support of a contraction. So one should study simple curves and their simple homotopies.
But before we would dig too deeply into the subject of lasso homotopy we should %glance at
remind ourselves of our main goal which is the algebra of ribbon operators on~$\D(\Sigma)^*$.
It turns out that the use of lassos in holonomy operators is extremely inconvenient. The lassos~$\sigma\lambda\sigma^{-1}$ are not proper
ribbons and the long tails $\sigma$ introduce serious non-commutativities which practically obstructs any computation.
We need some other kind of homotopy specialized to ribbon curves on $\D(\Sigma)^*$ by means of which we can deform (proper) ribbons in such a way that the
curve stays a (proper) ribbon in each step.

Let us introduce the following open ``lassos'' (Figure~\ref{fig: kappa-lambda}):
\begin{align}
\label{kappa_a}
&\kappa_a:=\bigl((T_1a)_2^+,\bigl(T_0^{m-1}a\bigr)_0^-,\dots,(T_0a)_0^-,a_2^+\bigr)\colon \ s(T_1a)\to s(T_2a),\\
\label{lambda_a}
&\lambda_a:=\bigl((T_1a)_0^-,(T_2a)_2^+,\dots,\bigl(T_2^{n-1}a\bigr)_2^+,a_0^-\bigr)\colon \ s(T_1a)\to s\bigl(T_0^{-1}a\bigr),
\end{align}
where $m=|\Ou_0(a)|$ and $n=|\Ou_2(a)|$. Note that $\kappa_a,\lambda_a\in\Ribb^\circ_L$.

\begin{figure}[t]\centering
\parbox{200pt}{
\begin{picture}(160,140)(40,20)
\thicklines
{\color{blue}\polyline(160,40)(120,40)(90,20)(40,50)(80,100)(120,80)(160,80)
\put(162,40){\vector(-1,0){40}}
\put(122,80){\vector(1,0){40}}
\put(90,60){\circle*{5}}
\put(50,50){$\kappa_a$}}
{\color{red}\polyline(160,40)(160,80)(200,110)(190,150)(140,155)(100,120)(120,80)(120,40)
\put(160,40){\vector(0,1){40}}
\put(120,80){\vector(0,-1){40}}
\put(150,120){\circle*{5}}
\put(180,120){$\lambda_a$}}
\put(120,80){\circle*{5}}\put(123,68){$s(a)$}
{\color{melegzold}\put(140,60){\circle*{5}}}
\end{picture}
}
\caption{The basic curves of ribbon homotopy.}
\label{fig: kappa-lambda}
\end{figure}
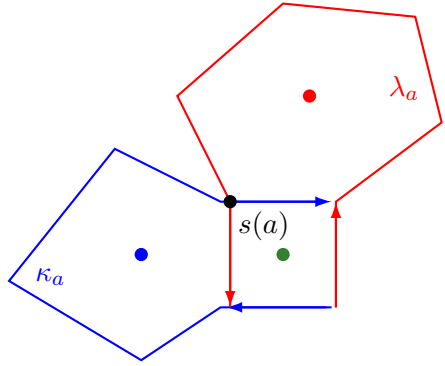

Unlike the lassos the $\kappa$-s and $\lambda$-s cannot be nicely
composed. So the strategy in ribbon curve homotopy is not composing the elementary paths, it is more like the procedure of grafting trees in which
$\kappa_a$ and $\lambda_a$ play the role of the grafts.

\begin{rmk}
Speaking about subcurves and subpaths, in general, requires some care. For opcurves~$\sigma$ and~$\rho$, the relation $\sigma\subset\rho$ means the
existence of opcurves $\rho_1$ and $\rho_2$ such that $\rho=\rho_1\sigma\rho_2$ as opcurves, hence no reduction is allowed. For paths of opcurves,
$\sigma\subset\rho$ means the existence of paths $\rho_1$ and $\rho_2$ such that $\rho=\rho_1\sigma\rho_2$ as paths.
In case of both $\sigma$ and $\rho$ are ribbons (of the same type), the relation $\sigma\subset\rho$ for the ribbon opcurves is stronger than the
same relation for their paths. However, if we think more categorically and utilize the fact that $\Ribb^\circ_{L/R}$ is a~category for concatenation, then
we can use the notion of subribbon as a third possibility. For left ribbons, e.g., $\sigma\subset\rho$ is a subribbon if there exist left ribbons
$\rho_1$ and $\rho_2$ such that $\rho=\rho_1\sigma\rho_2$ either as opcurves or as paths because concatenation of left ribbons is always reduced
so these two alternatives are equivalent.
This third version of ``subcurve'' is which we need in the next definition.
\end{rmk}

\begin{defi}\quad%\label{def: ribbon curve homotopy}
\begin{enumerate}\itemsep=0pt
\item[(a)] For left ribbon opcurves $\rho,\rho'\colon s_0\to s_1$, we say that the pair $(\rho,\rho')$ is an elementary contraction and that $(\rho',\rho)$
is an elementary relaxation in the following 4 cases:
\begin{itemize}\itemsep=0pt
\item Vertex contraction: For some $a\in\Arr(\Sigma)$ the $\rho$ contains a subribbon $\alpha_a^{-1}$ (see (\ref{alpha_a})) and $\rho'$ is
obtained from $\rho$ by discarding this $\alpha_a^{-1}$.
\item Face contraction: For some $a\in\Arr(\Sigma)$, the $\rho$ contains a subribbon $\gamma_a$ (see (\ref{gamma_a})) and $\rho'$ is
obtained from $\rho$ by discarding this $\gamma_a$.
\item $\kappa$-contraction: For some $a\in\Arr(\Sigma)$, the $\rho$ contains a subribbon $\kappa_a\subset\rho$ and $\rho'$ is obtained by replacing this
subribbon by $\bigl((T_1a)_0^-\bigr)$.
\item $\lambda$-contraction: For some $a\in\Arr(\Sigma)$, the $\rho$ contains a subribbon $\lambda_a\subset\rho$ and $\rho'$ is obtained by replacing this
subribbon by $\bigl((T_1a)_2^+\bigr)$.
\end{itemize}
For $\rho,\rho'\in\Ribb^\circ_L$, we write $\rho\approx\rho'$ and say that they are ribbon-homotopic if there is a~sequence of left ribbons $(\rho_0,\dots,\rho_k)$,
the homotopy, such that $\rho_0=\rho$, $\rho_k=\rho'$ and each move $(\rho_i,\rho_{i+1})$ is an elementary contraction or an elementary relaxation.

\item[(b)] For closed ribbons, we also define free ribbon-homotopy as a ribbon-homotopy in which we allow also circular moves: $(\rho_i,\rho_{i+1})$
is a circular move if $\rho_{i+1}$ differs from $\rho_i$ only in the choice of the base point. Free ribbon-homotopy is denoted by
$\overset{\star}{\approx}$.

\item[(c)] For right ribbons, the definitions (a) and (b) can be used without change except that in the definition of moves the curves $\alpha_a^{-1}$, $\gamma_a$, $\kappa_a$, $\lambda_a$ have to be replaced by~$\alpha_a$,~$\gamma_a^{-1}$,~$\kappa_a^{-1}$,~$\lambda_a^{-1}$, respectively.

\item[(d)] The support of a (free) ribbon-homotopy is the set $F\subseteq\Sigma$ of faces of $\D(\Sigma)^*$ arising as the union of the supports of the elementary
contractions/relaxations. In case of vertex or face contraction/relaxation, the support is just the vertex or face in question.
In case of $\kappa$-contraction/relaxation, the support has 2 elements, a vertex and an edge of $\Sigma$, the ones being winded around by $\kappa_a$.
For $\lambda$-contraction/relaxation, the support is a similarly defined pair of a face and of an edge.
\end{enumerate}
\end{defi}

Notice that we have not defined homotopy between a left and a right ribbon.

In order to place ribbon-homotopy into the wider context of homotopy, we need to see abundance of ribbons in some sense.
\begin{pro}%\label{pro: abundance of ribbons}
Let $\Sigma$ be a connected OCPM and let $s,s'\in\D(\Sigma)^{*0}$ be two sites.
\begin{enumerate}\itemsep=0pt
\item[$({\rm i})$] There exists a left ribbon $\rho\colon s\to s'$.
\item[$({\rm ii})$] If $\gamma\colon s\to s'$ is an opcurve, then there exists a left ribbon $\rho\colon s\to s'$ such that $\gamma\sim\rho$, i.e.,~$\gamma$ and~$\rho$ are
homotopic in the sense of Definition~$\ref{def: htp}$.
\item[$({\rm iii})$] If $\rho_1,\rho_2\colon s\to s'$ are left ribbons such that $\rho_1\approx\rho_2$, then $\rho_1\sim\rho_2$.
\end{enumerate}
\end{pro}
\begin{proof}
(ii) Let $\gamma$ be the composite $s=s(a_0)\rarr{c_1}s(a_1)\rarr{c_2}\cdots\rarr{c_n}s(a_n)=s'$, $v_k=\Ou_0(a_k)$, $f_k=\Ou_2(a_k)$ and let $w=W_n\cdots W_1$
be the code word of $\gamma$. The proof follows the principle of replacing bad letters with good words.
Construct the new word
\begin{gather*}
w':=w'_n\cdots w'_1,
\qquad \text{where}\quad w'_k:=\begin{cases}
W_k&\text{if}\ W_k\in\bigl\{T_0^{-1},T_2\bigr\},\\
T_0^{-|v_k|+1}&\text{if}\ W_k=T_0,\\
T_2^{|f_k|-1}&\text{if}\ W_k=T_2^{-1},
\end{cases}
\end{gather*}
and define the ribbons $\rho_k:=\bra a_0,w'_k\cdots w'_1\ket$. Then $\vartheta_k:=\rho_{k-1}(c_k)\rho_k^{-1}$ is a lasso such that
$\vartheta_1\vartheta_2\cdots\vartheta_n=\gamma\rho_n^{-1}$. Therefore, $\rho=\rho_n$ is the desired ribbon and
$\bigl(\gamma^{-1}\vartheta_1\gamma,\dots,\gamma^{-1}\vartheta_n\gamma\bigr)\colon\gamma\to\rho$ is the desired homotopy.

(i) Since $\D(\Sigma)^*$ is connected, there exists an opcurve $\gamma=s\to s'$. Thus, (i) follows from (ii).

(iii) It suffices to show that the elementary relaxations are homotopies. For the cases of vertex and face relaxations, this is obvious. For $\kappa_a$,
we can easily check the equality of paths $\bigl((T_1 a)_0^-\bigr)=\kappa_a\vartheta_v\vartheta_e$ with lassos $\vartheta_v=(a_2^-)\alpha_a\bigl(a_2^+\bigr)$ and
${\vartheta_e=(a_2^-,a_0^-)\beta_a^{-1}\bigl(a_0^+,a_2^+\bigr)}$. Thus, ${(\vartheta_v,\vartheta_e)\colon\kappa_a\to \bigl((T_1 a)_0^-\bigr)}$ is a homotopy.
Similar formulas can be found for $\lambda_a$.
\end{proof}

Assume that we are given a closed ribbon which is contractible, e.g., because we know that the ribbon bounds a disk. How can we design a ribbon-homotopy
which shrinks the ribbon to zero? This is the kind of problem in which ``tree curves'' play a role.

\begin{defi}
Tree curves or, more descriptively, tree dressing curves are closed ribbon opcurves occurring in two types:

A tree curve $\tau_F$ can be constructed from a tree $F\subset\Sigma^0\cup\Sigma^1$, as a subgraph of $\Sigma$, as follows. Considering
$F$ as a set of faces of $\D(\Sigma)^*$ we take $\tau_F$ to be any of the negatively oriented reduced boundary opcurves of this fattened tree $F$.

A dual tree curve $\tau_F$ can be constructed from a dual tree $F\subset\Sigma^2\cup\Sigma^1$, as a subgraph of $\Sigma^*$, as follows. Consider
$F$ as a set of faces of $\D(\Sigma)^*$ and take $\tau_F$ to be any of the positively oriented reduced boundary opcurves of this fattened tree~$F$.
\end{defi}

Predecessors of the tree curves are the ``dual blocks'' and ``direct blocks'' of Bombin and Martin-Delgado \cite{Bombin-Delgado}.
We need also rooted versions of them which are open ribbons:
\begin{defi}
A rooted tree curve is obtained from a tree curve $\tau_F$ by discarding the maximal arc running around a leaf
vertex $v\in F$. This curve is denoted by $\tau_{F'}$, where $F'=F\setminus\{v\}$ is the underlying rooted tree. The root of $\tau_{F'}$ is the arrow
which would complete the discarded arc to a full loop around $v$. If $r$ is the root, $(r)$ is called the root curve.

A rooted dual tree curve is obtained from a dual tree curve $\tau_F$ by discarding the maximal arc running around a leaf
face $f\in F$. This curve is denoted by $\tau_{F'}$, where $F'=F\setminus\{f\}$ is the underlying rooted dual tree. The root of $\tau_{F'}$ is
the arrow which would complete the discarded arc to a full loop around $f$. If $r$ is the root, $(r)$ is called the root curve.
\end{defi}

Note the deviation from the usual notion of rooted tree. Here the root vertex is not an arbitrary vertex but a leaf. Furthermore, the root vertex
does not belong to the rooted tree, so in fact there is only a root edge which is characterized by having only one boundary vertex (a~``half edge'').

The inverses $\tau^{-1}_F$ $\bigl(\tau^{-1}_{F'}\bigr)$ of the above (rooted) tree curves are needed in ribbon-homotopy of right ribbons,
so they should also be called (rooted) tree curves.

\begin{lem}
The following facts are either obvious or easy to prove:
\begin{enumerate}\itemsep=0pt
\item[$(i)$] For every $($rooted$)$ $($dual$)$ tree $F$, the $\tau_F$ is a proper left ribbon and $\tau_F^{-1}$ is a proper right ribbon.
\item[$(ii)$] If $F$ is a singleton tree in $\Sigma$ $($in $\Sigma^*)$, then the $($dual$)$ tree curve $\tau_F$ is $\alpha_a^{-1}$ $(\gamma_a)$ for some $a\in\Arr(\Sigma)$.
\item[$(iii)$] If the rooted tree $F'$ in $\Sigma$ $($in $\Sigma^*)$ consists of $2$ elements, then the rooted $($dual$)$ tree curve~$\tau_{F'}$ is a~$\kappa_a$ $(\lambda_a)$ for some $a\in\Arr(\Sigma)$.
\item[$(iv)$] Every $($dual$)$ tree curve can be obtained from an $\alpha_a^{-1}$ $(\gamma_a)$ by successive grafting of $\kappa_a$-s $(\lambda_a$-s$)$
by their roots to some arrows of color $T_0^{-1}$ $(T_2)$. The same can be said about the rooted $($dual$)$ tree curves if we start with a $\kappa$ $(\lambda)$.
\item[$(v)$] Every $\kappa_a$, $\lambda_a$ is ribbon-homotopic to its root curve.
\item[$(vi)$] Every tree curve and dual tree curve is ribbon-homotopic to some trivial ribbon.
\item[$(vii)$] Every rooted tree curve and rooted dual tree curve is ribbon-homotopic to its root curve.
\end{enumerate}
\end{lem}
Under (iv) the ``grafting'' means a $\kappa$- or $\lambda$-relaxation under the condition that the result remains a proper opcurve. We do not have to use
unproper relaxations to get tree curves although a general ribbon-homotopy would allow them.

\begin{thm}\label{thm: proper contract}
For a connected OCPM $\Sigma$, let $\rho\colon s\to s$ be a closed proper left ribbon.
We divide the faces of $\D(\Sigma)^*$ into connected components by saying that two faces $p$ and $q$ of $\D(\Sigma)^*$ are connected if there is a dual
curve $\delta\colon p\to q$ on $\D(\Sigma)^*$ such that neither the arrows of $\delta$ nor their opposites are members of $\rho$. Assume that
\begin{itemize}\itemsep=0pt
\item $\rho$ cuts the set of faces into $2$ connected components, $L$ consisting of the faces lying on the left and $R$ consisting of the
faces lying on the right of $\rho$;
\item as subcomplexes of $\Sigma$ either $L$ or $R$ has Euler characteristic $1$.
\end{itemize}
Then there is a ribbon-homotopy from $\rho$ to the trivial ribbon $(s)$ such that the support of the homotopy is $L$, resp.\ $R$, and in each step of the
homotopy the curve is proper.
\end{thm}

\begin{figure}[t]\centering
\includegraphics[width=9cm]{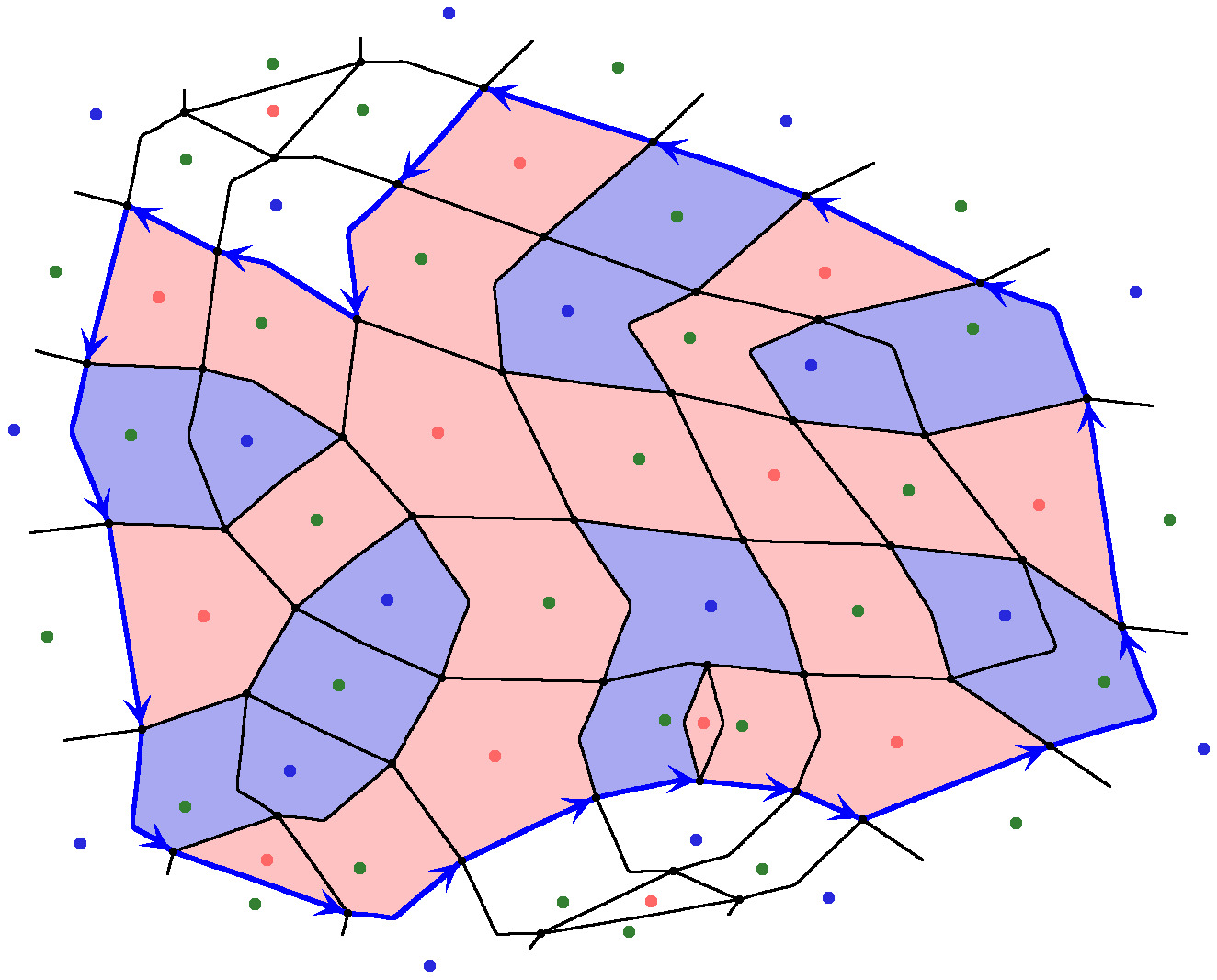}
\caption{A plan for contracting a left ribbon. (See the proof of Theorem~\ref{thm: proper contract}.)}
\label{fig: contraction}
\end{figure}

\begin{proof}Looking at the two connected components $L$ and $R$ as sets of cells of $\Sigma$, we have $L\cap R=\varnothing$ and $L\cup R=\Sigma$. But there is a substantial difference
between these subcomplexes. $L$ is upper semiclosed, i.e., together with any cell $c$ it contains also $\Cb(c)$, while $R$ is lower semiclosed\footnote{A subset of a~connected CPM $\Sigma$ which is both lower and upper semiclosed is either empty or the whole $\Sigma$.},
i.e., together with any cell $c$ it contains also $\Bd(c)$. This is visible on Figure~\ref{fig: contraction} by the appearance of alternating
red and green dots on the left-hand side of $\rho$ and of the blue and green dots on the right-hand side. These sequences of faces and edges on the left
and vertices and edges on the right are the boundary dual curve of $L$ and the boundary curve of $R$, respectively. They coincide with the two
borderlines $E^*(\rho)=\bra c,b\ket$ of $\rho$, see in and below Lemma~\ref{lem: E}. Since $\rho$ is a~ribbon, these borderlines never cross by
Lemma~\ref{lem: ribbon}.

Let $L^d=\Sigma^d\cap L$ for $d=0,1,2$ and assume that it is $L$ which has Euler characteristic $\chi(L)=1$. Choose a maximal dual tree
$F\subset L^2\cup L^1$ in the connected complex $L$. Let $\tau_F$ be its dual tree curve the base point of which is that of $\rho$.
Since $F$ contains all faces of $L$ but none of its vertices and $F$ has $\chi(F)=1$, the complement $G:=L\setminus F$ has Euler characteristic
\begin{gather*}
\chi(G)=\bigl|G^0\bigr|-\bigl|G^1\bigr|=\bigl|L^0\bigr|-\bigl(\bigl|L^1\bigr|-\bigl(\bigl|L^2\bigr|-1\bigr)\!\bigr)=\bigl|L^0\bigr|-\bigl|L^1\bigr|+\bigl|L^2\bigr|-1=\!\chi(L)-1=0.
\end{gather*}
As a subgraph of the planar $L$ the $G$ cannot contain circles because they would separate some faces from the maximal dual tree $F$. Therefore, $G$ is a disjoint union of trees $G=G_1\cup\dots\cup G_k$ each of which has to be rooted in order to satisfy $\chi(G_i)=0$.
The root edge of $G_i$, in the language of $\D(\Sigma)^*$, is a type 1 face $e_i$ having one arrow $r_i\in\rho$ of color $T_0^{-1}$ on its boundary.
This is the root of $G_i$. Starting with the root $r_i$, we can perform a sequence of $\kappa$-relaxations along the tree $G_i$. Doing this for all~$G_i$,
we arrive to that $\rho$ is a ribbon-homotopic to $\tau_F$. Since~$\tau_F$ can be contracted by $\lambda$-contractions to the trivial curve, we are done.
During each steps of the homotopy we had only proper ribbons and used only grafts $\kappa_a$, $\lambda_b$ which belonged entirely to~$L$. Furthermore,
the supports of the used grafts form a partition of $L$ into vertex-edge pairs and face-edge pairs, so the support of the homotopy is precisely $L$.
(In this case, we call $L$ the contraction disk of~$\rho$.)

If it is $R$ which has $\chi(R)=1$, then the proof goes by constructing a maximal tree $F\subset R^0\cup R^1$ and the forest of rooted dual trees
$G=R\setminus F$. Then the homotopy first $\lambda$-relaxes $\rho$ inside $R$ along the forest and then $\kappa$-contracts the resulting tree
curve~$\tau_F$.
\end{proof}

\begin{figure}[t]\centering
\includegraphics[width=7cm]{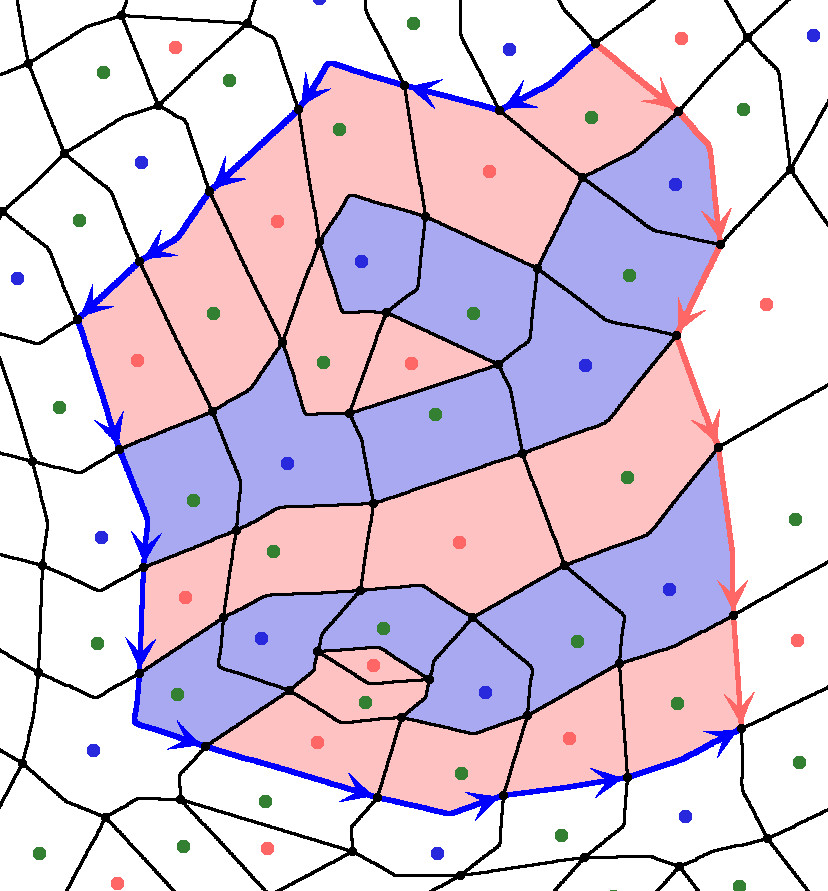}
\caption{The plan of a proper ribbon-homotopy from the blue to the red left ribbon. (See the proof of Theorem~\ref{thm: ribb2ribb htopy}.)}
\label{fig: ribb2ribb htopy}
\end{figure}

The special ribbon-homotopies we found in the above theorem turn out to be the only ribbon-homotopies applicable in the Kitaev model. So we give
them a name.
\begin{defi}
A (free) ribbon-homotopy $(\rho_0,\rho_1,\dots,\rho_n)$ is called a proper (free) ribbon-homotopy if each $\rho_i$ is proper (Definition~\ref{def: prop ribb}).
A proper ribbon which has a proper ribbon-homotopy to a trivial curve is called properly contractible.
\end{defi}
\begin{exa}
The figure~8 opcurve $8_a:=\alpha_a^{-1}\gamma_a$ is a left ribbon which is not proper. It has a~ribbon-homotopy to $(s(a))$ but no proper ribbon-homotopy
at all. The same can be said about the double loop $\gamma_a\gamma_a$.
\end{exa}

\begin{thm}\label{thm: ribb2ribb htopy}
Let $\rho_1,\rho_2\colon s(a)\to s(b)$ be parallel non-trivial proper left ribbons such that the closed opcurve $\rho_1\rho_2^{-1}$ is the positively oriented boundary
curve of a disk $B\subset \D(\Sigma)^{*2}$. Then there is a proper ribbon-homotopy $\rho_1\to \rho_2$ of support $B$.
\end{thm}
\begin{proof}
By assumption, $B$ lies on the left of $\rho_1$ and the right of $\rho_2$.
The $B$ is now neither upper nor lower semiclosed and, strangely enough, has Euler characteristic 0 in $\Sigma$. This is an artifact of the different
behaviour at the two borderline segments $\rho_1$ and $\rho_2$:
$B$ looks like upper semiclosed along the $\rho_1$ and lower semiclosed along $\rho_2$.
So we can start growing
rooted trees (with roots on~$\rho_1$ necessarily) until we get a maximal forest $F\subset B$ of rooted trees (see Figure~\ref{fig: ribb2ribb htopy}).
The complement $F^*=B\setminus F$ is then
a maximal forest of rooted dual trees (with roots on~$\rho_2$ necessarily). The homotopy then starts with $\kappa$-relaxations of $\rho_1$ along
the trees of~$F$. The result is a proper left ribbon $\gamma\colon s(a)\to s(b)$ which consists of the (broken) borderline between~$F$ and~$F^*$ completed by
those arrows of~$\rho_1$ and~$\rho_2$ which are not roots of some tree. Finally, a~sequence of $\lambda$-contractions applied to $\gamma$ along $F^*$
results in the ribbon~$\rho_2$.
\end{proof}

The intermediate curve $\gamma$ in the proof of the above theorem has the remarkable property of passing through every site of the disk $B$
exactly once. So it is a Hamiltonian opcurve on the graph $\downarrow\!B\setminus B$ obtained by taking the lower semiclosed closure of $B$ in the
complex $\D(\Sigma)^*$ and then discarding all its faces.
The existence of such Hamiltonian curve is crucial in the following free ribbon-homotopy.

\begin{thm}\label{thm: annulus homotopy}
Let $A$ be a finite connected set of faces of $\D(\Sigma)^*$ with Euler characteristic~$0$ in~$\Sigma$ such that the positively oriented boundary
of $A$ in $\D(\Sigma)^*$ consists of two components, a~closed proper left ribbon $\rho_1$ and a closed proper right ribbon $\rho_2^{-1}$.
Then there is a proper free ribbon-homotopy $\rho_1\to\rho_2$ supported on the annulus $A$.
\end{thm}
\begin{proof}
The proof is very similar to the proof of Theorem~\ref{thm: ribb2ribb htopy} therefore we content ourselves with referring to Figure~\ref{fig: annulus}.
\end{proof}

\begin{figure}[t]\centering
\includegraphics[width=10cm]{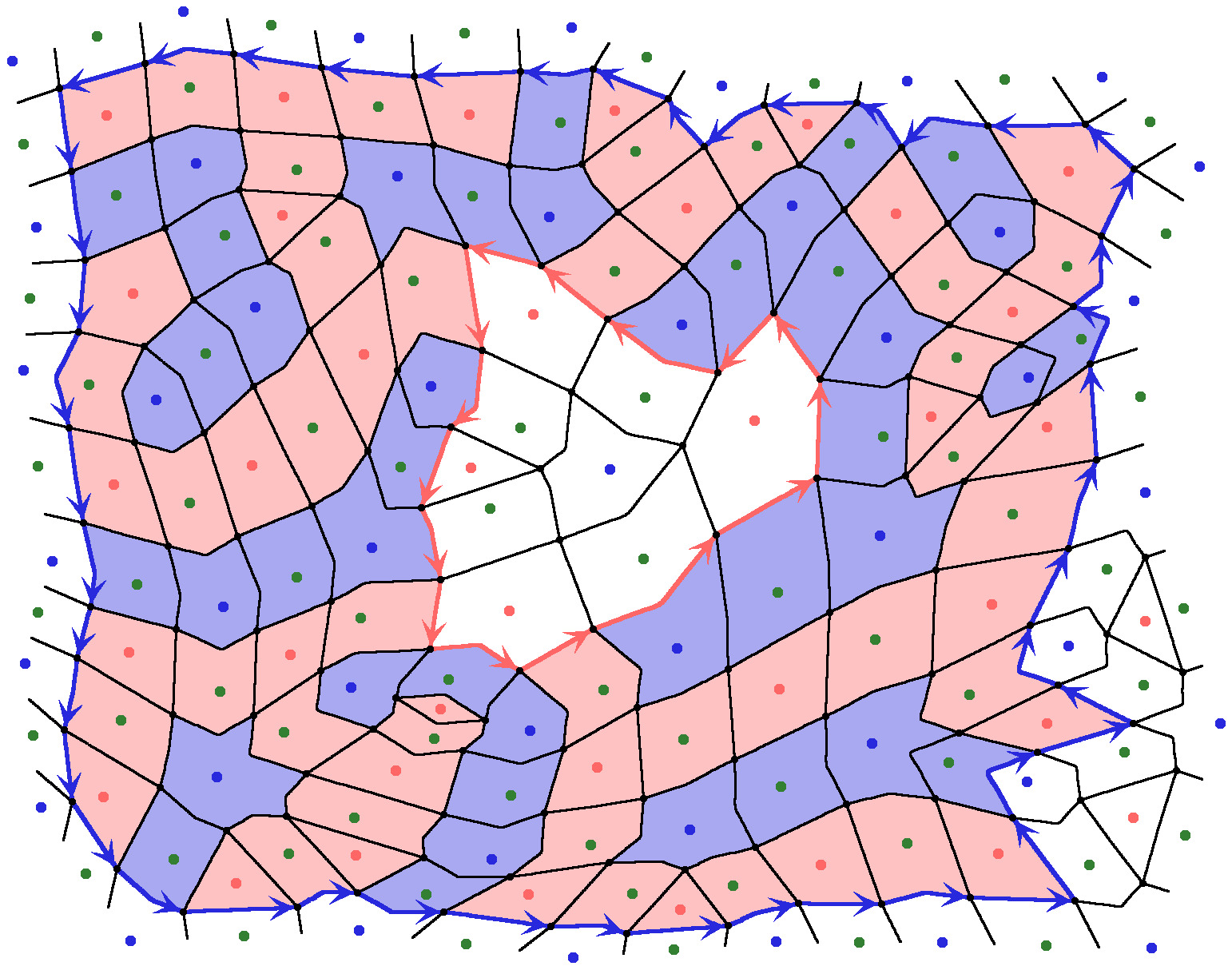}
\caption{For a free proper ribbon-homotopy from the blue ribbon to the red ribbon first choose a~maximal forest of rooted trees on the annulus.
These are the bluish faces.
Then perform $\kappa$-relaxations on the blue ribbon along the bluish faces. What we get is a Hamilton cycle; a closed proper ribbon (not shown) visiting every site of the annulus exactly once. Then we make a circular move along the Hamilton cycle in order to transplant the base point from the
blue ribbon to the red ribbon. Finally, we perform $\lambda$-contractions along the reddish faces.}
\label{fig: annulus}
\end{figure}

The annulus of Figure~\ref{fig: annulus} is obtained by taking the relative complement of a disk within a~larger disk. But this need not be so. The proof of Theorem~\ref{thm: annulus homotopy} nowhere uses that the ``hole'' is a~disk. So it applies to any annulus strip winding around a torus or on some handle of a
higher genus surface.

\section{Ground states}\label{section9}

So far the Kitaev model was treated on a pure algebraic level. Now we turn it into a model in quantum field theory. So we assume that the ground field
$K=\CC$, the field of complex numbers, and the finite-dimensional Hopf algebra $H$ is a C$^*$-Hopf algebra. This means, on the one hand, that $H$ is endowed
with an antilinear involution, called the star operation,\footnote{Having already many stars in this paper we shall denote the star operation, as well as the adjoints of bounded linear operators on a Hilbert space, by $(\cdot)^\dagger$.} satisfying
\[
(gh)^\dag=h^\dag g^\dag,\quad {h\p}^\dag\ot{h\pp}^\dag=\bigl(h^\dag\bigr)\p\ot\bigl(h^\dag\bigr)\pp,\qquad g,h\in H.
\]
Consequently,
\[
1^\dag=1,\qquad \eps\bigl(h^\dag\bigr)=\overline{\eps(h)},\qquad S(h)^\dag=S^{-1}\bigl(h^\dag\bigr),\qquad i^\dag=i.
\]
On the other hand, one requires that $H$ has a faithful star-representation on a Hilbert space, i.e., such that $h^\dag$ is represented by the adjoint of
the operator representing $h$. From now on $H$ denotes a finite-dimensional C$^*$-Hopf algebra. Finite-dimensional C$^*$-Hopf algebras are
semisimple and have $S^2=\id$. So every Hopf algebraic relations used before remain valid. In particular, $H^*$ has a Haar integral $\iota$ and the formula
\[
(g,h):=\bigl\bra\iota,g^\dag h\bigr\ket
\]
is a scalar product making $H$ a Hilbert space and the left regular representation a faithful star representation, $(k,hg)=\bigl(h^\dag k,g\bigr)$.

The dual Hopf algebra $H^*$ is then also a C$^*$-Hopf algebra. The star operation of $H^*$ is given~by
\[
\bigl\bra\varphi^\dag,h\bigr\ket:=\overline{\bigl\bra \varphi,S(h)^\dag\bigr\ket}, \qquad \varphi\in H^*,\quad h\in H.
\]
It follows that the double $\D(H)$ and its dual $\D(H)^*$ are both C$^*$-Hopf algebras with respective star operations
\begin{align*}
(\psi\ot g)^\dag =\psi^\dag\pp\ot g^\dag\pp\cdot\overline{\bigl\bra\psi\ppp,S\bigl(g\p\bigr)\bigr\ket}\,\overline{\bigl\bra\psi\p,g\ppp\bigr\ket},\qquad
(h\ot\varphi)^\dag =h^\dag\ot\varphi^\dag.
\end{align*}

Assume $\Sigma$ is a finite connected OCPM.
Since $\M(\Sigma)$ is a full matrix algebra, it has, up to isomorphisms, only one irreducible representation. The standard choice for the Hilbert space of this representation is $\Hil(\Sigma):=H^{\otimes \Sigma^1}$ which is the tensor product of copies $H_e$ of the Hilbert space $H$ for each edge
$e\in\Sigma^1$.
The action of $\M$ on $\Hil$ is defined by means of an orientation of the edges, i.e., by a section $e\mapsto a_e$ of the function
$\Ou_1\colon \Arr(\Sigma)\to \Sigma^1$. Namely,~$P_a(h)$ and~$Q_a(\varphi)$ act as the identity on each tensor factor $H_e$, $e\neq\Ou_1(a)$, and
for $e=\Ou_1(a)$ the $P_a(h)$ is left multiplication by $h$ and $Q_a(\varphi)$ is $\varphi\sweedl\under$ on $H_e$ if $a=a_e$ and
$P_a(h)$ is right multiplication by $S(h)$ and $Q_a(\varphi)=\under\sweedr S(\varphi)$ on $H_e$ if $a=T_1a_e$; cf.\ (\ref{lambda}) and (\ref{rho}).
In this way the abstract operators $P_a(h)$, $Q_a\bigl(\varphi\bigr)$ become concrete operators on the Hilbert space $\Hil$ such that $P_a(h)^\dag=P_a\bigl(h^\dag\bigr)$
and $Q_a(\varphi)^\dag=Q_a\bigl(\varphi^\dag\bigr)$, where $(\cdot)^\dag$ on the right-hand side is the abstract star operation while on the left-hand side it is the adjoint of an operator
on $\Hil$.

The adjoints of one-step holonomies (\ref{Ophol}) satisfy $\Ophol_{(d)}(\Psi)^\dag=\Ophol_{(d)}\bigl(\Psi^\dag\bigr)$. Therefore, if $\gamma$ is an arbitrary opcurve, then
\begin{equation}%\label{ophol dagger general}
\Ophol_\gamma(\Psi)^\dag=\Ophol_{\gamma^{-1}}\bigl(S_{\D^*}(\Psi)^\dag\bigr),\qquad \Phi\in\D(H)^*.
\end{equation}
But if the curve is a proper ribbon, then we also have
\begin{equation}%\label{ophol dagger}
\Ophol_\rho(\Psi)^\dag=\Ophol_\rho\bigl(\Psi^\dag\bigr),\qquad \Psi\in\D(H)^*,\quad\text{if}\ \rho\ \text{is proper}.
\end{equation}
In particular, the Gauss' law and flux operators satisfy $G_a(h)^\dag=G_a\bigl(h^\dag\bigr)$ and $F_a(\varphi)^\dag=F_a\bigl(\varphi^\dag\bigr)$. So $A_v$ and $B_f$
are self adjoint idempotents, i.e., projectors, and the Hamiltonian (\ref{hamiltonian}) is self adjoint, too.

The ground states of the Hamiltonian are represented by vectors $\Omega$ in $\Hil$ such that
\begin{equation*}%\label{vac}
A_v\Omega=\Omega\qquad\text{and}\qquad B_f\Omega=\Omega\qquad \forall v\in\Sigma^0,\quad f\in\Sigma^2.
\end{equation*}
Following \cite{Kitaev03}, we denote by $\Pee$ the subspace in $\Hil$ of such vacuum vectors $\Omega$ and call it the protected space.
Explicit wave functions of vacuum vectors have been calculated in \cite{BMCA} using a~tensor network formalism.
If $H$ is a group algebra of a finite group $G$, then a basis of $\Pee$ consists of gauge equivalence classes of $G$-valued flat gauge fields on~$\Sigma$~\cite{Cowtan-Majid}. So the dimension of $\Pee$ depends only on the genus of $\Sigma$ through the fundamental group $\pi_1(\Sigma)$.
A~far reaching generalization of the ``protected object'' $\Pee$ appeared recently in~\cite{Hirmer-Meusburger} in which the Kitaev quantum double model is studied over Hopf monoids in monoidal categories.

In terms of expectation values $\omega(M)=(\Omega,\,M\Omega)$, the ground states can be abstractly characterized by the equations
\[
\omega(MA_v)=\omega(M)=\omega(MB_f), \qquad M\in\M,\quad v\in\Sigma^0,\quad f\in\Sigma^2
\]
for the normalized ($\omega(\one)=1$) positive linear functional $\omega\colon\M\to\CC$. Knowing $\omega$, which we do not, is complete information about
the representation on $\Hil$ by the Gelfand--Neumark--Segal construction. So the representation of $\M$ on $\Hil$ may be called the vacuum representation.

The above properties of the vacua $\Omega\in\Pee$ imply
\begin{align}
\label{alpha on vac}
\Ophol_{\alpha_a}(\Phi)\Omega&=\varphi(1)G_a(h) G_a(i)\Omega=\varphi(1)\eps(h)\Omega=\Omega\cdot\eps_{\D^*}(\Phi),\\
\label{gamma on vac}
\Ophol_{\gamma_a}(\Phi)\Omega&=\eps(h)F_a(\varphi)F_a(\iota)\Omega=\eps(h)\varphi(1)\Omega=\Omega\cdot\eps_{\D^*}(\Phi),
\end{align}
hence also the formula expressing gauge invariance of the vacua
\[
\mathbf{D}_a(X)\Omega=\Omega\cdot\eps_D(X),\qquad a\in\Arr(\Sigma),\quad X\in\D(H),\quad \Omega\in\Pee.
\]

\begin{pro}\label{pro: contract on vac}
Let $a\in\Arr(\Sigma)$, $v=\Ou_0(a)$, $f=\Ou_2(a)$ and let $\Omega\in\Hil$ be any vector. Then for all $\Phi\in\D(H)^*$
\begin{alignat}{3}
\label{vertex Omega}
&\Ophol_{\alpha^{\pm 1}_a}(\Phi)\Omega=\Omega\cdot\eps_{\D^*}(\Phi)&&\qquad\text{whenever}\quad A_v\Omega=\Omega,&\\
\label{face Omega}
&\Ophol_{\gamma^{\pm 1}_a}(\Phi)\Omega=\Omega\cdot\eps_{\D^*}(\Phi)&&\qquad\text{whenever}\quad B_f\Omega=\Omega,&\\
\label{kappa Omega}
&\Ophol_{\kappa_a}(\Phi)\Omega=\Ophol_{((T_1a)_0^-)}\Omega&&\qquad\text{whenever}\quad A_v\Omega=\Omega,&\\
\label{lambda Omega}
&\Ophol_{\lambda_a}(\Phi)\Omega=\Ophol_{((T_1a)_2^+)}\Omega&&\qquad\text{whenever}\quad B_f\Omega=\Omega,&\\
\label{kappa-inv Omega}
&\Ophol_{\kappa^{-1}_a}(\Phi)\Omega=\Ophol_{((T_1a)_0^+)}\Omega&&\qquad\text{whenever}\quad A_v\Omega=\Omega,&\\
\label{lambda-inv Omega}
&\Ophol_{\lambda^{-1}_a}(\Phi)\Omega=\Ophol_{((T_1a)_2^-)}\Omega&&\qquad\text{whenever}\quad B_f\Omega=\Omega,&
\end{alignat}
where $\kappa_a$, $\lambda_a$ are the left ribbons defined in \eqref{kappa_a} and \eqref{lambda_a}.
\end{pro}
\begin{proof}
Formulas (\ref{vertex Omega}) and (\ref{face Omega}) can be proven as in (\ref{alpha on vac}) and (\ref{gamma on vac}).
The proof of (\ref{kappa Omega}) goes as follows.
$\kappa_a$ contains $\alpha_a^{-1}$ except one arrow, \smash{$\bigl(T_0^ma\bigr)_0^-=a_0^-$}. So we write its opholonomy as
\[
\Ophol_{\kappa_a}(\Phi)=\Ophol_{((T_1a)_2^+,a_0^+)}\bigl(\Phi\p\bigr)\Ophol_{\alpha_a^{-1}}\bigl(\Phi\pp\bigr)\Ophol_{(a_2^+)}\bigl(\Phi\ppp\bigr).
\]
We leave the first term for a while and concentrate on the action of the rest on the vacuum:
\begin{align*}
\Ophol_{\alpha_a^{-1}}\bigl(\Phi\p\bigr)\Ophol_{(a_2^+)}\bigl(\Phi\pp\bigr)\Omega
&{}=G_a(S\bigl(h\p\bigr))\bigl\bra\xi_i\varphi\p\xi_j,1\bigr\ket \eps\bigl(S(x_j)h\pp x_i\bigr)Q_a\bigl(\varphi\pp\bigr)\Omega\\
&{}=G_a(S(h))Q_a(\varphi)\Omega\eqby{Q G}Q_a\bigl(\varphi\sweedr h\p\bigr)G_a\bigl(S\bigl(h\pp\bigr)\bigr)\Omega\\
&{}=Q_a(\varphi\sweedr h)\Omega,
\end{align*}
where, as usual, $\Phi=h\ot\varphi$ stands for an arbitrary element of $\D(H)^*$.
In order to express this result in terms of $\Phi$, we remark that
\[
u=\xi_i\ot S(x_i)\qquad\text{and}\qquad u^{-1}=\xi_i\ot x_i.
\]
Therefore, \begin{align*}
u^{-1}\sweedl \Phi&=\bigl(h\p\ot\xi_j\varphi\p\xi_k\bigr)\cdot\bigl\bra S(x_k)h\pp x_j\ot\varphi\pp,\xi_i\ot x_i\bigr\ket\\
&=\bigl(h\p\ot\xi_j\varphi\p\xi_k\bigr)\cdot\bigl\bra\varphi\pp,S(x_k)h\pp x_j\bigr\ket\\
&=h\p\ot\varphi\pppp\varphi\p S\bigl(\varphi\pp\bigr)\cdot\bigl\bra\varphi\ppp,h\pp\bigr\ket\\
&=h\p\ot\varphi\sweedr h\pp
\end{align*}
and $\varphi\sweedr h=(\eps\ot\id)\bigl(u^{-1}\sweedl \Phi\bigr)$, so we have
\[
Q_a(\varphi\sweedr h)=\Ophol_{(a_2^+)}\bigl(u^{-1}\sweedl \Phi\bigr).
\]
In this way, we have proven that
\begin{align*}
\Ophol_{\kappa_a}(\Phi)\Omega&=\Ophol_{((T_1a)_2^+,a_0^+)}\bigl(\Phi\p\bigr)\Ophol_{(a_2^+)}\bigl(u^{-1}\sweedl \Phi\pp\bigr)\Omega=\\
&=\Ophol_{((T_1a)_2^+,a_0^+,a_2^+,(T_1a)_0^+)}\bigl(\Phi\p\bigr) \Ophol_{((T_1a)_0^-)}\bigl(u^{-1}\sweedl \Phi\pp\bigr)\Omega.
\end{align*}
Now the opcurve in the first term is a conjugate of the central curve $\beta_a$ which, by (\ref{edge loop ophol}) and Lemma~\ref{lem: central},
we know to have opholonomy equal to $\one_\M\cdot\bra\under,S_\D(u)\ket$. Thus, \[
\Ophol_{\kappa_a}(\Phi)\Omega=\Ophol_{((T_1a)_0^-)}\bigl(S_\D(u)u^{-1}\sweedl \Phi\bigr)\Omega.
\]
For C$^*$-Hopf algebras, $S_\D(u)=u$, therefore (\ref{kappa Omega}) is proven.

The proof of (\ref{lambda Omega}) goes as follows.
We insert $\bigl(a_2^-,a_2^+\bigr)$ into $\lambda_a$ in order to complete the incomplete face loop,
\[
\Ophol_{\lambda_a}(\Phi)=\Ophol_{((T_1a)_0^-,a_2^-)}\bigl(\Phi\p\bigr)\Ophol_{\gamma_a}\bigl(\Phi\pp\bigr)\Ophol_{(a_0^-)}\bigl(\Phi\ppp\bigr).
\]
Next we compute the action of the last two terms on the vacuum by using the (inverse of the) exchange relation (\ref{P F}),
\begin{align*}
\Ophol_{\gamma_a}\bigl(\Phi\p\bigr)\Ophol_{(a_0^-)}\bigl(\Phi\pp\bigr)\Omega
&{}=\eps\bigl(h\p\bigr)F_a\bigl(\xi_i\varphi\p\xi_j\bigr)P_a\bigl(S\bigl(S(x_j)h\pp x_i\bigr)\bigr)\varphi\pp(1)\\
&{}=P_a\bigl((\xi_i\varphi\xi_j)\p\sweedl S(x_i)S(h)x_j\bigr)F_a\bigl((\xi_i\varphi\xi_j)\pp\bigr)\Omega\\
&{}=P_a(\xi_i\varphi\xi_j\sweedl S(x_i)S(h)x_j)\Omega.
\end{align*}
Next we simplify the last expression by proving the identity
\begin{align*}
\xi_i\varphi\xi_j\sweedl S(x_i)S(h)x_j
&{}=\bigl\bra\Phi\p,\eps\ot x_i\bigr\ket\bigl\bra\Phi\pp,\xi_j\ot 1\bigr\ket\cdot(\xi_i\sweedl S(x_j))\\
&{}=\bra\Phi,(\eps\ot x_i)(\xi_j\xi_k\ot 1)\ket S(x_k) \bra\xi_i,S(x_j)\ket\\
&{}=\bra\Phi,\underset{S_\D(R_1)R_2=S_\D(u)}{\underbrace{(\eps\ot S(x_j)(\xi_j\ot 1)}}(\xi_k\ot 1)\ket S(x_k)\\
&{}=S((\id\ot\eps)(\Phi\sweedr S_\D(u))),
\end{align*}
which implies
\[
P_a(\xi_i\varphi\xi_j\sweedl S(x_i)S(h)x_j)=\Ophol_{(a_0^-)}(\Phi\sweedr S_\D(u)).
\]
Putting together and using that $S_\D(u)$ is central, we obtain
\begin{align*}
\Ophol_{\lambda_a}(\Phi)\Omega&=\Ophol_{(T_1a)_0^-,a_2^-)}\bigl(\Phi\p\bigr)\Ophol_{(a_0^-)}\bigl(\Phi\pp\sweedr S_\D(u)\bigr)\Omega\\
&=\Ophol_{(T_1a)_0^-,a_2^-,a_0^-)}\bigl(\Phi\p\bigr) \bigl\bra\Phi\pp,S_\D(u)\bigr\ket\\
&=\Ophol_{(T_1a)_0^-,a_2^-,a_0^-,(T_1a)_2^-)}\bigl(\Phi\p\bigr)\Ophol_{((T_1a)_2^+)}\bigl(S_\D(u)\sweedl\Phi\pp\bigr).
\end{align*}
Now the opcurve in the first term is a conjugate of $\beta_a^{-1}$ therefore it is central and its opholonomy is
$\one_\M\cdot\bigl\bra\under,S_\D(u)^{-1}\bigr\ket$ by Lemma~\ref{lem: central}. This immediately implies (\ref{lambda Omega}).

Since $\kappa_a$, $\lambda_a$ are proper ribbons, (\ref{kappa-inv Omega}) and (\ref{lambda-inv Omega}) for all $\Phi$ are equivalent to (\ref{kappa Omega})
and (\ref{lambda Omega}), respectively, by (\ref{proper inverse}).
\end{proof}

As a generalization of the protected space, let us introduce the spaces $\Pee(F)\subseteq\Hil$ of partial vacua as follows. Consider $F\subseteq\Sigma$ as a
set of faces of $\D(\Sigma)^*$ and let $F^d:=F\cap\Sigma^d$. Then define
\[
\Pee(F):=\bigl\{\Omega\in\Hil\mid A_v\Omega=\Omega,B_f\Omega=\Omega,\ \text{for}\ v\in F^0,\, f\in F^2\bigr\}.
\]
So in particular $\Pee(\Sigma)=\Pee$ and $\Pee(\varnothing)=\Hil$.

\begin{thm}[topological invariance]\label{thm: topological invariance}
For a connected OCPM $\Sigma$, let $\rho_1$, $\rho_2$ be proper ribbons on $\D(\Sigma)^*$ of the same left/right type and let
$F\subset\D(\Sigma)^{*2}$. Then
\begin{enumerate}\itemsep=0pt
\item[$(i)$] if there is a proper ribbon-homotopy $\rho_1\to\rho_2$ of support $F$, then
\[
\Ophol_{\rho_1}(\Phi)\Omega=\Ophol_{\rho_2}(\Phi)\Omega\qquad\text{for all}\quad \Phi\in\D(H)^*,\quad \Omega\in \Pee(F);
\]
\item[$(ii)$] if $\rho_1$, $\rho_2$ are closed and there is a proper free ribbon-homotopy $\rho_1\to\rho_2$ of support $F$ then
\[
\Ophol_{\rho_1}(\Psi)\Omega=\Ophol_{\rho_2}(\Psi)\Omega\qquad\text{for all}\quad \Psi\in\Cocom\D(H)^*,\quad \Omega\in \Pee(F).
\]
\end{enumerate}
\end{thm}
\begin{proof}
(i) It suffices to consider elementary ribbon-homotopies $(\rho_1,\rho_2)$ in which a subribbon $\sigma_1\subset\gamma\sigma_1\delta=\rho_1$ is replaced
by the subribbon $\sigma_2\subset\gamma\sigma_2\delta=\rho_2$. (So $\sigma_1\mapsto\sigma_2$ is either a vertex or face or $\kappa$- or
$\lambda$-contraction/relaxation.) Then the argument is this
\begin{align*}
\Ophol_{\rho_1}(\Phi)\Omega&=\Ophol_{\gamma}\bigl(\Phi\p\bigr)\Ophol_{\sigma_1}\bigl(\Phi\pp\bigr)\Ophol_{\delta}\bigl(\Phi\ppp\bigr)\Omega= [\rho_1\text{ is proper}]\\
&=\Ophol_{\gamma}\bigl(\Phi\p\bigr)\Ophol_{\delta}\bigl(\Phi\ppp\bigr) \Ophol_{\sigma_1}\bigl(\Phi\pp\bigr)\Omega= [\text{Proposition~\ref{pro: contract on vac}}]\\
&=\Ophol_{\gamma}\bigl(\Phi\p\bigr)\Ophol_{\delta}\bigl(\Phi\ppp\bigr) \Ophol_{\sigma_2}\bigl(\Phi\pp\bigr)\Omega= [\rho_2 \ \text{is proper}]\\
&=\Ophol_{\gamma}\bigl(\Phi\p\bigr)\Ophol_{\sigma_2}\bigl(\Phi\pp\bigr)\Ophol_{\delta}\bigl(\Phi\ppp\bigr)\Omega=\Ophol_{\rho_2}(\Phi)\Omega.
\end{align*}

(ii) follows from (i) if we add the observation that for cocommutative $\Psi$ we can cyclically permute proper ribbons already at the level of operators:
$\Ophol_{\rho_1}(\Psi)=\Ophol_{\rho_2}(\Psi)$ whenever $(\rho_1,\rho_2)$ is a circular move, see (\ref{proper cyclic}).
\end{proof}

\begin{cor}%\label{cor: topological invariance}
Let $\gamma$ be a closed proper ribbon, $\rho\colon\bra v_0,f_0\ket\to\bra v_1,f_1\ket$ any proper ribbon and let $\Psi\in\Cocom\D(H)^*$ and $\Phi\in\D(H)^*$.
\begin{enumerate}\itemsep=0pt
\item[$(i)$] If $\Omega\in\Pee$ ,then $\Ophol_\gamma(\Psi)\Omega$ also belongs to $\Pee$ and depends only on the proper free ribbon-homotopy class of $\gamma$.
\item[$(ii)$] If $\gamma$ is properly contractible $($Theorem~$\ref{thm: proper contract})$ and $\Omega\in\Pee$, then
$\Ophol_\gamma(\Psi)\Omega=\Omega\cdot\eps_{\D^*}(\Psi)$.
\item[$(iii)$] If $\Omega\in\Pee$, then $\Ophol_\rho(\Phi)\Omega$ belongs to $\Pee(\Sigma\setminus\{v_0,v_1,f_0,f_1\})$ and depends only on the proper ribbon-homotopy
class of $\rho$.
\end{enumerate}
\end{cor}
\begin{proof}
This is a direct consequence of Corollary~\ref{cor: ribop gauge invar} and Theorem~\ref{thm: topological invariance} with the additional remark that
in (iii) the support of any proper ribbon-homotopy $\rho\to\rho'$ is automatically disjoint from $\{v_0,v_1,f_0,f_1\}$.
\end{proof}

The natural question that emerges is whether the proper ribbon-homotopy classes are the same as the ribbon-homotopy classes or the homotopy classes.
If $\Sigma$ is finite and has genus $\mathbf{g}>0$, then the answer is no. While the latter two are infinite groupoids there are only finitely many proper ribbons on a finite $\Sigma$. In addition, proper ribbons cannot be composed so the proper ribbon-homotopy classes need not even form a groupoid.
On the infinite plane, however, one expects that every ribbon-homotopy class contains a proper ribbon so all the three should reduce to the pair groupoid
on the set of sites.

\section{Excited states}\label{section10}

As explained in \cite{Kitaev03}, the holonomy of an open ribbon $\rho\colon s_0\to s_1$ can be thought to create some charge at $s_1$ and the anticharge
at $s_0$. These charges can then be measured by contractible closed ribbon operators winding around~$s_i$. We can now make this precise in the Hopf
algebraic Kitaev model.

\begin{thm}\label{thm: gamma measures what rho creates}
Let $\Sigma$ be a connected OCPM. On $\D(\Sigma)^*$ let $\rho\colon s_0\to s_1$ be a proper left ribbon and $\gamma\colon s\to s$ a closed proper left ribbon.
\begin{enumerate}\itemsep=0pt
\item[$(i)$] We assume that $\gamma$ is properly contractible in the sense of Theorem~$\ref{thm: proper contract}$ so divides $\D(\Sigma)^{*2}$ into two
connected components $L$ and $R$ among which now $L$ is the disk.
\item[$(ii)$] $L$ contains both faces of the site $s_1=\bra v_1,f_1\ket$.
\item[$(iii)$] $R$ contains both faces of the site $s_0=\bra v_0,f_0\ket$.
\end{enumerate}
Then for all partial vacuum vectors $\Omega\in\Pee(L)$
\begin{equation}\label{gamma-rho-Omega}
\Ophol_\gamma(\Psi)\Ophol_\rho(\Phi)\Omega=\Ophol_\rho(\overset{\raisebox{-1pt}{\scriptsize $\leftrightarrow$}}{\Psi}\sweedl\Phi)\Omega
\end{equation}
for all $\Phi\in\D(H)^*$ and $\Psi\in\Cocom\D(H)^*$ where the linear isomorphism $\D(H)^*\to\D(H)$,
$\Psi=k\ot \psi\mapsto\overset{\raisebox{-1pt}{\scriptsize $\leftrightarrow$}}{\Psi}:=\psi\ot k$ maps $\Cocom\D(H)^*$ onto $\Center\D(H)$.
\end{thm}
\begin{proof}
$L$ is upper semiclosed and contains $v_1$ and $f_1$. The smallest upper semiclosed subcomplex of $\Sigma$ which contains $v_1$ and $f_1$ is nothing but $\uparrow\! v_1$ consisting\footnote{An example of a $\uparrow\! v$ is the hole in the annulus of Figure~\ref{fig: annulus}.}
of $v_1$ and of its vertex neighbourhood $\Nb(v_1)$ (Definition~\ref{def: CPM}). So $\uparrow\! v_1\,\subseteq L$ and its boundary left
ribbon $\delta$ is proper. The set $A$ of cells between $\delta$ and $\gamma$ is an annulus as in Theorem~\ref{thm: annulus homotopy} so there is a proper
free ribbon-homotopy $\gamma\to\delta$ of support $A$. Since $L\supseteq A$, $\Omega\in\Pee(L)\subseteq\Pee(A)$ and by Corollary
\ref{cor: ribop gauge invar} also $\Ophol_\rho(\Phi)\Omega\in \Pee(A)$. Using topological invariance in the sense of Theorem~\ref{thm: topological invariance}\,(ii), we have
\[
\Ophol_\gamma(\Psi)\Ophol_\rho(\Phi)\Omega=\Ophol_\delta(\Psi)\Ophol_\rho(\Phi)\Omega.
\]
Now using topological invariance, this time in the sense of Theorem~\ref{thm: topological invariance}\,(i), of $\Ophol_\rho(\Phi)\Omega$, we can choose
$\rho$ to be as simple inside $\delta$ as possible. By proper ribbon-homotopy supported in $L$, we can deform $\rho$ in such a way that only the last arrow of the new
$\rho$ lies
$$
\parbox{220pt}{
\begin{picture}(250,120)
\linethickness{0.3pt}
%outside gamma
\color{black}
\polyline(-10,100)(-15,110)(-5,115)(0,105)(-10,100)
\color{blue}\put(-7.5,107.5){\circle*{3}}
\color{black}\polyline(0,105)(8,95)(15,108)(0,105)
\color{red}\put(7.7,102.7){\circle*{3}}
%inside gamma
\color{black}
\polyline(150,65)(165,75)(158,90)(143,90)(135,75)
\color{red}\put(150.5,77.5){\circle*{3}} \put(151,72){$\scriptstyle f_1$}
\color{black}
\polyline(150,65)(138,45)(162,45)(150,65)
\color{blue} \put(150,53){\circle*{3}} \put(151,48){$\scriptstyle v_1$}
\color{black}
\polyline(138,45)(123,55)(135,75)% e^L
\color{mzold}\put(150,34){\circle*{3}}
\color{black}
\polyline(138,45)(138,23)(162,23)
\color{mzold}\put(136.5,60){\circle*{3}}
\color{black}
\polyline(150,65)(165,75)(177,55)(162,45)
\color{mzold}\put(163.5,60){\circle*{3}}
\color{black}
\polyline(162,45)(162,23)(177,13)(192,23)(192,45)(177,55)
\polyline(123,55)(123,33)(138,23)
\color{red}\put(130.5,39){\circle*{3}}\put(177,34){\circle*{3}}
\color{black}
\put(59,86){\line(1,-4){6}}
\put(77,83){\line(1,-4){6}}
\polyline(123,55)(83,59)(83,37)(123,33)(123,55)%tEFut
\polyline(65,62)(83,59)(83,37)(65,62)%DEFD
\polyline(143,90)(113,105)(105,90)%dC'C
\polyline(59,86)(55,101)(70,98)(77,83)%AA'B'B
\polyline(70,98)(98,105)(105,90)%B'C''C
\polyline(98,105)(113,105)%C''C'
\polyline(113,105)(148,107)(143,90)%C'd'd
\polyline(148,107)(163,107)(158,90)%d'c'c
\polyline(158,90)(175,92)(182,77)(165,75)%cIJb
\polyline(163,107)(180,112)(175,92)%c'HI
\polyline(177,55)(197,60)(212,50)(192,45)%aLMz
\polyline(182,77)(207,70)(197,60)%JKL
\polyline(123,33)(108,18)(123,8)(138,23)%uu'v'v
\polyline(65,62)(55,52)(73,27)(83,37)%DD'F'F
\polyline(73,27)(88,12)(108,18)%F'Gu'
\color{red}
\put(77,52.7){\circle*{3}}
\put(65,92){\circle*{3}}
\put(105.3,100){\circle*{3}}
\color{blue}
\put (95,25.4){\circle*{3}}
\put(104.6,72.4){\circle*{3}}
\put(134.7,100.7){\circle*{3}}
\put(169,100.2){\circle*{3}}
\put(185.6,67.4){\circle*{3}}
\color{mzold}
\put(123,20.5){\circle*{3}}
\put(69,44.5){\circle*{3}}
\put(103,46){\circle*{3}}
\put(71,72.5){\circle*{3}}
\put(87.5,94){\circle*{3}}
\put(124,90){\circle*{3}}
\put(153,98.5){\circle*{3}}
\put(170,83.5){\circle*{3}}
\put(194.5,52.5){\circle*{3}}
\linethickness{0.8pt}
\color{orange}
\polyline(135,75)(123,55)(123,33)(138,23)(162,23)(177,13)(192,23)(192,45)(177,55)(165,75)(158,90)(143,90)(135,75)
\put(177,13){\vector(3,2){10}}\put(182,7){$\delta$}
\color{magenta}
\put(137,60){\oval(180,120)}
\put(227,30){\vector(0,1){3}}\put(235,30){$\gamma$}
\color{cian}
\qbezier(-10,100)(-40,85)(-10,70)
\qbezier(-10,70)(0,65)(10,70)
\put(10,70){\line(3,1){49}}% rho arrives inside gamma at (59,86)
\put(-10,100){\line(2,1){10}}
\put(10,70){\vector(2,1){3}}
\put(10,60){$\rho$}
\put(135,75){\vector(3,-2){15}}
\put(59,86){\vector(6,-1){18}}
\put(77,83){\vector(4,1){28}}
\put(105,90){\vector(2,-1){30}}
\color{black}
\put(0,105){\circle*{4}} \put(1,110){$\scriptstyle s_0$}
\put(150,65){\circle*{4}} \put(153,62){$\scriptstyle s_1$}
\put(135,75){\circle*{4}} \put(138,74){$\scriptstyle s$}
\put(205,92){$L$}
\put(240,92){$R$}
\end{picture}}
$$
inside $\delta$ and the last two arrows of this $\rho$ form a straight curve intersecting $\delta$ only at a single site. So we can assume
that $\rho$ and $\delta$ have the following form:
\[
\rho=\rho'\bigl((T_0a)_0^-,a_2^+\bigr),\qquad\delta=\bigl(a_0^-\bigr)\delta'\bigl(\bigl(T_2^{-1}a\bigr)_2^+\bigr),
\]
where $a$ is chosen such that $s(a)$ is the intersection point of $\rho$ and $\delta$, hence $s_1=s(T_2a)$. Notice that we have chosen the base point
of $\delta$ to be the intersection point. The $\rho'$ and $\delta'$ are proper ribbons with holonomies that commute with that of all other pieces
of opcurves. Now we can compute the exchange relation of $H_\delta(\Psi)=\Ophol_\delta(\Psi)$ and $H_\rho(\Phi)=\Ophol_\rho(\Phi)$ as follows:
\begin{align*}
H_\delta(\Psi)H_\rho(\Phi) ={}& H_{(a_0^-)}\bigl(\Psi\p\bigr)H_{\delta'}\bigl(\Psi\pp\bigr)H_{((T_2^{-1}a)_2^+)}\bigl(\Psi\ppp\bigr)\\
&{}\times H_{\rho'}\bigl(\Phi\p\bigr)H_{((T_0a)_0^-)}\bigl(\Phi\pp\bigr)H_{(a_2^+)}\bigl(\Phi\ppp\bigr)\\
\eqby{ophol exch rel2}{}&H_{\rho'}\bigl(\Phi\p\bigr)H_{(a_0^-)}\bigl(\Psi\p\bigr)H_{\delta'}\bigl(\Psi\pp\bigr)\\
&{} \times H_{((T_0a)_0^-)}\bigl(R_2\sweedl\Phi\pp\bigr)H_{((T_2^{-1}a)_2^+)}\bigl(R_1\sweedl\Psi\ppp\bigr)H_{(a_2^+)}\bigl(\Phi\ppp\bigr)\\
={}&H_{\rho'}\bigl(\Phi\p\bigr)H_{((T_0a)_0^-)}\bigl(R_2\sweedl\Phi\pp\bigr)H_{(a_0^-)}\bigl(\Psi\p\bigr)H_{(a_2^+)}\bigl(\Phi\ppp\bigr)\\
&{} \times H_{\delta'}\bigl(\Psi\pp\bigr)H_{((T_2^{-1}a)_2^+)}\bigl(R_1\sweedl\Psi\ppp\bigr)\\
\eqby{ophol exch rel1}{}& H_{\rho'}\bigl(\Phi\p\bigr)H_{((T_0a)_0^-)}\bigl(R_2\sweedl\Phi\pp\bigr)H_{(a_2^+)}\bigl(\Phi\ppp\sweedr R_{1'}\bigr)H_{(a_0^-)}\bigl(\Psi\p\sweedr R_{2'}\bigr)\\
&{} \times H_{\delta'}\bigl(\Psi\pp\bigr)H_{((T_2^{-1}a)_2^+)}\bigl(R_1\sweedl\Psi\ppp\bigr)\\
={}&H_{\rho'}\bigl(\Phi\p\bigr)H_{((T_0a)_0^-)}\bigl(\Phi\pp\bigr)\!\cdot\!\bigl\bra\Phi\ppp,R_2R_{1'}\bigr\ket\!\cdot H_{(a_2^+)}\bigl(\Phi\pppp\bigr)\!\cdot H_\delta(R_1\sweedl\Psi\sweedr R_{2'}).
\end{align*}
The result is reminiscent of (\ref{longi-merid}) but now, since the endpoints of $\rho$ and $\delta$ are different, a~term $\bigl\bra\Phi\ppp,R_2R_{1'}\bigr\ket$
is wedged into the holonomy of $\rho$ which forbids to write the result as $H_\rho$ of something times $H_\delta$ of something. Fortunately, letting this
operator acting on $\Omega$ will eliminate the problem. Applying Theorem~\ref{thm: topological invariance}\,(ii) to the properly contractible $\delta$ within the disk $L$
and using that $\Omega\in\Pee(L)$, we get
\begin{align*}
R_2R_{1'}\ot \Ophol_\delta(R_1\sweedl\Psi\sweedr R_{2'})\Omega&{}=R_2R_{1'}\ot \Omega\cdot\eps_{\D^*}(R_1\sweedl\Psi\sweedr R_{2'})\\
&{}=R_2R_{1'}\ot \Omega\cdot\bigl\bra\Psi,R_{2'}R_1\bigr\ket.
\end{align*}
Substituting the definition of the $R$-matrix, we see that $R_2R_{1'}\ot R_{2'}R_1=(\xi_i\ot x_j)\ot(\xi_j\ot x_i)$, therefore
\[
R_2R_{1'} \bra\Psi,R_{2'}R_1\ket=\smash{\overset{\raisebox{-1pt}{\scriptsize $\leftrightarrow$}}{\Psi}}.
\]
Thus, the wedge term is simply \smash{$\bigl\bra\Phi\ppp,\overset{\raisebox{-1pt}{\scriptsize $\leftrightarrow$}}{\Psi}\bigr\ket$}. If we knew that \smash{$\overset{\raisebox{-1pt}{\scriptsize $\leftrightarrow$}}{\Psi}$} is a central element
of $\D(H)$, then the wedge term would be shiftable along $\rho$ until the endpoint and this would prove (\ref{gamma-rho-Omega}).
So it remains to show that \smash{$\overset{\raisebox{-1pt}{\scriptsize $\leftrightarrow$}}{\Psi}\in\Center\D(H)$} if and only if $\Psi\in\Cocom\D(H)^*$.

Let us introduce the projections $\mathbf{P}$ and its transpose $\mathbf{P}^\trans$ by
\begin{alignat*}{3}
&\mathbf{P}\colon \ \D(H)\to\D(H),\qquad && \mathbf{P}(X):= {i_\D}\p XS_\D\bigl({i_\D}\pp\bigr),& \\
&\mathbf{P}^\trans\colon \ \D(H)^*\to\D(H)^*,\qquad && \mathbf{P}^\trans(\Phi):= {i_\D}\p\sweedl \Phi\sweedr S_\D\bigl({i_\D}\pp\bigr),&
\end{alignat*}
where $i_\D:=\iota\ot i$ is the Haar integral of $\D(H)$. $\mathbf{P}$ projects onto $\Center(\D(H))$ and $\mathbf{P}^\trans$ onto $\Cocom\D(H)^*$.
Using the identity $R\cop_\D(X)=\cop^\op(X)R$ of $R$-matrices, we obtain
\begin{align*}
\mathbf{P}(R_2R_{1'})\ot R_{2'}R_1&={i_\D}\p R_2R_{1'}S_\D\bigl({i_\D}\pp\bigr)\ot R_{2'}R_1\\
&={i_\D}\p R_2R_{1'}S_\D\bigl({i_\D}\pppp\bigr)\ot R_{2'} S_\D\bigl({i_\D}\ppp\bigr) {i_\D}\pp R_1\\
&=R_2{i_\D}\pp R_{1'}S_\D\bigl({i_\D}\pppp\bigr)\ot R_{2'} S_\D\bigl({i_\D}\ppp\bigr) R_1{i_\D}\p\\
&=R_2{i_\D}\pp S_\D\bigl({i_\D}\ppp\bigr) R_{1'}\ot S_\D\bigl({i_\D}\pppp\bigr) R_{2'}R_1 {i_\D}\p\\
&=R_2R_{1'}\ot\mathbf{P}(R_{2'}R_1).
\end{align*}
Therefore, \begin{align*}
\overset{\raisebox{-1pt}{\scriptsize $\leftrightarrow$}}{\Psi}\in\Center\D(H)\quad&\Leftrightarrow\quad \mathbf{P}(\overset{\raisebox{-1pt}{\scriptsize $\leftrightarrow$}}{\Psi})=\overset{\raisebox{-1pt}{\scriptsize $\leftrightarrow$}}{\Psi}\\
&\Leftrightarrow\quad \mathbf{P}(R_2R_{1'}) \bra\Psi,R_{2'}R_1\ket=R_2R_{1'}\bra\Psi,R_{2'}R_1\ket\\
&\Leftrightarrow\quad R_2R_{1'}\bra\mathbf{P}^T(\Psi),R_{2'}R_1\ket=R_2R_{1'}\bra\Psi,R_{2'}R_1\ket\\
&\Leftrightarrow\quad \mathbf{P}^\trans(\Psi)=\Psi.
\end{align*}
So \smash{$\overset{\raisebox{-1pt}{\scriptsize $\leftrightarrow$}}{\Psi}$} is central if and only if $\Psi$ is tracial and the theorem is proven.
\end{proof}

Let $D_r$ be fixed irreducible unitary representations of $\D(H)$ for each isomorphism class $r$ of irreps. Then the set of matrix elements \smash{$\bigl\{D_r^{i,j}\bigr\}$}
is a basis of~$\D(H)^*$. For a proper left ribbon $\rho\colon s(a)\to s(b)$ between disjoint sites, the (\ref{source gauge transf}) and (\ref{target gauge transf}) imply that
the vectors
\[
\Lambda_r^{i,j}:=\Ophol_\rho\bigl(D_r^{i,j}\bigr)\Omega \in\Hil
\]
are irreducible $\D(H)$-multiplets in the sense of the equations
\begin{align*}
&\mathbf{D}_c(X)\Lambda_r^{i,k}=\Lambda_r^{i,k}\cdot\eps_\D(X)\qquad\text{if}\quad c\neq a,b,\\
&\mathbf{D}_b(X)\Lambda_r^{i,k}=\Lambda_r^{i,j}\cdot D_r^{j,k}(X),\\
&\mathbf{D}_a(X)\Lambda_r^{i,k}=D_r^{i,j}(S_\D(X))\cdot\Lambda_r^{j,k}
\end{align*}
provided $\Omega\in\Pee$. Notice that $\eps_\D= D_0$ if $r=0$ designates the trivial representation and $D^\mathsf{T}_r\ci S_\D$ is, up to unitary
equivalence, some $D_{\bar r}$ if $r\mapsto \bar r$ designates charge conjugation.

Let $\gamma$ be a closed proper ribbon winding around $s(b)$ but not $s(a)$
as in Theorem~\ref{thm: gamma measures what rho creates}. Let~$\Psi_r$ be the unique cocommutative element for which \smash{$\overset{\raisebox{-1pt}{\scriptsize $\leftrightarrow$}}{\Psi}_r$}
is the minimal central idempotent $e_r=e_r^{i,i}$, where \smash{$\bigl\{e_r^{i,j}\bigr\}$} is the dual basis of \smash{$\bigl\{D_r^{i,j}\bigr\}$}. Then
\[
\Ophol_\gamma(\Psi_q)\Lambda_r^{i,j}=\Ophol_\rho\bigl(e_q\sweedl D_r^{i,j}\bigr)\Omega=\Lambda^{i,j}_q\cdot\delta_{q,r}.
\]
In this sense, the closed ribbon operators $\Ophol_\gamma(\Psi)$ are able to detect the charge $r$ inside the contraction disk of $\gamma$.
This happens in spite of the fact that for whatever closed $\gamma$ for which the ribbon-homotopy class of $\rho$ contains a $\rho'$ disjoint from $\gamma$
the $\Ophol_\gamma(\Psi)$
detects only the trivial charge on \smash{$\Lambda_r^{i,j}$}. One says that the charge of \smash{$\Lambda_r^{i,j}$} can be localized in semi-infinite strings but not
in a finite region.
This kind of charges are characteristic of massive gauge theories \cite{BF}. This is the first evidence of that the Kitaev model is like
a gauge theory. The more precise analysis should chart the superselection sectors similarly to the Doplicher--Haag--Roberts theory
\cite{DHR-I,DHR-II,DHR-III,DHR-IV} but for string localized charges. In the abelian Kitaev model on a square lattice such a program can be carried out \cite{CNN, Naaijkens}
but for arbitrary f.d.\ C$^*$-Hopf algebra $H$ and for arbitrary infinite (planar) OCPM $\Sigma$ it is not yet in sight.

\subsection*{Acknowledgments}

The author wishes to thank the anonymous referees for their thorough reading of the manuscript and for their precious comments which
improved considerably the final version of this paper.

%\bibliographystyle{sigma.bst}
%\bibliography{article.bib}

\pdfbookmark[1]{References}{ref}
\LastPageEnding

\end{document}